\UseRawInputEncoding
\documentclass[a4paper,10.5pt]{article}
\usepackage{ascmac}
\usepackage{amsmath,amssymb,ascmac,framed,cases,amsthm,cancel,multicol,wrapfig,color}
\usepackage{mathrsfs}
\usepackage[dvipdfmx]{hyperref}
\usepackage{comment}
\usepackage[toc]{appendix}
\usepackage{ifthen}
\usepackage{empheq}
\usepackage{xspace}
\usepackage{bm}
\usepackage{enumitem}

\usepackage[dvips]{graphicx}
\usepackage[margin=1in]{geometry}
\AtBeginDvi{}

\newtheorem{thm}{Theorem}[section]
\newtheorem{lem}[thm]{Lemma}
\newtheorem{cor}[thm]{Corollary}
\newtheorem{prop}[thm]{Proposition}

\theoremstyle{definition}
\newtheorem{dfn}[thm]{Definition}

\theoremstyle{remark}
\newtheorem{rem}[thm]{Remark}

\numberwithin{equation}{section}

\makeatletter
\newcommand{\gl@align}[2]{\lower.6ex\vbox{\baselineskip\z@skip\lineskip\z@
\ialign{$\m@th#1\hfil##\hfil$\crcr#2\crcr=\crcr}}}
\newcommand{\Leqq}{\mathrel{\mathpalette\gl@align<}}
\newcommand{\Geqq}{\mathrel{\mathpalette\gl@align>}}
\makeatother

\newcommand{\zcenp}[1]{#1 concentric strategy}
\newcommand{\zcenc}[1]{#1 concentric strategy}
\newcommand{\zcencr}[1]{reversed #1 concentric strategy}
\newcommand{\sente}{Paul\xspace}
\newcommand{\gote}{Carol\xspace}

\providecommand{\keywords}[1]
{
  \small	
  \textbf{\textit{Keywords---}} #1
}

\providecommand{\MSC}[1]
{
  \small	
  \textbf{\textit{Mathematics Subject Classification 2020---}} #1
}
\begin{document}
\title{A game-theoretic approach to the asymptotic behavior of solutions to an obstacle problem for the mean curvature flow equation}
\author{Kuniyasu Misu}
\date{\empty}
\maketitle
\begin{center}
email: \href{mailto:kuniyasu.misu@math.sci.hokudai.ac.jp}{\nolinkurl{kuniyasu.misu@math.sci.hokudai.ac.jp}}
\end{center}
\begin{abstract}
We consider the asymptotic behavior of solutions to an obstacle problem for the mean curvature flow equation by using a game-theoretic approximation, to which we extend that of Kohn and Serfaty \cite{KS.06}. The paper \cite{KS.06} gives a deterministic two-person zero-sum game whose value functions approximate the solution to the level set mean curvature flow equation without obstacle functions. We prove that moving curves governed by the mean curvature flow converge in time to the boundary of the convex hull of obstacles under some assumptions on the initial curves and obstacles. Convexity of the initial set, as well as smoothness of the initial curves and obstacles, are not needed. In these proofs, we utilize properties of the game trajectories given by very elementary game strategies and consider reachability of each player. Also, when the equation has a driving force term, we present several examples of the asymptotic behavior, including a problem dealt in \cite{GMT.16}.
\end{abstract}

\


\MSC{35B40, 35D40, 91A50}
\section{Introduction}\label{sec1}


\label{sec:Intro}
\paragraph{Obstacle problem for the mean curvature flow equation.}
We consider the following obstacle problem for the mean curvature flow equation:
\begin{equation}
\label{eq:main}
\begin{cases}
V=-\kappa \ \mbox{on $\partial D_t$}, \\
O_{-}\subset D_t, 
\end{cases}
\end{equation}
where $\{D_t\}_{t>0}$ is the unknown family of open sets in $\mathbb{R}^d$, $V$ is the velocity of a point in $\partial D_t$ in the direction of its outward normal vector, $\kappa$ is the mean curvature of $\partial D_t$ at the point and $O_{-}$ is a fixed open set in $\mathbb{R}^d$. Our main goal is to investigate the asymptotic behavior of solutions to \eqref{eq:main} by the level set equation and its game-theoretic approximation. We mainly deal with the case $d=2$ in this manuscript. 

The mean curvature flow equation has been attracting much attention. In the early stages, the smoothness of the initial surface was naturally assumed and the surface evolution was considered as long as singularities do not occur. In paticular when $d=2$, the mean curvature flow equation is often called {\it the curve shortening problem} and the curve evolution was analysed in e.g. \cite{gage,Grayson87}. 

The level set method for surface evolution equations was first rigorously analyzed in \cite{CGG.91,ES.91}. The basic idea of this method is to represent moving surfaces as level sets of auxiliary functions and to rewrite surface evolution equations by level set equations, whose unknown functions are the auxiliary functions. A great advantage is the point that viscosity solutions of level set equations follow the long time behavior of the moving surfaces even after topological change of surfaces. The level set method is applied to various surface evolution equations including the mean curvature flow equation. See also \cite{G.06} in detail.

Recently obstacle problems for the mean curvature flow equation have been considered in \cite{gigachan,koike,Mercier_Novoga}. Obstacle problems are problems that have regions called obstacles which the solutions cannot exceed.

According to the unpublished paper \cite{Mercier} by Mercier, the level set method is still valid for \eqref{eq:main}. The corresponding level set equation to \eqref{eq:main} is the following:
\begin{equation}
\label{eq:application}
\begin{cases}
u_t(x,t)+F(Du(x), D^2 u(x))=0
& \quad \mbox{in} \ \mathbb{R}^d \times (0,\infty), \\
u(x,0)=u_0(x)
& \quad \mbox{in} \ \mathbb{R}^d,\\
\Psi_{-}(x)\le u(x,t)
& \quad \mbox{in} \ \mathbb{R}^d \times (0,\infty),
\end{cases}
\end{equation}
where $\Psi_{-}\in Lip(\mathbb{R}^d)$ is a given obstacle function that satisfies $O_{-}=\{x\in\mathbb{R}^d \mid \Psi_{-}(x)>0\}$. The function $u_0 \in C(\mathbb{R}^d)$ is an initial datum and $F$ is given by 
\begin{equation}
F(Du,D^2u)=-\lvert Du\rvert \mathrm{div} \left( \frac{Du}{\lvert Du\rvert} \right).
\nonumber
\end{equation}
Namely $F$ is the level-set mean curvature flow operator
defined as
\begin{equation*}
F(p,X)=-\mathrm{Tr}\left(\left(I-\frac{p\otimes p}{\lvert p\rvert^2}\right)X\right), \quad p\in \mathbb{R}^d\setminus \{0\}, \ X\in\mathbb{S}^d,
\end{equation*}
where $p\otimes p=(p_i p_j)_{i,j=1}^{d}$ 
for a vector $p=(p_1,\dots,p_d) \in \mathbb{R}^d$
and $\mathbb{S}^d$ is the set of $d\times d$ real symmetric matrices. For a comparison principle to \eqref{eq:application}, see also \cite{koike}.

Throughout this paper we follow the 0 level set of the solution $u$. Together with it, we assume on the initial data $u_0$ as follows:
\begin{equation}
\label{assum:ini}
\mbox{For some $a<0$ and $R>0$, $u_0=a$ in $B^c(0,R)$.}
\end{equation}

\paragraph{Intuitive observation.}

\begin{figure}[h]
\begin{center}
{\unitlength 0.1in%
\begin{picture}(46.9100,34.4200)(-0.8000,-40.4200)%
%
\special{sh 0.300}%
\special{ia 677 2596 164 164 0.0000000 6.2831853}%
\special{pn 8}%
\special{ar 677 2596 164 164 0.0000000 6.2831853}%
%
\special{sh 0.300}%
\special{ia 1989 2588 164 164 0.0000000 6.2831853}%
\special{pn 8}%
\special{ar 1989 2588 164 164 0.0000000 6.2831853}%
%
\special{sh 0.300}%
\special{ia 3136 2596 164 164 0.0000000 6.2831853}%
\special{pn 8}%
\special{ar 3136 2596 164 164 0.0000000 6.2831853}%
%
\special{sh 0.300}%
\special{ia 4447 2588 164 164 0.0000000 6.2831853}%
\special{pn 8}%
\special{ar 4447 2588 164 164 0.0000000 6.2831853}%
%
\special{sh 0.300}%
\special{ia 668 3714 164 164 0.0000000 6.2831853}%
\special{pn 8}%
\special{ar 668 3714 164 164 0.0000000 6.2831853}%
%
\special{sh 0.300}%
\special{ia 1979 3706 164 164 0.0000000 6.2831853}%
\special{pn 8}%
\special{ar 1979 3706 164 164 0.0000000 6.2831853}%
%
\special{sh 0.300}%
\special{ia 3127 3714 164 164 0.0000000 6.2831853}%
\special{pn 8}%
\special{ar 3127 3714 164 164 0.0000000 6.2831853}%
%
\special{sh 0.300}%
\special{ia 4438 3706 164 164 0.0000000 6.2831853}%
\special{pn 8}%
\special{ar 4438 3706 164 164 0.0000000 6.2831853}%
%
\special{pn 8}%
\special{pa 349 2596}%
\special{pa 345 2564}%
\special{pa 341 2531}%
\special{pa 340 2499}%
\special{pa 342 2467}%
\special{pa 348 2436}%
\special{pa 359 2405}%
\special{pa 375 2376}%
\special{pa 395 2348}%
\special{pa 418 2323}%
\special{pa 444 2301}%
\special{pa 473 2284}%
\special{pa 502 2272}%
\special{pa 533 2265}%
\special{pa 564 2263}%
\special{pa 628 2267}%
\special{pa 661 2269}%
\special{pa 695 2268}%
\special{pa 728 2266}%
\special{pa 762 2263}%
\special{pa 794 2262}%
\special{pa 824 2265}%
\special{pa 852 2273}%
\special{pa 876 2288}%
\special{pa 899 2308}%
\special{pa 921 2332}%
\special{pa 942 2358}%
\special{pa 963 2385}%
\special{pa 985 2411}%
\special{pa 1009 2436}%
\special{pa 1036 2458}%
\special{pa 1064 2477}%
\special{pa 1094 2492}%
\special{pa 1125 2505}%
\special{pa 1157 2515}%
\special{pa 1190 2521}%
\special{pa 1223 2525}%
\special{pa 1256 2525}%
\special{pa 1289 2523}%
\special{pa 1321 2517}%
\special{pa 1352 2508}%
\special{pa 1381 2497}%
\special{pa 1410 2483}%
\special{pa 1438 2468}%
\special{pa 1494 2434}%
\special{pa 1548 2400}%
\special{pa 1576 2384}%
\special{pa 1603 2368}%
\special{pa 1631 2352}%
\special{pa 1659 2338}%
\special{pa 1688 2323}%
\special{pa 1746 2297}%
\special{pa 1777 2285}%
\special{pa 1807 2275}%
\special{pa 1839 2265}%
\special{pa 1871 2257}%
\special{pa 1903 2250}%
\special{pa 1936 2244}%
\special{pa 1969 2241}%
\special{pa 2002 2240}%
\special{pa 2035 2241}%
\special{pa 2067 2244}%
\special{pa 2098 2251}%
\special{pa 2129 2260}%
\special{pa 2158 2272}%
\special{pa 2186 2287}%
\special{pa 2213 2306}%
\special{pa 2238 2327}%
\special{pa 2261 2350}%
\special{pa 2282 2376}%
\special{pa 2300 2403}%
\special{pa 2316 2432}%
\special{pa 2329 2462}%
\special{pa 2339 2493}%
\special{pa 2346 2525}%
\special{pa 2351 2557}%
\special{pa 2352 2590}%
\special{pa 2351 2623}%
\special{pa 2347 2656}%
\special{pa 2341 2688}%
\special{pa 2332 2720}%
\special{pa 2320 2752}%
\special{pa 2305 2782}%
\special{pa 2288 2812}%
\special{pa 2268 2839}%
\special{pa 2246 2864}%
\special{pa 2222 2886}%
\special{pa 2195 2904}%
\special{pa 2167 2918}%
\special{pa 2137 2928}%
\special{pa 2106 2934}%
\special{pa 2073 2935}%
\special{pa 2041 2933}%
\special{pa 2008 2928}%
\special{pa 1977 2921}%
\special{pa 1946 2912}%
\special{pa 1916 2901}%
\special{pa 1888 2888}%
\special{pa 1859 2875}%
\special{pa 1831 2860}%
\special{pa 1804 2844}%
\special{pa 1776 2828}%
\special{pa 1748 2811}%
\special{pa 1720 2795}%
\special{pa 1692 2778}%
\special{pa 1663 2761}%
\special{pa 1603 2729}%
\special{pa 1572 2715}%
\special{pa 1540 2702}%
\special{pa 1509 2691}%
\special{pa 1477 2682}%
\special{pa 1445 2676}%
\special{pa 1414 2672}%
\special{pa 1383 2671}%
\special{pa 1352 2674}%
\special{pa 1322 2681}%
\special{pa 1293 2691}%
\special{pa 1264 2705}%
\special{pa 1235 2721}%
\special{pa 1207 2738}%
\special{pa 1123 2792}%
\special{pa 1095 2808}%
\special{pa 1067 2822}%
\special{pa 1037 2834}%
\special{pa 1008 2841}%
\special{pa 977 2845}%
\special{pa 946 2846}%
\special{pa 914 2845}%
\special{pa 882 2843}%
\special{pa 848 2842}%
\special{pa 814 2843}%
\special{pa 780 2845}%
\special{pa 745 2849}%
\special{pa 677 2857}%
\special{pa 644 2860}%
\special{pa 612 2861}%
\special{pa 581 2859}%
\special{pa 552 2855}%
\special{pa 525 2847}%
\special{pa 500 2835}%
\special{pa 477 2818}%
\special{pa 456 2797}%
\special{pa 437 2773}%
\special{pa 420 2746}%
\special{pa 403 2717}%
\special{pa 388 2686}%
\special{pa 373 2653}%
\special{pa 359 2620}%
\special{pa 349 2596}%
\special{fp}%
%
\special{pn 8}%
\special{ar 668 3714 328 328 0.0000000 6.2831853}%
%
\special{pn 8}%
\special{ar 1979 3714 328 328 0.0000000 6.2831853}%
%
\special{pn 8}%
\special{pa 3136 2432}%
\special{pa 4447 2432}%
\special{fp}%
\special{pa 4447 2760}%
\special{pa 3136 2760}%
\special{fp}%
\put(26.3500,-25.5900){\makebox(0,0){$\Longrightarrow$}}%
\put(26.3500,-37.0600){\makebox(0,0){$\Longrightarrow$}}%
%
\special{pn 0}%
\special{sh 0.300}%
\special{pa 996 780}%
\special{pa 996 1764}%
\special{pa 1979 1764}%
\special{pa 1979 780}%
\special{pa 1816 780}%
\special{pa 1816 1600}%
\special{pa 1160 1600}%
\special{pa 1160 780}%
\special{pa 1160 780}%
\special{pa 996 780}%
\special{ip}%
\special{pn 8}%
\special{pa 996 780}%
\special{pa 996 1764}%
\special{pa 1979 1764}%
\special{pa 1979 780}%
\special{pa 1816 780}%
\special{pa 1816 1600}%
\special{pa 1160 1600}%
\special{pa 1160 780}%
\special{pa 996 780}%
\special{pa 996 1764}%
\special{fp}%
%
\special{pn 0}%
\special{sh 0.300}%
\special{pa 3127 764}%
\special{pa 3127 1747}%
\special{pa 4110 1747}%
\special{pa 4110 764}%
\special{pa 3946 764}%
\special{pa 3946 1583}%
\special{pa 3291 1583}%
\special{pa 3291 764}%
\special{pa 3291 764}%
\special{pa 3127 764}%
\special{ip}%
\special{pn 8}%
\special{pa 3127 764}%
\special{pa 3127 1747}%
\special{pa 4110 1747}%
\special{pa 4110 764}%
\special{pa 3946 764}%
\special{pa 3946 1583}%
\special{pa 3291 1583}%
\special{pa 3291 764}%
\special{pa 3127 764}%
\special{pa 3127 1747}%
\special{fp}%
%
\special{pn 8}%
\special{pa 832 600}%
\special{pa 832 1911}%
\special{fp}%
\special{pa 832 1911}%
\special{pa 2143 1911}%
\special{fp}%
\special{pa 2143 600}%
\special{pa 2143 1911}%
\special{fp}%
\special{pa 2143 600}%
\special{pa 1652 600}%
\special{fp}%
\special{pa 1652 600}%
\special{pa 1652 1420}%
\special{fp}%
\special{pa 1652 1420}%
\special{pa 1324 1420}%
\special{fp}%
\special{pa 1324 1420}%
\special{pa 1324 600}%
\special{fp}%
\special{pa 1324 600}%
\special{pa 832 600}%
\special{fp}%
%
\special{pn 8}%
\special{pa 3291 764}%
\special{pa 3946 764}%
\special{fp}%
\put(26.4300,-12.4700){\makebox(0,0){$\Longrightarrow$}}%
\put(26.4300,-10.8400){\makebox(0,0){$t\to\infty$}}%
\put(26.4300,-23.9500){\makebox(0,0){$t\to\infty$}}%
\put(26.4300,-35.4200){\makebox(0,0){$t\to\infty$}}%
%
\special{pn 8}%
\special{pa 1078 1665}%
\special{pa 750 1829}%
\special{fp}%
\put(7.5000,-18.1300){\makebox(0,0)[rt]{obstacle $O_{-}$}}%
%
\special{pn 8}%
\special{pa 2143 1747}%
\special{pa 2307 1911}%
\special{fp}%
\put(23.0700,-19.1100){\makebox(0,0)[lb]{initial curve $\partial D_0$}}%
\end{picture}}%
\end{center}
\caption{Conjecutures on the asymptotic shapes}
\label{fig:conje}
\end{figure}
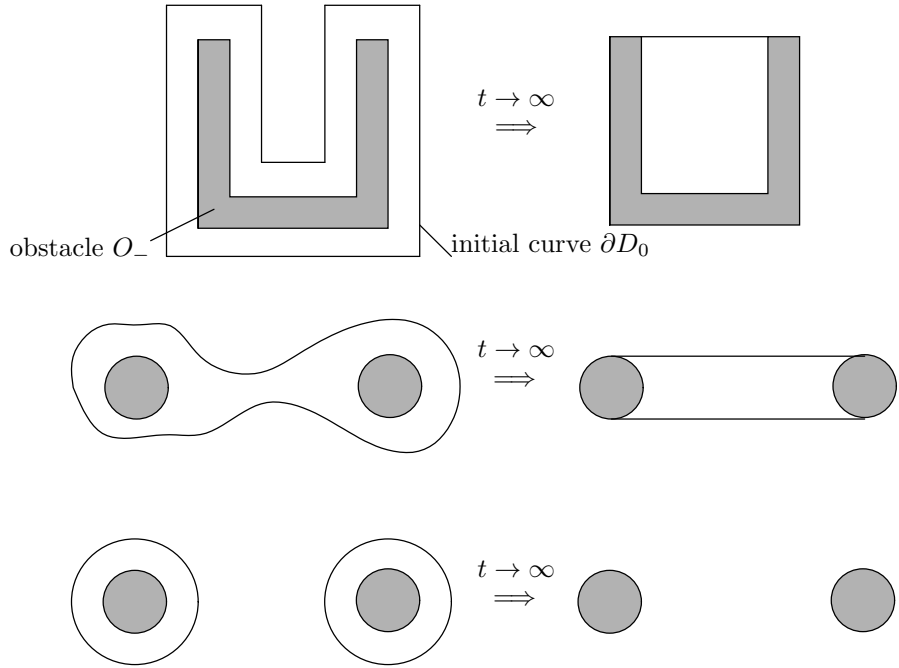

\label{subsec:IO}
For the solution to \eqref{eq:main} with $d=2$ and without obstacles, it is known that $D_t$ becomes convex at some time, the moving curve $\partial D_t$ converges to a single point and then vanishes, provided $\partial D_0$ is a smooth closed curve (\cite{gage,Grayson87}). On the other hand, for our problem \eqref{eq:main}, it is obvious that the solution does not converge to any single point. Also it is not clear whether $D_t$ becomes convex at some time. However it is natural to expect that in many cases $D_t$ converges to the convex hull of $O_{-}$, which we hereafter denote by $Co(O_{-})$, as $t\to\infty$ because of the curve shortening property of the solution and the smoothing effect of the curvature flow as we draw some examples in Figure \ref{fig:conje}. As shown in Figure \ref{fig:conje}, even for the same obstacles, different initial curves may converge to different limits. Thus we shall assume at least that one connected component of $D_0$ contains the whole $O_{-}$ and expect that the asymptotic shape is $Co(O_{-})$ under this assumption. Our main theorem (Theorem \ref{thm:conv_general}) is intended to justify this expectation as much as possible.

\paragraph{Game interpretation.}
\label{subsec:gi}
Our first result is the extension of \cite{KS.06} to problems including \eqref{eq:application}.
First, let us briefly explain the game rule for \eqref{eq:application} with $d=2$ and without obstacle function by following \cite[Section 1.6]{KS.06}.
The game is a deterministic two-person zero-sum game. For convenience, we name the first player \sente and the second player \gote. Let $\epsilon>0$.
Also, let $x_0=x \in \mathbb{R}^2$ be the initial position of this game
and $t>0$ be the terminal time. 
At the $i$-th round of this game,
\sente chooses directions $v_i\in \mathbb{R}^2$ with $\lvert v_i\rvert=1$
and \gote chooses a number $b_i=\pm 1$ after \sente's choice.
Then the game position that we henceforth regard as \sente's position conveniently moves from $x_{i-1}$ to the next place $x_i$
depending on their choice as follows:
\begin{equation}
\label{trajectory2}
x_{i}=x_{i-1}+\sqrt{2} \epsilon b_i v_i
\end{equation} 
After the $N$-th round, where $N \sim t\epsilon^{-2}$,
the game ends and \gote pays the terminal cost $u_0(x_N)$ to \sente.
\sente's goal is maximizing the cost while \gote's goal is minimizing it. 
The value function $u^{\epsilon}(x,t)$ is defined as 
the cost optimized by both the players, that is,
\begin{equation}
u^{\epsilon}(x,t)
=\max_{v_1} \min_{b_1} \dots \max_{v_N} \min_{b_N} u_0(x_N). 
\nonumber
\end{equation}
This value function approximates 
the viscosity solution $u$ of \eqref{eq:application} with $d=2$ and without obstacle function.
In fact the convergence $u^{\epsilon} \to u$ is shown in \cite{KS.06}.

In order to handle \eqref{eq:application} that has the obstacle function $\Psi_{-}$,
we modify the game rule as follows. At each $i$-th round, we suppose that  
\sente has the right to quit the game. If \sente quits the game, the game cost is given by $\Psi_{-}(x_i)$. By doing this modification, the value function $u^{\epsilon}$ is restricted to $\Psi_{-}\le u^{\epsilon}$. Such an interpretation of parameters of PDEs is well understood for first order equations;
see \cite{bardi}. The cost $\Psi_{-}(x_i)$ is called {\it stopping cost} and an optimal control problem with stopping cost is called {\it optimal stopping time problem}. For second order equations, see e.g. \cite{manfredi,CLM17}, which deal with the optimal stopping time problem corresponding to the obstacle problem for the infinity Laplacian equation and p-Laplacian equation respectively.

The value function $u^{\epsilon}(x,t)$ satisfies the following {\it Dynamic Programming Principle}:
\begin{equation*}
u^{\epsilon}(x,t)=\max\{\Psi_{-}(x), \max_{\lvert v\rvert=1}\min_{b=\pm1}u^{\epsilon}(x+\sqrt{2}\epsilon b v, t-\epsilon^2)\}
\end{equation*}
for $t>0$. This is a key equation in the proof of the convergence result.

The paper \cite{KS.06} also mention the game interpretation for higher dimensional case. Based on this, we can generalize our game to the case $d\ge 3$. In the game, \sente chooses $d-1$ orthogonal unit vectors $v_i^{j} (j=1,2,\cdots, d-1)$, \gote chooses $d-1$ values $b_i^j\in \{\pm1\} (j=1,2,\cdots, d-1)$, and the state equation is $x_{i}=x_{i-1}+\sqrt{2} \epsilon \sum_{j=1}^{d-1}b_i^j v_i^j$
instead of \eqref{trajectory2}.

The precise statement of the convergence of the value functions is described by the half relaxed limits of the value functions, which are defined as follows:
\[\overline{u}(x,t):=\varlimsup_{\substack{(y,s) \to (x,t) \\ \epsilon \searrow 0}} u^{\epsilon}(y,s), \ \underline{u}(x,t):=\varliminf_{\substack{(y,s) \to (x,t) \\ \epsilon \searrow 0}} u^{\epsilon}(y,s).\]

As a consequence of Proposition \ref{prop:gconv_gene}, we present the convergence result for \eqref{eq:application} at the moment. We describe a game interpretation and the same type of convergence result for more general PDEs than \eqref{eq:application} in Appendix \ref{app_ginp}.
\begin{prop}
\label{prop:gconv_cor}
The functions $\overline{u}$ and $\underline{u}$ are respectively viscosity sub- and supersolution of \eqref{eq:application}. Moreover $\overline{u}(x,0)=\underline{u}(x,0)=u_0(x)$ for $x\in\mathbb{R}^d$. 
\end{prop}

\paragraph{Asymptotic behavior.}
\begin{figure}[htbp]
 \begin{minipage}{0.5\hsize}
  \begin{center}
{\unitlength 0.1in%
\begin{picture}(22.7400,18.2100)(3.1500,-21.4800)%
%
\special{sh 0.300}%
\special{ia 1331 660 267 267 0.0000000 6.2831853}%
\special{pn 8}%
\special{ar 1331 660 267 267 0.0000000 6.2831853}%
%
\special{pn 0}%
\special{sh 0.300}%
\special{pa 435 1614}%
\special{pa 761 1614}%
\special{pa 761 1939}%
\special{pa 435 1939}%
\special{pa 435 1614}%
\special{ip}%
\special{pn 8}%
\special{pa 435 1614}%
\special{pa 761 1614}%
\special{pa 761 1939}%
\special{pa 435 1939}%
\special{pa 435 1614}%
\special{pa 761 1614}%
\special{fp}%
%
\special{pn 0}%
\special{sh 0.300}%
\special{pa 2182 1614}%
\special{pa 1934 1987}%
\special{pa 2431 1987}%
\special{pa 2431 1987}%
\special{pa 2182 1614}%
\special{ip}%
\special{pn 8}%
\special{pa 2182 1614}%
\special{pa 1934 1987}%
\special{pa 2431 1987}%
\special{pa 2182 1614}%
\special{pa 1934 1987}%
\special{fp}%
%
\special{pn 8}%
\special{pa 360 2072}%
\special{pa 395 2080}%
\special{pa 429 2087}%
\special{pa 464 2094}%
\special{pa 498 2101}%
\special{pa 532 2106}%
\special{pa 565 2110}%
\special{pa 598 2112}%
\special{pa 630 2112}%
\special{pa 662 2110}%
\special{pa 692 2106}%
\special{pa 722 2099}%
\special{pa 750 2090}%
\special{pa 778 2077}%
\special{pa 804 2060}%
\special{pa 829 2041}%
\special{pa 852 2018}%
\special{pa 873 1993}%
\special{pa 893 1967}%
\special{pa 910 1938}%
\special{pa 925 1908}%
\special{pa 938 1876}%
\special{pa 950 1844}%
\special{pa 960 1810}%
\special{pa 969 1775}%
\special{pa 976 1740}%
\special{pa 982 1703}%
\special{pa 988 1667}%
\special{pa 992 1630}%
\special{pa 996 1592}%
\special{pa 1000 1555}%
\special{pa 1003 1517}%
\special{pa 1009 1443}%
\special{pa 1012 1407}%
\special{pa 1016 1371}%
\special{pa 1020 1336}%
\special{pa 1025 1302}%
\special{pa 1031 1269}%
\special{pa 1038 1237}%
\special{pa 1046 1206}%
\special{pa 1055 1177}%
\special{pa 1066 1150}%
\special{pa 1078 1124}%
\special{pa 1092 1100}%
\special{pa 1108 1078}%
\special{pa 1127 1058}%
\special{pa 1147 1040}%
\special{pa 1170 1025}%
\special{pa 1196 1013}%
\special{pa 1224 1003}%
\special{pa 1255 996}%
\special{pa 1288 992}%
\special{pa 1321 990}%
\special{pa 1355 992}%
\special{pa 1387 996}%
\special{pa 1418 1003}%
\special{pa 1446 1013}%
\special{pa 1472 1026}%
\special{pa 1495 1041}%
\special{pa 1515 1059}%
\special{pa 1534 1079}%
\special{pa 1550 1102}%
\special{pa 1564 1126}%
\special{pa 1576 1152}%
\special{pa 1587 1180}%
\special{pa 1596 1210}%
\special{pa 1604 1241}%
\special{pa 1611 1273}%
\special{pa 1617 1307}%
\special{pa 1622 1341}%
\special{pa 1630 1413}%
\special{pa 1633 1449}%
\special{pa 1636 1486}%
\special{pa 1639 1524}%
\special{pa 1643 1561}%
\special{pa 1646 1599}%
\special{pa 1650 1636}%
\special{pa 1655 1673}%
\special{pa 1660 1709}%
\special{pa 1666 1745}%
\special{pa 1673 1781}%
\special{pa 1682 1815}%
\special{pa 1692 1848}%
\special{pa 1703 1880}%
\special{pa 1716 1911}%
\special{pa 1731 1940}%
\special{pa 1748 1968}%
\special{pa 1767 1994}%
\special{pa 1789 2018}%
\special{pa 1812 2040}%
\special{pa 1837 2060}%
\special{pa 1864 2078}%
\special{pa 1893 2094}%
\special{pa 1923 2109}%
\special{pa 1954 2120}%
\special{pa 1987 2130}%
\special{pa 2020 2138}%
\special{pa 2054 2143}%
\special{pa 2088 2147}%
\special{pa 2122 2148}%
\special{pa 2157 2147}%
\special{pa 2191 2144}%
\special{pa 2225 2138}%
\special{pa 2258 2131}%
\special{pa 2290 2121}%
\special{pa 2321 2109}%
\special{pa 2351 2094}%
\special{pa 2379 2077}%
\special{pa 2406 2058}%
\special{pa 2430 2037}%
\special{pa 2453 2014}%
\special{pa 2474 1990}%
\special{pa 2493 1965}%
\special{pa 2511 1940}%
\special{pa 2526 1914}%
\special{pa 2540 1888}%
\special{pa 2552 1861}%
\special{pa 2562 1834}%
\special{pa 2571 1807}%
\special{pa 2583 1751}%
\special{pa 2586 1722}%
\special{pa 2588 1693}%
\special{pa 2589 1664}%
\special{pa 2588 1635}%
\special{pa 2585 1606}%
\special{pa 2582 1576}%
\special{pa 2570 1516}%
\special{pa 2562 1486}%
\special{pa 2542 1426}%
\special{pa 2530 1395}%
\special{pa 2517 1365}%
\special{pa 2503 1334}%
\special{pa 2487 1304}%
\special{pa 2471 1273}%
\special{pa 2454 1243}%
\special{pa 2435 1212}%
\special{pa 2416 1182}%
\special{pa 2374 1122}%
\special{pa 2328 1062}%
\special{pa 2304 1032}%
\special{pa 2279 1003}%
\special{pa 2254 973}%
\special{pa 2227 944}%
\special{pa 2200 916}%
\special{pa 2173 887}%
\special{pa 2115 831}%
\special{pa 2086 804}%
\special{pa 2056 777}%
\special{pa 1994 723}%
\special{pa 1963 698}%
\special{pa 1931 672}%
\special{pa 1899 647}%
\special{pa 1800 575}%
\special{pa 1766 552}%
\special{pa 1733 530}%
\special{pa 1699 508}%
\special{pa 1665 487}%
\special{pa 1597 447}%
\special{pa 1563 428}%
\special{pa 1529 410}%
\special{pa 1495 394}%
\special{pa 1461 379}%
\special{pa 1428 365}%
\special{pa 1395 353}%
\special{pa 1362 343}%
\special{pa 1330 335}%
\special{pa 1299 330}%
\special{pa 1268 327}%
\special{pa 1238 327}%
\special{pa 1208 329}%
\special{pa 1180 335}%
\special{pa 1152 343}%
\special{pa 1126 355}%
\special{pa 1100 371}%
\special{pa 1076 390}%
\special{pa 1052 411}%
\special{pa 1030 435}%
\special{pa 1009 461}%
\special{pa 989 487}%
\special{pa 970 514}%
\special{pa 952 541}%
\special{pa 935 569}%
\special{pa 919 598}%
\special{pa 904 626}%
\special{pa 889 656}%
\special{pa 875 685}%
\special{pa 862 715}%
\special{pa 838 775}%
\special{pa 826 806}%
\special{pa 815 836}%
\special{pa 793 898}%
\special{pa 783 929}%
\special{pa 772 960}%
\special{pa 762 991}%
\special{pa 751 1021}%
\special{pa 741 1052}%
\special{pa 730 1083}%
\special{pa 708 1143}%
\special{pa 684 1203}%
\special{pa 672 1232}%
\special{pa 659 1261}%
\special{pa 645 1290}%
\special{pa 631 1318}%
\special{pa 616 1346}%
\special{pa 600 1373}%
\special{pa 583 1399}%
\special{pa 565 1425}%
\special{pa 546 1451}%
\special{pa 527 1476}%
\special{pa 506 1500}%
\special{pa 460 1546}%
\special{pa 435 1568}%
\special{pa 411 1590}%
\special{pa 388 1612}%
\special{pa 367 1635}%
\special{pa 350 1660}%
\special{pa 337 1686}%
\special{pa 327 1714}%
\special{pa 320 1743}%
\special{pa 316 1773}%
\special{pa 315 1805}%
\special{pa 316 1837}%
\special{pa 319 1870}%
\special{pa 324 1904}%
\special{pa 330 1939}%
\special{pa 337 1974}%
\special{pa 345 2010}%
\special{pa 353 2045}%
\special{pa 360 2072}%
\special{fp}%
%
\special{pn 8}%
\special{pa 1535 1539}%
\special{pa 1640 1433}%
\special{fp}%
\put(15.3500,-15.3900){\makebox(0,0)[rt]{$\partial D_0$}}%
%
\special{pn 8}%
\special{pa 1535 792}%
\special{pa 1640 899}%
\special{fp}%
\put(16.4000,-8.9900){\makebox(0,0)[lt]{$O_{-}$}}%
\end{picture}}%
  \end{center}
  \caption{Example of $D_0$ and $O_{-}$}
  \label{fig:base}
 \end{minipage}
 \begin{minipage}{0.5\hsize}
  \begin{center}
{\unitlength 0.1in%
\begin{picture}(22.5400,18.0600)(3.1600,-21.3300)%
%
\special{sh 0.300}%
\special{ia 1323 657 265 265 0.0000000 6.2831853}%
\special{pn 8}%
\special{ar 1323 657 265 265 0.0000000 6.2831853}%
%
\special{pn 0}%
\special{sh 0.300}%
\special{pa 434 1604}%
\special{pa 757 1604}%
\special{pa 757 1925}%
\special{pa 434 1925}%
\special{pa 434 1604}%
\special{ip}%
\special{pn 8}%
\special{pa 434 1604}%
\special{pa 757 1604}%
\special{pa 757 1925}%
\special{pa 434 1925}%
\special{pa 434 1604}%
\special{pa 757 1604}%
\special{fp}%
%
\special{pn 0}%
\special{sh 0.300}%
\special{pa 2167 1604}%
\special{pa 1921 1973}%
\special{pa 2414 1973}%
\special{pa 2414 1973}%
\special{pa 2167 1604}%
\special{ip}%
\special{pn 8}%
\special{pa 2167 1604}%
\special{pa 1921 1973}%
\special{pa 2414 1973}%
\special{pa 2167 1604}%
\special{pa 1921 1973}%
\special{fp}%
%
\special{pn 8}%
\special{pa 360 2057}%
\special{pa 395 2065}%
\special{pa 430 2072}%
\special{pa 464 2080}%
\special{pa 498 2086}%
\special{pa 532 2091}%
\special{pa 566 2095}%
\special{pa 598 2097}%
\special{pa 630 2097}%
\special{pa 662 2095}%
\special{pa 692 2091}%
\special{pa 722 2084}%
\special{pa 750 2074}%
\special{pa 777 2061}%
\special{pa 803 2044}%
\special{pa 828 2024}%
\special{pa 851 2001}%
\special{pa 872 1976}%
\special{pa 891 1949}%
\special{pa 908 1920}%
\special{pa 923 1890}%
\special{pa 936 1858}%
\special{pa 947 1825}%
\special{pa 957 1791}%
\special{pa 965 1756}%
\special{pa 972 1720}%
\special{pa 978 1684}%
\special{pa 983 1647}%
\special{pa 991 1573}%
\special{pa 995 1535}%
\special{pa 1004 1424}%
\special{pa 1007 1388}%
\special{pa 1011 1352}%
\special{pa 1016 1317}%
\special{pa 1021 1283}%
\special{pa 1027 1251}%
\special{pa 1034 1219}%
\special{pa 1042 1189}%
\special{pa 1052 1160}%
\special{pa 1063 1133}%
\special{pa 1076 1108}%
\special{pa 1091 1085}%
\special{pa 1108 1064}%
\special{pa 1127 1045}%
\special{pa 1148 1028}%
\special{pa 1172 1014}%
\special{pa 1199 1002}%
\special{pa 1228 994}%
\special{pa 1260 988}%
\special{pa 1293 985}%
\special{pa 1327 985}%
\special{pa 1360 987}%
\special{pa 1392 993}%
\special{pa 1422 1002}%
\special{pa 1449 1013}%
\special{pa 1473 1027}%
\special{pa 1495 1043}%
\special{pa 1514 1063}%
\special{pa 1531 1084}%
\special{pa 1546 1107}%
\special{pa 1560 1133}%
\special{pa 1571 1160}%
\special{pa 1581 1189}%
\special{pa 1590 1219}%
\special{pa 1597 1251}%
\special{pa 1603 1284}%
\special{pa 1608 1318}%
\special{pa 1613 1353}%
\special{pa 1617 1389}%
\special{pa 1620 1425}%
\special{pa 1624 1462}%
\special{pa 1627 1499}%
\special{pa 1630 1537}%
\special{pa 1633 1574}%
\special{pa 1637 1611}%
\special{pa 1641 1649}%
\special{pa 1647 1685}%
\special{pa 1652 1721}%
\special{pa 1659 1757}%
\special{pa 1667 1791}%
\special{pa 1677 1825}%
\special{pa 1688 1857}%
\special{pa 1700 1889}%
\special{pa 1715 1918}%
\special{pa 1731 1946}%
\special{pa 1750 1973}%
\special{pa 1771 1997}%
\special{pa 1793 2020}%
\special{pa 1818 2041}%
\special{pa 1845 2059}%
\special{pa 1873 2076}%
\special{pa 1903 2090}%
\special{pa 1934 2103}%
\special{pa 1966 2113}%
\special{pa 1999 2121}%
\special{pa 2033 2127}%
\special{pa 2067 2131}%
\special{pa 2102 2132}%
\special{pa 2136 2132}%
\special{pa 2170 2129}%
\special{pa 2204 2124}%
\special{pa 2237 2117}%
\special{pa 2269 2107}%
\special{pa 2301 2095}%
\special{pa 2331 2081}%
\special{pa 2359 2064}%
\special{pa 2386 2045}%
\special{pa 2411 2024}%
\special{pa 2434 2002}%
\special{pa 2455 1978}%
\special{pa 2474 1953}%
\special{pa 2492 1928}%
\special{pa 2507 1902}%
\special{pa 2521 1875}%
\special{pa 2533 1849}%
\special{pa 2543 1822}%
\special{pa 2552 1794}%
\special{pa 2559 1766}%
\special{pa 2564 1738}%
\special{pa 2568 1710}%
\special{pa 2570 1681}%
\special{pa 2570 1652}%
\special{pa 2569 1623}%
\special{pa 2567 1594}%
\special{pa 2563 1564}%
\special{pa 2557 1534}%
\special{pa 2550 1504}%
\special{pa 2542 1474}%
\special{pa 2533 1444}%
\special{pa 2522 1413}%
\special{pa 2510 1383}%
\special{pa 2497 1353}%
\special{pa 2483 1322}%
\special{pa 2468 1292}%
\special{pa 2451 1261}%
\special{pa 2433 1231}%
\special{pa 2415 1200}%
\special{pa 2395 1170}%
\special{pa 2353 1110}%
\special{pa 2307 1050}%
\special{pa 2283 1020}%
\special{pa 2258 991}%
\special{pa 2232 962}%
\special{pa 2178 904}%
\special{pa 2122 848}%
\special{pa 2093 820}%
\special{pa 2063 792}%
\special{pa 2033 765}%
\special{pa 1971 713}%
\special{pa 1907 661}%
\special{pa 1875 637}%
\special{pa 1842 612}%
\special{pa 1809 588}%
\special{pa 1776 565}%
\special{pa 1742 542}%
\special{pa 1708 520}%
\special{pa 1675 499}%
\special{pa 1641 478}%
\special{pa 1606 458}%
\special{pa 1572 438}%
\special{pa 1504 402}%
\special{pa 1470 386}%
\special{pa 1437 372}%
\special{pa 1404 359}%
\special{pa 1371 348}%
\special{pa 1338 339}%
\special{pa 1307 332}%
\special{pa 1275 328}%
\special{pa 1245 326}%
\special{pa 1215 327}%
\special{pa 1186 331}%
\special{pa 1158 339}%
\special{pa 1131 349}%
\special{pa 1105 363}%
\special{pa 1080 380}%
\special{pa 1056 401}%
\special{pa 1034 424}%
\special{pa 1012 449}%
\special{pa 992 475}%
\special{pa 972 502}%
\special{pa 954 529}%
\special{pa 937 557}%
\special{pa 921 585}%
\special{pa 905 614}%
\special{pa 890 643}%
\special{pa 876 672}%
\special{pa 850 732}%
\special{pa 826 792}%
\special{pa 814 823}%
\special{pa 803 854}%
\special{pa 792 884}%
\special{pa 782 915}%
\special{pa 771 946}%
\special{pa 761 977}%
\special{pa 750 1007}%
\special{pa 739 1038}%
\special{pa 729 1069}%
\special{pa 707 1129}%
\special{pa 683 1189}%
\special{pa 671 1218}%
\special{pa 658 1247}%
\special{pa 644 1276}%
\special{pa 630 1304}%
\special{pa 615 1332}%
\special{pa 599 1359}%
\special{pa 582 1386}%
\special{pa 565 1412}%
\special{pa 527 1462}%
\special{pa 506 1487}%
\special{pa 484 1510}%
\special{pa 461 1533}%
\special{pa 436 1555}%
\special{pa 412 1578}%
\special{pa 389 1600}%
\special{pa 369 1623}%
\special{pa 351 1648}%
\special{pa 338 1674}%
\special{pa 328 1702}%
\special{pa 321 1731}%
\special{pa 317 1762}%
\special{pa 316 1793}%
\special{pa 317 1826}%
\special{pa 320 1859}%
\special{pa 325 1893}%
\special{pa 331 1927}%
\special{pa 338 1962}%
\special{pa 346 1998}%
\special{pa 354 2033}%
\special{pa 360 2057}%
\special{fp}%
%
\special{pn 8}%
\special{pn 8}%
\special{pa 676 1741}%
\special{pa 676 1749}%
\special{fp}%
\special{pa 668 1781}%
\special{pa 665 1787}%
\special{fp}%
\special{pa 648 1813}%
\special{pa 642 1818}%
\special{fp}%
\special{pa 613 1837}%
\special{pa 607 1839}%
\special{fp}%
\special{pa 573 1845}%
\special{pa 565 1845}%
\special{fp}%
\special{pa 532 1837}%
\special{pa 525 1834}%
\special{fp}%
\special{pa 498 1814}%
\special{pa 493 1808}%
\special{fp}%
\special{pa 476 1781}%
\special{pa 474 1774}%
\special{fp}%
\special{pa 468 1740}%
\special{pa 468 1731}%
\special{fp}%
\special{pa 476 1700}%
\special{pa 479 1694}%
\special{fp}%
\special{pa 498 1668}%
\special{pa 503 1663}%
\special{fp}%
\special{pa 531 1646}%
\special{pa 536 1643}%
\special{fp}%
\special{pa 571 1637}%
\special{pa 579 1637}%
\special{fp}%
\special{pa 612 1645}%
\special{pa 619 1648}%
\special{fp}%
\special{pa 646 1668}%
\special{pa 651 1673}%
\special{fp}%
\special{pa 668 1700}%
\special{pa 670 1707}%
\special{fp}%
\special{pa 676 1741}%
\special{pa 676 1741}%
\special{fp}%
%
\special{pn 8}%
\special{pn 8}%
\special{pa 1312 682}%
\special{pa 1312 690}%
\special{fp}%
\special{pa 1303 724}%
\special{pa 1300 730}%
\special{fp}%
\special{pa 1281 757}%
\special{pa 1275 762}%
\special{fp}%
\special{pa 1247 780}%
\special{pa 1241 782}%
\special{fp}%
\special{pa 1205 788}%
\special{pa 1197 788}%
\special{fp}%
\special{pa 1164 779}%
\special{pa 1157 776}%
\special{fp}%
\special{pa 1131 756}%
\special{pa 1126 752}%
\special{fp}%
\special{pa 1109 724}%
\special{pa 1106 717}%
\special{fp}%
\special{pa 1100 684}%
\special{pa 1100 675}%
\special{fp}%
\special{pa 1108 642}%
\special{pa 1111 635}%
\special{fp}%
\special{pa 1129 609}%
\special{pa 1134 604}%
\special{fp}%
\special{pa 1161 586}%
\special{pa 1168 583}%
\special{fp}%
\special{pa 1204 576}%
\special{pa 1212 576}%
\special{fp}%
\special{pa 1247 584}%
\special{pa 1253 587}%
\special{fp}%
\special{pa 1280 606}%
\special{pa 1285 611}%
\special{fp}%
\special{pa 1303 640}%
\special{pa 1306 646}%
\special{fp}%
\special{pa 1312 682}%
\special{pa 1312 682}%
\special{fp}%
\put(8.8900,-12.1200){\makebox(0,0)[rb]{$x$}}%
%
\special{pn 8}%
\special{pa 572 1741}%
\special{pa 1206 682}%
\special{dt 0.045}%
%
\special{pn 8}%
\special{pa 889 1317}%
\special{pa 995 1142}%
\special{fp}%
\special{sh 1}%
\special{pa 995 1142}%
\special{pa 943 1189}%
\special{pa 967 1188}%
\special{pa 978 1209}%
\special{pa 995 1142}%
\special{fp}%
\put(9.4100,-12.4400){\makebox(0,0)[lt]{$v$}}%
%
\special{pn 4}%
\special{sh 1}%
\special{ar 889 1212 8 8 0 6.2831853}%
\special{sh 1}%
\special{ar 889 1212 8 8 0 6.2831853}%
\end{picture}}%
  \end{center}
  \caption{Strategy}
  \label{fig:0525}
 \end{minipage}
\end{figure}
We study the asymptotic behavior of solutions to \eqref{eq:main}. To explain an outline of the proof of the main theorem (Theorem \ref{thm:conv_general}), at the moment, we identify $u^{\epsilon}$ with $u$ and consider a specific figure (Figure \ref{fig:base}). Since we consider 0 level set of solutions to \eqref{eq:application}, our concern is whether the game cost is positive or negative. Thus, from \sente's point of view, the victory condition is that the game cost becomes positive. Namely, from \gote's point of view, the victory condition is that the game cost becomes negative. There is no need to give optimal strategies. Hereafter, even if a strategy taken by the players is not optimal, we often use present tense such as "\sente takes some strategy when he is in some domain". To show that the asymptotic shape is $Co(O_{-})$, we have to prove that \sente wins if he starts from $Co(O_{-})$ and \gote wins if \sente starts from $\overline{Co(O_{-})}^c$. (We avoid the argument on the boundary of $Co(O_{-})$.) Furthermore, by the rule of the game explained above, we see that the victory condition of \sente is whether he reaches $O_{-}$ at some round or $D_0$ at the final round. 

The easiest situation for \sente is that the initial game position $x\in O_{-}$. In this case, it suffices for \sente to quit the game at the first round and gain the stopping cost $\Psi_{-}(x)>0$. If we take $x$ as shown in Figure \ref{fig:0525}, a strategy for \sente to win is the following: He keeps taking $v$ parallel to the dotted line segment as in Figure \ref{fig:0525} until he reaches the domain inside the dotted circle. Once he gets there, he quits the game and gains the positive stopping cost. Even if he does not get there, he can gain the positive terminal cost at the final round of the game because the dotted line segment in Figure \ref{fig:0525} is contained in $D_0$.

\begin{figure}[h]
\begin{center}
\input{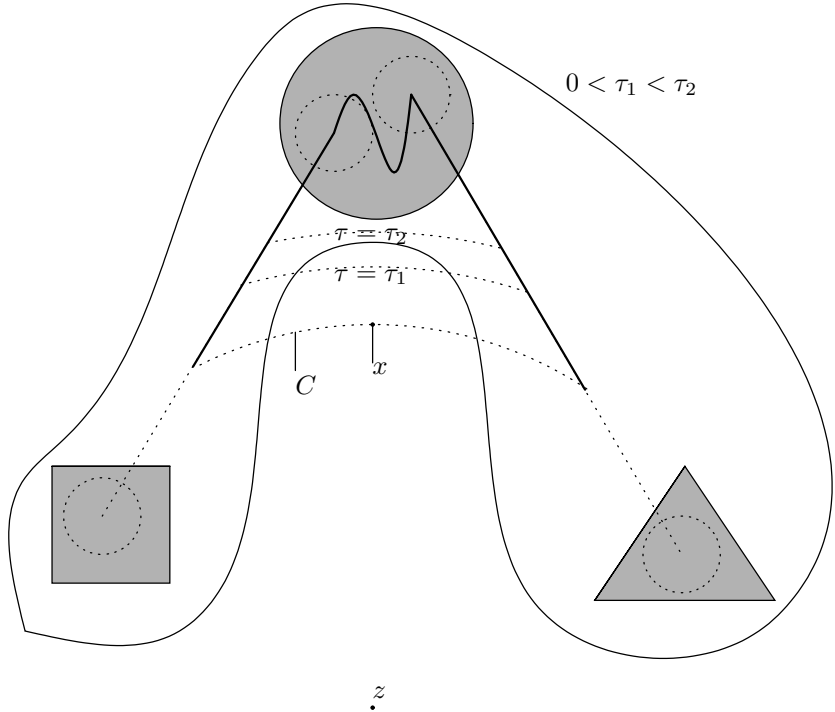}
\end{center}
\caption{Strategy for $x\in Co(O_{-})$}
\label{fig:st0611}
\end{figure}

For the other $x\in Co(O_{-})$, we consider a strategy for \sente to reach the domain from where we above overview that he could win if he started. To construct it, we prepare a type of strategies of the game called {\it \zcenp{}}, which is also introduced in \cite{KS.06} and \cite[Lemma 2.5 2.6]{liu}. See Definition \ref{dfn:const0}. If \sente takes a\zcenp{}, he can choose his favorable point $z\in\mathbb{R}^2$ and can control the distance from $z$ to game positions regardless of \gote's choices as follows:
\begin{equation}
\nonumber
\lvert x_n-z\rvert=\sqrt{\lvert x_0-z\rvert^2+2n\epsilon^2}.
\end{equation} 
In paticular $\lvert x_n-z\rvert$ is monotonically increasing with respect to $n$ and, denoting the game time $n\epsilon^2$ by $\tau$, it goes to infinity as $\tau\to\infty$. Figure \ref{fig:st0611} shows an example of $x\in Co(O_{-})$ and an appropriate\zcenp{}. In Figure \ref{fig:st0611} the center of the arc $C$ is $z$. \sente's strategy is to choose this $z$ and keep taking the concentric strategy until he reaches a neighborhood of the bold curve in $D_0$. Since the domain enclosed by the arc $C$ and the bold curve is bounded, he indeed reaches a neighborhood of the bold curve. Therefore he wins if he starts at this initial position $x$. 

In the main theorem we state a condition on $D_0$ and $O_{-}$ that we are able to apply above technique. To indicate above bounded domain, we construct an appropriate Jordan closed curve in the proof and Appendix \ref{sec:curve} and then use the Jordan curve theorem.

One also define a\zcenc{} of \gote (Definition \ref{dfn:const0}) that has similar effect to that of \sente.  Namely if \gote chooses a point $z$ and takes the\zcenc{}, she can force the distance $\lvert x_n-z\rvert$ to be monotonically increasing with respect to $n$ and go to infinity as $\tau\to\infty$. For $x\in\overline{Co(O_{-})}^c$ we can take an open ball $B$ such that $\overline{Co(O_{-})}\subset B$ and $x\in B^c$ by the hyperplane separation theorem and the boundedness of $\overline{Co(O_{-})}$. If \gote chooses $z$ and takes the\zcenc{}, she wins for sufficiently large $\tau$ owe to the boundedness of $D_0$. If we assume a kind of strict convexity on the obstacle $O_{-}$, we can take above open ball $B$ for $x\in\overline{Co(O_{-})}^c$ whose radius does not depend on $x$. This means that the moving surface sticks to the obstacle in finite time (Theorem \ref{thm:stick}).

\paragraph{Literature.}
We give some other related works on the asymptotic behavior of solutions to obstacle problems for the mean curvature flow equation. Spadaro considers \eqref{eq:main} to characterize the mean-convex hull set in his unpublished paper \cite{Spadaro}. He considers \eqref{eq:main} by a variational discrete scheme, which is different from our approach, but is guaranteed to approximate the viscosity solution to \eqref{eq:application} by \cite{Mercier}. According to \cite{Spadaro}, the part of the limit hypersurface that does not touch the obstacle is a minimal surface. (This result enhances the plausibility of our expectation.) Compared to our result, his result works in higher dimensional case $d\le 7$, while it needs to assume at least that the initial set $D_0$ is convex when $d=2$. For $d\ge 8$, \cite[Proposition 4.2]{Spadaro20} implies that the limit hypersurface may have non empty boundary.
\cite{Mercier_Novoga} proves the convergence of moving surfaces in a situation that
both initial surface and obstacles are given as periodic graphs. For problems with driving force, \cite{gigachan} proves the solution $u(x,t)$ to the problem \eqref{eq:obstacle} with $f=0$ (We will introduce it later.) converges as $t\to \infty$ to the stationary solution. They also give the result concerning to the shape of the stationary solution. However it is limited to the case where the initial data and the obstacle function are radially symmetric.
 
Concerning to an approach other than the level set method, Takasao \cite{Takasao21} considers an obstacle problem for the mean curvature flow equation in the sense of Brakke's mean curvature flow (\cite{Brakke}). He proves the global existence of the weak solution by using the Allen-Cahn equation with forcing term.

While \cite{koike,Takasao21} and this paper consider given obstacle problems, they arise from many different situations. In \cite{GMT.16} an obstacle problem naturally appears in the motion of the top and the bottom of the solution of birth and spread type equations though the equations have no obstacle functions. \cite{LiuYamada} shows that large exponent limit of power mean curvature flow equation is formulated by an obstacle problem involving 1-Laplacian.

\paragraph{Organization.}
This paper is organized as follows. Section \ref{sec:pr2} contains definitions, notations and lemmas that are needed to prove our results. In Section \ref{sec:main} we prove the theorems on the asymptotic shape of the solution to \eqref{eq:main}. Section \ref{sec:dri} is devoted to compute several examples of asymptotic shapes of solutions to problems with driving force. The convergence of the value functions of the game and some arguments to complement the proof of the main theorem are presented in appendices.

\section{Preliminary result}
\label{sec:pr2}
\subsection{Definitions and notations}
In this subsection we introduce the notion of viscosity solution to the following obstacle problem, which is the most general form in this manuscript. Moreover we remark some known results.
\begin{equation}
\label{eq:obstacle}
\begin{cases}
u_t(x,t)-\nu\lvert Du(x,t)\rvert+F(Du(x,t), D^2 u(x,t))=f(x)
& \quad \mbox{in} \ \mathbb{R}^d \times (0,\infty), \\
u(x,0)=u_0(x)
& \quad \mbox{in} \ \mathbb{R}^d,\\
\Psi_{-}(x)\le u(x,t)\le \Psi_{+}(x)
& \quad \mbox{in} \ \mathbb{R}^d \times (0,\infty),
\end{cases}
\end{equation}
where $\Psi_{+}, \Psi_{-}\in Lip(\mathbb{R}^d)$ are obstacle functions which satisfy $\Psi_{-}\le \Psi_{+}$ in $\mathbb{R}^d$. A real number $\nu$ is a constant and $f$ is a locally bounded function. This equation without obstacles is a birth and spread type equation introduced in \cite{GMT.16}. Though the source term $f$ is not considered in the proof of the main theorem, we take it into consideration in Appendix \ref{app_ginp} to prepare for the forthcoming paper ''Asymptotic shape of solutions to the mean curvature flow equation with discontinuous source terms'' by Hamamuki and the author. We do so all the more because it is natural extension in light of optimal control theory. 

We denote the upper and lower semicontinous envelope of $u$ by $u^{\ast}$ and $u_{\ast}$ respectively.
\begin{dfn}[Viscosity solution]
\label{def:vis}
\begin{enumerate}
\item A function $u$ is a viscosity subsolution of \eqref{eq:obstacle} if it satisfies the following conditions.
\begin{enumerate}
\item $\Psi_{-}(x)\le u^{\ast}(x,t)\le \Psi_{+}(x)$ for all $(x,t)\in\mathbb{R}^d\times(0,\infty)$.
\item $u^{\ast}(x,0)\le u_0(x)$ for all $x\in\mathbb{R}^d$.
\item Whenever $\phi(x,t)$ is smooth, $u^{\ast}-\phi$ has a local maximum at $(x_0,t_0)\in\mathbb{R}^d\times(0,\infty)$ and $u^{\ast}(x_0,t_0)-\Psi_{-}(x_0)>0$, we have
\begin{equation}
\phi_t(x_0,t_0)-\nu\lvert D\phi(x_0,t_0)\rvert+F_{\ast}(D\phi(x_0,t_0),D^2 \phi(x_0,t_0))\le f^{\ast}(x_0). \nonumber
\end{equation}
\end{enumerate}
\item A function $u$ is a viscosity supersolution of \eqref{eq:obstacle} if it satisfies the following conditions.
\begin{enumerate}
\item $\Psi_{-}(x)\le u_{\ast}(x,t)\le \Psi_{+}(x)$ for all $(x,t)\in\mathbb{R}^d\times(0,\infty)$.\label{def:vissol_ob_1}
\item $u_{\ast}(x,0)\ge u_0(x)$ for all $x\in\mathbb{R}^d$.
\item Whenever $\phi(x,t)$ is smooth, $u_{\ast}-\phi$ has a local minimum at $(x_0,t_0)\in\mathbb{R}^d\times(0,\infty)$ and $u_{\ast}(x_0,t_0)<\Psi_{+}(x_0)$, we have
\begin{equation}
\phi_t(x_0,t_0)-\nu\lvert D\phi(x_0,t_0)\rvert+F^{\ast}(D\phi(x_0,t_0),D^2 \phi(x_0,t_0))\ge f_{\ast}(x_0). \nonumber
\end{equation}
\end{enumerate}
\item A function $u$ is a viscosity solution of \eqref{eq:obstacle} if it is a viscosity subsolution and a viscosity supersolution of \eqref{eq:obstacle}.
\end{enumerate}
\end{dfn}

We now give the definition of solutions of the following surface evolution equation:
\begin{equation}
\label{eq:general_ev}
\begin{cases}
V=-\kappa+\nu, \ \mbox{on $\partial D_t$}, \\
O_{-}\subset D_t\subset O_{+}, 
\end{cases}
\end{equation}
where $\{D_t\}_{t>0}$ is the unknown family of open sets in $\mathbb{R}^d$. Furthermore $O_{-}$ and $O_{+}$ are fixed open sets in $\mathbb{R}^d$.
We also introduce the closed version of \eqref{eq:general_ev}:
\begin{equation}
\label{eq:closed}
\begin{cases}
V=-\kappa+\nu, \ \mbox{on $\partial E_t$} \\
C_{-}\subset E_t\subset C_{+},  
\end{cases}
\end{equation}
where $\{E_t\}_{t>0}$ is the unknown family of closed sets in $\mathbb{R}^d$. $C_{-}$ and $C_{+}$ are fixed closed sets in $\mathbb{R}^d$. The PDE \eqref{eq:obstacle} with $f=0$ is the level set equation for these surface evolution equations. Since we only consider bounded initial surfaces in this manuscript, we employ the following class of solutions:
\begin{align*}
&K_{a}(\mathbb{R}^d\times [0,\infty)):= \\
&\{u\in C(\mathbb{R}^d\times [0,\infty)) \mid \forall T>0 \ \exists R>0 \ \mbox{s.t. $u=a$ in} \ B_R^c(0)\times [0,T]\}.
\end{align*}
\begin{dfn}
\label{def:evolution}
\begin{enumerate}
\item Let $D_0$, $O_{-}$ and $O_{+}$ be open sets in $\mathbb{R}^d$. A family of open sets $\{D_t\}_{t\ge 0}$ is called an {\it open evolution} of \eqref{eq:general_ev} with $D_0$, $O_{-}$ and $O_{+}$ if there exist $\Psi_{-},\Psi_{+}\in Lip(\mathbb{R}^d)$, $u_0\in C(\mathbb{R}^d)$ and a solution $u\in K_{a}(\mathbb{R}^d\times [0,\infty))$ of \eqref{eq:obstacle} with $\Psi_{-}$, $\Psi_{+}$, $u_0$ and $f=0$ such that $O_{-}=\{x\in\mathbb{R}^d \mid \Psi_{-}(x)>0\}$, $O_{+}=\{x\in\mathbb{R}^d \mid \Psi_{+}(x)>0\}$ and $D_t=\{x\in \mathbb{R}^d \mid u(x,t)>0\}$ for $t\ge 0$.

\item Let $E_0$, $C_{-}$ and $C_{+}$ be closed sets in $\mathbb{R}^d$. A family of closed sets $\{D_t\}_{t\ge 0}$ is called an {\it closed evolution} of \eqref{eq:closed} with $E_0$, $C_{-}$ and $C_{+}$ if there exist $\Psi_{-},\Psi_{+}\in Lip(\mathbb{R}^d)$, $u_0\in C(\mathbb{R}^d)$ and a solution $u\in K_{a}(\mathbb{R}^d\times [0,\infty))$ of \eqref{eq:obstacle} with $\Psi_{-}$, $\Psi_{+}$, $u_0$ and $f=0$ such that $C_{-}=\{x\in\mathbb{R}^d \mid \Psi_{-}(x)\ge 0\}$, $C_{+}=\{x\in\mathbb{R}^d \mid \Psi_{+}(x)\ge 0\}$ and $E_t=\{x\in \mathbb{R}^d \mid u(x,t)\ge 0\}$ for $t\ge 0$. 
\end{enumerate}
\end{dfn}
\begin{rem}
The open evolutions and the closed evolutions uniquely exist (\cite{Mercier}).
\end{rem}
\begin{rem}
Our main equation in this manuscript is \eqref{eq:main}, which has an obstacle on one side. We interpret problems that have only $O_{-}$ as \eqref{eq:general_ev} with $O_{+}=\mathbb{R}^d$ and problems that have only $O_{+}$ as \eqref{eq:general_ev} with $O_{-}=\emptyset$.
\end{rem}
Whenever we consider \eqref{eq:general_ev} in this manuscript, we simultaniouly consider the solution $\{E_t\}_{t\ge 0}$ to \eqref{eq:closed} with $E_0=\overline{D_0}$, $C_{-}=\overline{O_{-}}$ and $C_{+}=\overline{O_{+}}$. Throughout the manuscript, we denote an open evolution by $\{D_t\}_{t\ge 0}$ and a closed evolution by $\{E_t\}_{t\ge 0}$. As explained above, we assume that $D_0$ is bounded.

\paragraph{Notations.}
For a point $z\in\mathbb{R}^d$, we denote the set $\{x\in\mathbb{R}^d \mid \lvert x-z\rvert<r\}$ by $B_r(z)$ or sometimes $B(z,r)$. For a set $S\subset \mathbb{R}^d$, we denote the set $\{x\in\mathbb{R}^d \mid dist(x,S)<r\}$ by $B_r(S)$. We denote by $S^{d-1}$ the set of unit vectors in $\mathbb{R}^d$. The line segment with end points $x$ and $y$ will be denoted by $l_{x,y}$. When two lines $l_1$ and $l_2$ are parallel, we will write $l_1 \parallel l_2$. For a set $A$, we denote the convex hull of $A$ by $Co(A)$. For a family of sets $\{D_t\}_{t\ge 0}$, we define
\[\varlimsup_{t\to\infty}D_t:=\bigcap_{\tau>0}\bigcup_{t>\tau}D_t, \ \varliminf_{t\to\infty}D_t:=\bigcup_{\tau>0}\bigcap_{t>\tau}D_t.\]
If $\varlimsup_{t\to\infty}D_t=\varliminf_{t\to\infty}D_t$, we will write
\[\lim_{t\to\infty}D_t:=\varlimsup_{t\to\infty}D_t=\varliminf_{t\to\infty}D_t.\]
\subsection{Basic strategy of the game}
\label{subsec:bsotg}
We prepare special strategies of both players that we explained in Section \ref{sec:Intro}. 
\begin{dfn}[Concentric strategy]
\label{dfn:const0}
Let $\epsilon>0$ and $z\in\mathbb{R}^d$. Let $x\in\mathbb{R}^d$ be the current position of the game. 
\begin{enumerate}
\item A set of $d-1$ orthogonal unit vectors $v^{j}\in S^{d-1} (j=1,2,\cdots d-1)$ chosen by \sente is called a {\it \zcenp{$z$}} (by \sente) if 
$\langle v^{j}, x-z\rangle=0$ for all $j$, where $\langle \cdot, \cdot\rangle$ stands for the Euclidean inner product.
\item Let $\{v^1, v^2, \cdots, v^{d-1}\}$ be a choice by \sente in the same round. A choice $(b^1,b^2,\cdots,b^{d-1})\in \{\pm1\}^{d-1}$ by \gote is called a {\it \zcenc{$z$}} (by \gote) if $\langle b^j v^j, x-z\rangle \ge0$ for all $j$.
\end{enumerate}
\end{dfn}
One can easily understand the behaviors of trajectories given when one player takes above strategies. Let $d_n=\lvert x_n-z\rvert$ for fixed $z\in\mathbb{R}^d$. If \sente takes a \zcenp{$z$} through the game, then we see that the sequence $\{d_n\}$ satisfies 
\begin{equation}
R_{n+1}=\sqrt{R_n^2+2(d-1)\epsilon^2}
\label{seq:12}
\end{equation}
by the Pythagorean theorem regardless of \gote's choices. The solution $\{R_n\}$ of \eqref{seq:12} is explicitly obtained by
\begin{equation}
\label{game_traj}
R_n=\sqrt{R_0^2+2(d-1)n\epsilon^2}.
\end{equation} 

On the other hand, if \gote takes a \zcenc{$z$} through the game, we have
\begin{align*}
\left\lvert x+\sum_{j}(\sqrt{2}\epsilon b^j v^j)-z\right\rvert^2&=\lvert x-z\rvert^2+2(d-1)\epsilon^2+\sum_{j}\sqrt{2}\epsilon \langle b^j v^j, x-z\rangle \\
&\ge \lvert x-z\rvert^2+2(d-1)\epsilon^2,
\end{align*}
which implies $d_{n+1}\ge\sqrt{d_n^2+2(d-1)\epsilon^2}$. Therefore we obtain $d_n\ge \sqrt{d_0^2+2(d-1)n\epsilon^2}$.

\begin{rem}
When \gote takes a \zcenc{$z$}, she can control the distance $|x_n-z|$. However we notice that she can not control the moves in tangential direction. For instance, let $d=2$, $z=(0,0)$ and the current game position $x_n=(0,1)$. If \sente chooses $v=(1,0)$ at this point, both $b=1$ and $b=-1$ are $(0,0)$ concentric strategies by Carol. One may think that if \gote makes the further decision that for \sente's choice $v$ tangential to the circle centered at the origin passing through $x_n$, she chooses $b$ so that $bv$ becomes clockwise, then she could control the trajectory of the game to be clockwise. However this is not true. \gote's greedy attempt to move as she pleases in the tangential direction will be thwarted by Paul. Indeed when \gote makes above decision, \sente can move counterclockwise while the distance $|x_n-z|$ meets almost the condition \eqref{game_traj} by slightly leaning the vector $v$ from tangential one. (e.g. $v=(\cos{-\epsilon^2},\sin{-\epsilon^2})$ at $x_n=(0,1)$)
\end{rem}

\section{Asymptotic behavior of solutions}
\label{sec:main}Throughout this section we consider the main equation \eqref{eq:main}.

In the following lemma, we estimate the asymptotic shape from above by considering \gote's strategies.
\begin{lem}
\label{lem:elgeom}
\[\varlimsup_{t\to \infty}E_t\subset Co(C_{-}).\]
\end{lem}
\begin{proof}
Let $u$ be the unique solution to \eqref{eq:application} with $u_0$ and $\Psi_{-}$ that are as in Definition \ref{def:evolution}. We notice that the conclusion holds if and only if for $x\in Co(C_{-})^c$, there exists $\tau>0$ such that $u(x,t)<0$ for $t>\tau$. To prove $u^{\epsilon}<0$, it is sufficient to give a \gote's strategy that makes the game cost negative but is not necessarily optimal one. For $x\in Co(C_{-})^{c}$ we can take an open ball $B$ such that $C_{-}\subset B$ and $x\in\partial B$ by the hyperplane separation theorem and the boundedness of $Co(C_{-})$. Let $z$ be the center of $B$ and $r=\lvert z-x\rvert$. If \gote takes a \zcenc{$z$}, then we see that regardless of \sente's choice, the game trajectory $\{x_n\}$ satisfies $\lvert x_n-z\rvert\ge \sqrt{r^2+2n(d-1)\epsilon^2}$, where $\sqrt{r^2+2n(d-1)\epsilon^2}$ is the solution of \eqref{seq:12} with $R_0=r$. Letting $\tau=2R(R+r)/(d-1)$, we have 
\begin{equation}
\label{eq:0709}
\lvert x_N-z\rvert\ge r+2R
\end{equation} 
for the last position $x_N$ of the game, where $R>0$ is a constant taken in \eqref{assum:ini}. The inequality \eqref{eq:0709} and $Co(C_{-})\subset B$ imply $dist(x_N, Co(C_{-}))\ge dist(x_N, B)\ge 2R$. Also $C_{-}\subset E_0\subset B_R(0)$ implies $Co(C_{-})\subset B_R(0)$. Hence we have $x_N\notin B_R(0)$, which means $u_0(x_N)=a<0$ by \eqref{assum:ini}. If \sente quits the game on the way, the stopping cost is at most
\[\sup_{b\in B^c}\Psi_{-}(b)<0.\]
The comparison principle (\cite{koike}) for \eqref{eq:application} and the convergence results in Appendix \ref{app_ginp} imply that $u^{\epsilon}(x,t)$ converges to $u(x,t)$ locally uniformly in $(x,t)$. See also \cite[Chapter V Lemma 1.9]{bardi} if necessary. We also notice that the uniform boundedness of $u^{\epsilon}$ is satisfied owe to the rule of the game and the boundedness of $u_0$ and $\Psi_{-}$. Since both upper bound of the terminal cost $a$ and that of the stopping cost $\sup_{b\in B^c}\Psi_{-}(b)$ do not depend on $\epsilon$, we conclude that $u(x,t)<0$ for $t>\tau$ and $x\in Co(C_{-})^{c}$.
\end{proof}

\begin{figure}[h]
\begin{center}
\input{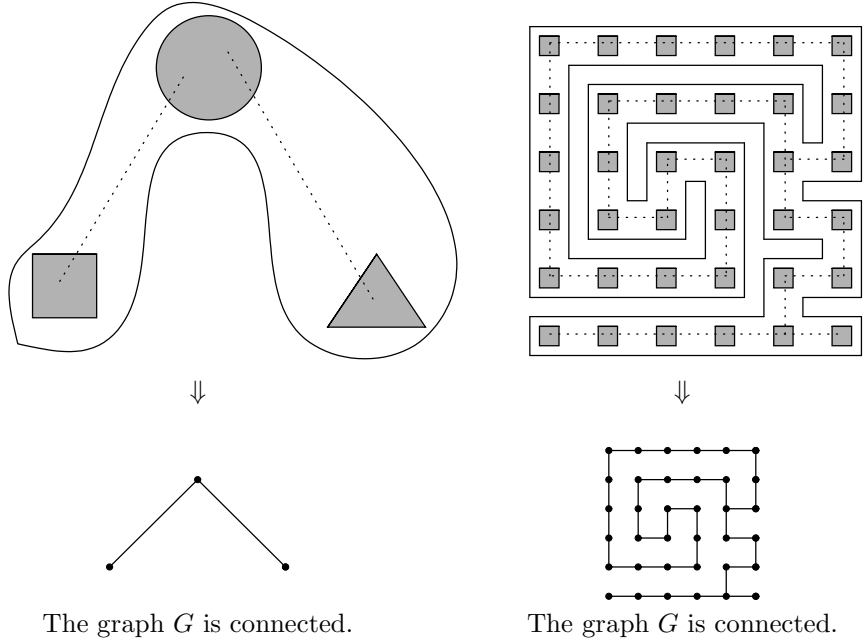}
\end{center}
\hspace{10truemm}
\caption{$D_0$ and $O_{-}$ that satisfy the assumption of Theorem \ref{thm:conv_general} (In this figure the loops of $G$ are ignored.)}
\label{fig:complex_obstacle}
\end{figure}

\begin{figure}[h]
\begin{center}
\input{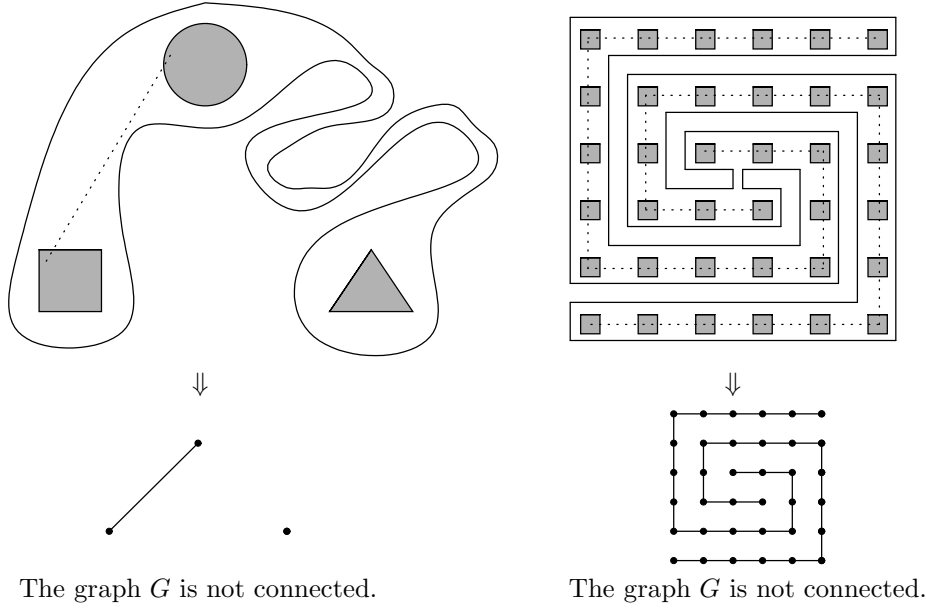}
\end{center}
\hspace{10truemm}
\caption{$D_0$ and $O_{-}$ that do not satisfy the assumption of Theorem \ref{thm:conv_general}}
\label{fig:complex_obstacle_gngn}
\end{figure}

\begin{figure}[h]
\begin{center}
\input{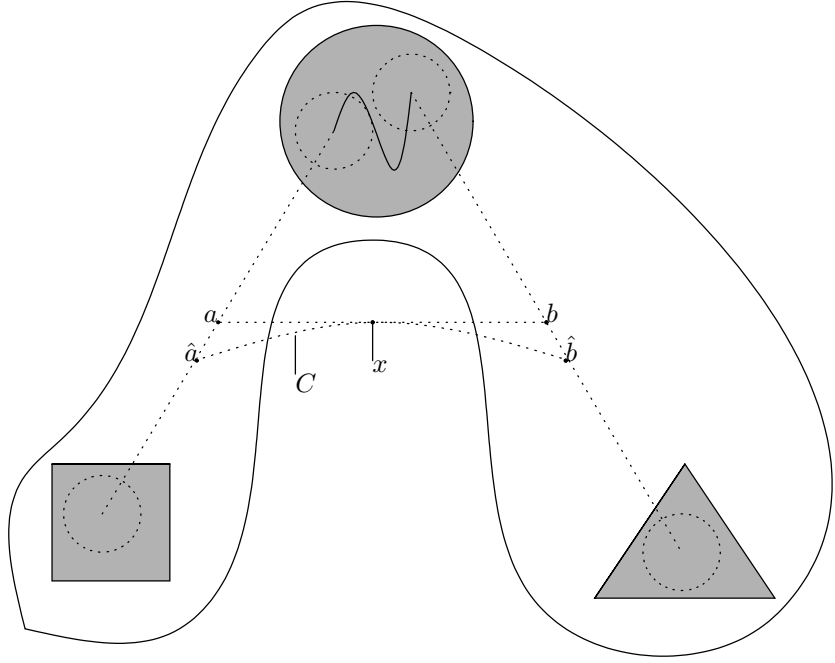}
\end{center}
\caption{strategy in the case 3)}
\label{fig:concentric_strategy}
\end{figure}

For an obstacle $O_{-}$ and an initial set $D_0$, we define the graph $G=(V,E)$ as follows: 
\[V:=\{O\subset \mathbb{R}^2 \mid O \ \mbox{ is a connected component of} \ O_{-}\},\]
\[E:=\{\langle O,P\rangle \mid O,P\in V \ \mbox{and $l_{x,y}\subset D_0$ for some $x\in O$ and $y\in P$}\}.\] 
See Appendix \ref{sec:graphtheory} for definitions of terms in graph theory.

\begin{thm}
\label{thm:conv_general}
Assume that $d=2$ and the graph $G$ is connected. Then 
\[Co(O_{-})\subset \varliminf_{t\to\infty}D_t \subset \varlimsup_{t\to\infty}D_t\subset \overline{Co(O_{-})}\]
and
\[Co(O_{-})\subset \varliminf_{t\to\infty}E_t \subset \varlimsup_{t\to\infty}E_t\subset \overline{Co(O_{-})}.\]
\end{thm}
\begin{rem}
Figure \ref{fig:complex_obstacle} (resp. Figure \ref{fig:complex_obstacle_gngn}) shows examples of $D_0$ and $O_{-}$ that satisfy (resp. do not satisfy) the assumption of Theorem \ref{thm:conv_general}. 
\end{rem}
\begin{proof}
For $D_0$ and $E_0$, we take $u_0$ as in Definition \ref{def:evolution}. Indeed it suffices to let
\begin{equation}
\label{good_ini}
u_0(x)=\begin{cases}
dist(x, \partial D_0), \ x\in D_0 \\
\max\{a, -dist(x, \partial D_0)\}, \ x\in D_0^c.
\end{cases}
\end{equation}
Similarly it suffices to let
\begin{equation}
\label{good_obs}
\Psi_{-}(x)=\begin{cases}
dist(x, \partial O_{-}), \ x\in O_{-} \\
\max\{a, -dist(x, \partial O_{-})\}, \ x\in O_{-}^c.
\end{cases}
\end{equation}
By Lemma \ref{lem:elgeom} and $D_t\subset E_t$, it suffices to prove $Co(O_{-})\subset \varliminf_{t\to\infty}D_t$. Namely our goal is to prove that for $x\in Co(O_{-})$, there exists $\tau>0$ such that $u(x,t)>0$ for $t>\tau$.
To prove $u^{\epsilon}>0$, it is sufficient to give a \sente's strategy, which makes the game cost positive and is not necessarily optimal one. Let $x\in Co(O_{-})$. It would be convenient to introduce the set
\[L:=\{z\in l_{x,y} \mid x,y\in O_{-}, l_{x,y}\subset D_0\}\]
in doing case analysis for $x\in Co(O_{-})$.

\bm{$1)~x\in O_{-}.$} In this case, it suffices for \sente to quit the game at the first round and gain the stopping cost $\Psi_{-}(x)>0$. Recall that $u^{\epsilon}(x,t)$ converges to $u(x,t)$ locally uniformly in $(x,t)$. Thus we obtain $u(x,t)>0$ for any $t>0$. 

\bm{$2)~x\in L\setminus O_{-}.$} Let $z,w\in O_{-}$ satisfy $x\in l_{z,w}$ and $l_{z,w}\subset D_0$. We take $\delta>0$ to satisfy $\overline{B_{\delta}(z)}\subset O_{-}$ and $\overline{B_{\delta}(w)}\subset O_{-}$. For the initial position $x$, \sente's strategy is to keep taking $v=\frac{z-x}{\lvert z-x\rvert}$ until he reaches $B_{\delta}(z)\cup B_{\delta}(w)$. If he reaches $B_{\delta}(z)\cup B_{\delta}(w)$, then he quits the game. By doing this, \sente gains positive game cost in either case he quits the game or not. See Figure \ref{fig:0525}. More precisely, \sente gains at least
\[\min\left\{\min_{y\in \overline{B_{\delta}(z)}\cup \overline{B_{\delta}(w)}}\Psi_{-}(y), \min_{y\in l_{z,w}}u_0(y)\right\}>0\] regardless of $\epsilon\in(0,\delta/\sqrt{2})$, where $\epsilon$ is taken small enough for \sente not to stride over $B_{\delta}(z)$ or $B_{\delta}(w)$. Hence, as in the case 1), we obtain $u(x,t)>0$ for any $t>0$.

\bm{$3)~x\in Co(O_{-})\setminus L.$} Henceforth we give a strategy by \sente that includes a \zcenp{$z$} and makes the game cost positive. To do so, we are going to construct a closed curve that consists of an arc $C$ with its center at $z$ and a path $\hat{\Gamma}$ in $L$. Since $O_{-}\subset L\subset Co(O_{-})$, we have $Co(O_{-})=Co(L)$. By Lemma \ref{lem:convCara} in Appendix \ref{app:set}, we can take $a,b\in L$ such that $x\in l_{a,b}$. We only show the case $a,b\notin O_{-}$, since otherwise we would prove it in a simpler manner. 

We first explain how to construct a path $\Gamma\subset L$ that contains $a$ and $b$. We take a specific path rather than just a path. By doing so, we are able to indicate a region that includes final positions of the games to guarantee that $u^{\epsilon}$ is uniformly positive. Since $a\in L$, there is a line segment that is in $D_0$, contains $a$, and has endpoints in some connected components $A$ and $B$ of $O_{-}$ respectively. Similarly there is a line segment that is in $D_0$, contains $b$, and has endpoints in some connected components $C$ and $D$ of $O_{-}$ respectively. Notice that $\langle A,B\rangle,\langle C,D\rangle\in E$, recalling $E$ is the set of unordered pairs of the graph $G$ defined above. From Proposition \ref{prop:graph} there is a path $P=(V^{\prime},E^{\prime})$ of the graph $G$ such that $A,B,C,D\in V^{\prime}$ and $\langle A,B\rangle,\langle C,D\rangle\in E^{\prime}$. Writing $V^{\prime}=\{O_0,O_1,\cdots,O_n\}$ and $E^{\prime}=\{\langle O_i,O_{i+1} \rangle \mid i=0,1,\cdots,n-1\}$, we see that for some points $y_i,\tilde{y}_{i}\in O_{i}$ $(i=0,1,\cdots,n-1)$, there are line segments $l_{\tilde{y}_i, y_{i+1}}$ $(i=0,1,\cdots,n-1)$ such that $a\in l_{\tilde{y}_0, y_{1}}$, $b\in l_{\tilde{y}_{n-1}, y_{n}}$ and $l_{\tilde{y}_i, y_{i+1}}\subset D_0$ $(i=0,1,\cdots,n-1)$. Let $\Gamma_i$ be a polygonal line in $O_i$ with endpoints $y_i$ and $\tilde{y}_i$. Now we define
\[\Gamma:=l_{\tilde{y}_0,y_1}\cup\Gamma_1\cup l_{\tilde{y}_1,y_2}\cup \Gamma_2\cup l_{\tilde{y}_2,y_3}\cup \cdots \cup l_{\tilde{y}_{n-2}, y_{n-1}}\cup\Gamma_{n-1}\cup l_{\tilde{y}_{n-1}, y_n}.\]

We may assume $l_{a,b}\nparallel l_{\tilde{y}_0,y_1}$ and $l_{a,b}\nparallel l_{\tilde{y}_{n-1}, y_n}$, since otherwise we would retake either $a$ or $b$ as an element of $O_{-}$. Without loss of generality we can also assume that $\Gamma$ and $l_{a,b}$ do not cross each other except at $a$ and $b$. By Lemma \ref{lem:topolo} in Appendix \ref{app:set} we take $\delta>0$ small enough to satisfy $B_{3\delta}(\Gamma)\subset L$, noticing that $L$ is an open set and $\Gamma$ is a compact set. Let $w_0,w_1,w_2$ be unit vectors in $\mathbb{R}^2$ such that $w_0\parallel l_{a,b}$, $w_1\parallel l_{\tilde{y}_0,y_1}$, $w_2\parallel l_{\tilde{y}_{n-1}, y_n}$ and $(w_0\cdot w_1)(w_0\cdot w_2)\ge 0$. Let $\hat{a}=a+\delta w_1$ and $\hat{b}=b+\delta w_2$. We temporarily define $C$ as the arc passing through $\hat{a}$, $\hat{b}$ and $x$. Combining $\Gamma$ and $C$, we make a closed curve $C\cup \hat{\Gamma}$, where
\[\hat{\Gamma}:=l_{\hat{a},y_1}\cup\Gamma_1\cup l_{\tilde{y}_1,y_2}\cup \Gamma_2\cup l_{\tilde{y}_2,y_3}\cup \cdots \cup l_{\tilde{y}_{n-2}, y_{n-1}}\cup\Gamma_{n-1}\cup l_{\tilde{y}_{n-1}, \hat{b}}.\]

\begin{figure}[h]
\begin{center}
\input{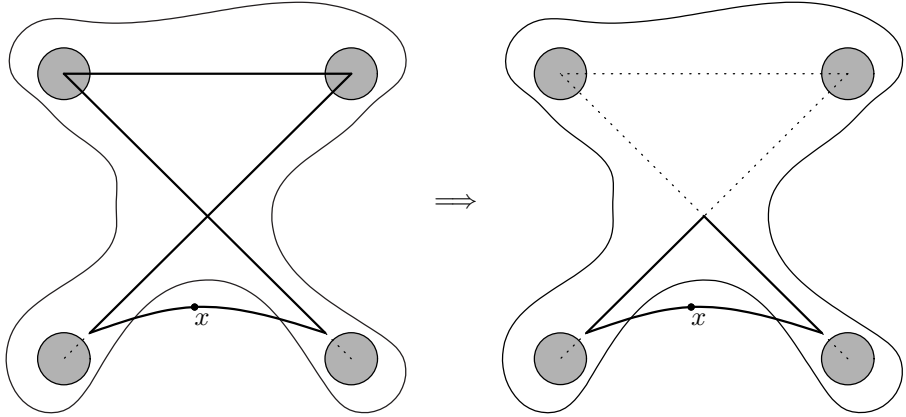}
\end{center}
\caption{make a Jordan closed curve}
\label{fig:jordan}
\end{figure}

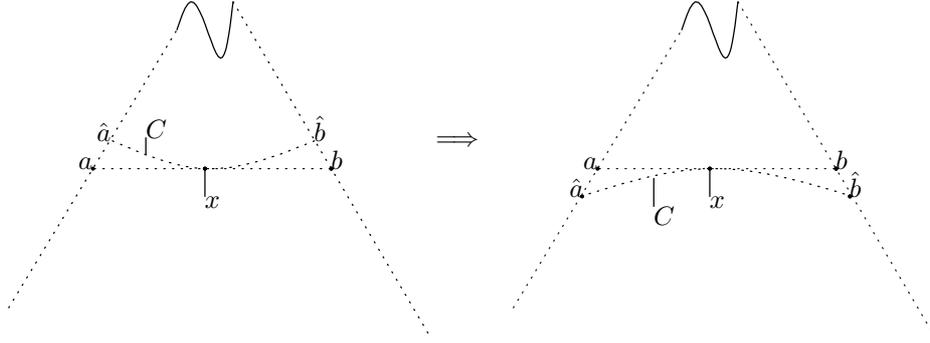
\begin{figure}[h]
\begin{center}
{\unitlength 0.1in%
\begin{picture}(47.9200,17.4300)(6.0000,-25.4300)%
%
\special{pn 8}%
\special{pa 3213 2397}%
\special{pa 4085 946}%
\special{dt 0.045}%
\special{pa 4376 800}%
\special{pa 5392 2543}%
\special{dt 0.045}%
%
\special{pn 8}%
\special{pa 3657 1671}%
\special{pa 4884 1671}%
\special{dt 0.045}%
%
\special{pn 4}%
\special{sh 1}%
\special{ar 4230 1671 8 8 0 6.2831853}%
\special{sh 1}%
\special{ar 4230 1671 8 8 0 6.2831853}%
%
\special{pn 8}%
\special{pa 3940 1722}%
\special{pa 3940 1868}%
\special{fp}%
\put(39.4000,-18.6800){\makebox(0,0)[lt]{$C$}}%
%
\special{pn 8}%
\special{pa 4230 1671}%
\special{pa 4230 1816}%
\special{fp}%
\put(42.3000,-18.1600){\makebox(0,0)[lt]{$x$}}%
%
\special{pn 8}%
\special{pn 8}%
\special{pa 3562 1816}%
\special{pa 3570 1814}%
\special{fp}%
\special{pa 3605 1802}%
\special{pa 3612 1800}%
\special{fp}%
\special{pa 3648 1789}%
\special{pa 3655 1787}%
\special{fp}%
\special{pa 3690 1776}%
\special{pa 3698 1774}%
\special{fp}%
\special{pa 3733 1763}%
\special{pa 3741 1761}%
\special{fp}%
\special{pa 3776 1750}%
\special{pa 3784 1748}%
\special{fp}%
\special{pa 3820 1739}%
\special{pa 3827 1737}%
\special{fp}%
\special{pa 3863 1727}%
\special{pa 3871 1725}%
\special{fp}%
\special{pa 3907 1716}%
\special{pa 3914 1714}%
\special{fp}%
\special{pa 3950 1706}%
\special{pa 3958 1705}%
\special{fp}%
\special{pa 3994 1697}%
\special{pa 4002 1696}%
\special{fp}%
\special{pa 4038 1690}%
\special{pa 4046 1689}%
\special{fp}%
\special{pa 4083 1683}%
\special{pa 4091 1681}%
\special{fp}%
\special{pa 4127 1678}%
\special{pa 4135 1677}%
\special{fp}%
\special{pa 4172 1674}%
\special{pa 4180 1673}%
\special{fp}%
\special{pa 4217 1671}%
\special{pa 4225 1671}%
\special{fp}%
\special{pa 4262 1671}%
\special{pa 4270 1671}%
\special{fp}%
\special{pa 4307 1672}%
\special{pa 4314 1672}%
\special{fp}%
\special{pa 4351 1675}%
\special{pa 4359 1675}%
\special{fp}%
\special{pa 4396 1679}%
\special{pa 4404 1680}%
\special{fp}%
\special{pa 4440 1684}%
\special{pa 4448 1686}%
\special{fp}%
\special{pa 4485 1691}%
\special{pa 4493 1693}%
\special{fp}%
\special{pa 4529 1699}%
\special{pa 4537 1701}%
\special{fp}%
\special{pa 4573 1708}%
\special{pa 4581 1710}%
\special{fp}%
\special{pa 4617 1718}%
\special{pa 4624 1719}%
\special{fp}%
\special{pa 4660 1728}%
\special{pa 4668 1730}%
\special{fp}%
\special{pa 4704 1739}%
\special{pa 4712 1741}%
\special{fp}%
\special{pa 4747 1751}%
\special{pa 4755 1753}%
\special{fp}%
\special{pa 4790 1763}%
\special{pa 4798 1765}%
\special{fp}%
\special{pa 4833 1775}%
\special{pa 4841 1777}%
\special{fp}%
\special{pa 4877 1788}%
\special{pa 4884 1790}%
\special{fp}%
\special{pa 4919 1801}%
\special{pa 4927 1803}%
\special{fp}%
\special{pa 4962 1814}%
\special{pa 4970 1816}%
\special{fp}%
%
\special{pn 4}%
\special{sh 1}%
\special{ar 3649 1671 8 8 0 6.2831853}%
\special{sh 1}%
\special{ar 4884 1671 8 8 0 6.2831853}%
\special{sh 1}%
\special{ar 4884 1671 8 8 0 6.2831853}%
\put(36.4900,-16.7100){\makebox(0,0)[rb]{$a$}}%
\put(48.8400,-16.7100){\makebox(0,0)[lb]{$b$}}%
%
\special{pn 4}%
\special{sh 1}%
\special{ar 3569 1816 8 8 0 6.2831853}%
\special{sh 1}%
\special{ar 4956 1816 8 8 0 6.2831853}%
\special{sh 1}%
\special{ar 4956 1816 8 8 0 6.2831853}%
\put(35.7700,-18.1600){\makebox(0,0)[rb]{$\hat{a}$}}%
\put(49.5600,-18.1600){\makebox(0,0)[lb]{$\hat{b}$}}%
%
\special{pn 8}%
\special{pa 4085 946}%
\special{pa 4099 905}%
\special{pa 4113 866}%
\special{pa 4127 834}%
\special{pa 4142 811}%
\special{pa 4156 800}%
\special{pa 4172 804}%
\special{pa 4187 821}%
\special{pa 4202 849}%
\special{pa 4218 883}%
\special{pa 4233 923}%
\special{pa 4247 964}%
\special{pa 4262 1003}%
\special{pa 4275 1039}%
\special{pa 4288 1068}%
\special{pa 4300 1087}%
\special{pa 4311 1094}%
\special{pa 4321 1089}%
\special{pa 4330 1073}%
\special{pa 4339 1047}%
\special{pa 4346 1014}%
\special{pa 4353 974}%
\special{pa 4360 929}%
\special{pa 4366 881}%
\special{pa 4372 830}%
\special{pa 4376 800}%
\special{fp}%
%
\special{pn 8}%
\special{pa 600 2397}%
\special{pa 1471 946}%
\special{dt 0.045}%
\special{pa 1762 800}%
\special{pa 2779 2543}%
\special{dt 0.045}%
%
\special{pn 8}%
\special{pa 1043 1671}%
\special{pa 2270 1671}%
\special{dt 0.045}%
%
\special{pn 4}%
\special{sh 1}%
\special{ar 1616 1671 8 8 0 6.2831853}%
\special{sh 1}%
\special{ar 1616 1671 8 8 0 6.2831853}%
%
\special{pn 8}%
\special{pa 1616 1671}%
\special{pa 1616 1816}%
\special{fp}%
\put(16.1600,-18.1600){\makebox(0,0)[lt]{$x$}}%
%
\special{pn 4}%
\special{sh 1}%
\special{ar 1036 1671 8 8 0 6.2831853}%
\special{sh 1}%
\special{ar 2270 1671 8 8 0 6.2831853}%
\special{sh 1}%
\special{ar 2270 1671 8 8 0 6.2831853}%
\put(10.3600,-16.7100){\makebox(0,0)[rb]{$a$}}%
\put(22.7000,-16.7100){\makebox(0,0)[lb]{$b$}}%
%
\special{pn 8}%
\special{pa 1471 946}%
\special{pa 1485 905}%
\special{pa 1499 866}%
\special{pa 1513 834}%
\special{pa 1528 811}%
\special{pa 1542 800}%
\special{pa 1558 804}%
\special{pa 1573 821}%
\special{pa 1588 849}%
\special{pa 1604 883}%
\special{pa 1619 923}%
\special{pa 1633 964}%
\special{pa 1648 1003}%
\special{pa 1661 1039}%
\special{pa 1674 1068}%
\special{pa 1686 1087}%
\special{pa 1697 1094}%
\special{pa 1707 1089}%
\special{pa 1716 1073}%
\special{pa 1725 1047}%
\special{pa 1732 1014}%
\special{pa 1739 974}%
\special{pa 1746 929}%
\special{pa 1752 881}%
\special{pa 1758 830}%
\special{pa 1762 800}%
\special{fp}%
%
\special{pn 8}%
\special{pn 8}%
\special{pa 1137 1526}%
\special{pa 1144 1529}%
\special{fp}%
\special{pa 1179 1544}%
\special{pa 1186 1547}%
\special{fp}%
\special{pa 1220 1561}%
\special{pa 1228 1564}%
\special{fp}%
\special{pa 1262 1579}%
\special{pa 1269 1582}%
\special{fp}%
\special{pa 1304 1596}%
\special{pa 1311 1599}%
\special{fp}%
\special{pa 1346 1611}%
\special{pa 1354 1614}%
\special{fp}%
\special{pa 1389 1626}%
\special{pa 1397 1628}%
\special{fp}%
\special{pa 1432 1639}%
\special{pa 1440 1641}%
\special{fp}%
\special{pa 1476 1651}%
\special{pa 1484 1652}%
\special{fp}%
\special{pa 1520 1660}%
\special{pa 1528 1661}%
\special{fp}%
\special{pa 1565 1667}%
\special{pa 1573 1668}%
\special{fp}%
\special{pa 1610 1670}%
\special{pa 1618 1671}%
\special{fp}%
\special{pa 1655 1671}%
\special{pa 1663 1671}%
\special{fp}%
\special{pa 1700 1669}%
\special{pa 1708 1668}%
\special{fp}%
\special{pa 1745 1665}%
\special{pa 1753 1664}%
\special{fp}%
\special{pa 1790 1656}%
\special{pa 1798 1655}%
\special{fp}%
\special{pa 1834 1647}%
\special{pa 1842 1645}%
\special{fp}%
\special{pa 1878 1636}%
\special{pa 1886 1634}%
\special{fp}%
\special{pa 1921 1624}%
\special{pa 1929 1621}%
\special{fp}%
\special{pa 1964 1609}%
\special{pa 1972 1607}%
\special{fp}%
\special{pa 2007 1595}%
\special{pa 2015 1592}%
\special{fp}%
\special{pa 2049 1579}%
\special{pa 2057 1576}%
\special{fp}%
\special{pa 2091 1562}%
\special{pa 2099 1559}%
\special{fp}%
\special{pa 2134 1546}%
\special{pa 2141 1543}%
\special{fp}%
\special{pa 2176 1529}%
\special{pa 2183 1526}%
\special{fp}%
\put(11.3000,-15.2600){\makebox(0,0)[rb]{$\hat{a}$}}%
\put(21.8300,-15.2600){\makebox(0,0)[lb]{$\hat{b}$}}%
\put(29.2300,-15.2600){\makebox(0,0){$\Longrightarrow$}}%
%
\special{pn 8}%
\special{pa 1310 1599}%
\special{pa 1310 1510}%
\special{fp}%
\put(13.1000,-15.1000){\makebox(0,0)[lb]{$C$}}%
\end{picture}}%
\end{center}
\caption{retake $C$ and $\Omega$}
\label{fig:switch}
\end{figure}

For the closed curve $C\cup \hat{\Gamma}$ and $x$ in it, there is a Jordan closed curve $\hat{C}$ such that $x\in\hat{C}$ and $\hat{C}\subset C\cup \hat{\Gamma}$ (Figure \ref{fig:jordan}). See Appendix \ref{sec:curve} in detail. Thus, based on the Jordan curve theorem, we let $\Omega$ be the bounded domain that satifies $\partial \Omega=\hat{C}$. If $\Omega$ touches the arc $C$ from inside, then we retake the other pair of $(\hat{a}, \hat{b})$ and, together with it, retake $C$ and $\Omega$ so that $\Omega$ touches the arc $C$ from outside (Figure \ref{fig:switch}). We notice that the domain enclosed by $l_{\tilde{y}_0,y_1}$, $l_{\tilde{y}_{n-1}, y_n}$ and the two arcs shown in Figure \ref{fig:switch} is bounded, and hence, so is the new $\Omega$. We further notice that we can take $C$ so that $C$ and $\hat{\Gamma}$ intersect only at $\hat{a}$ and $\hat{b}$ by taking $\delta$ smaller if necessary. 

We now give a strategy by \sente for the initial position $x$. \sente first takes a \zcenp{$z$}, where $z$ is the center of the arc $C$. If \sente enters $B_{\delta}(\Gamma_i)$ for some $i$, then he quits the game at this point. Once he enters $B_{\delta}(l_{\tilde{y}_i,y_{i+1}})$ for some $i$, he takes a similar strategy to that in the case 2). To see that it attains positive game cost, we notice two propeties of game trajectories $\{x_n\}$ given when \sente keeps taking a \zcenp{$z$} by some round. One is that $x_n\in B\left(z, \sqrt{\lvert x_0-z\rvert^2+2n\epsilon^2}\right)^c$. The other is that $x_n\in (C\cup \Omega)\setminus N_{\delta}$ implies $x_{n+1}\in\Omega$ for $\sqrt{2}\epsilon<\delta$, where we denote $(\bigcup_{i}B_{\delta}(\Gamma_i))\cup (\bigcup_{i}B_{\delta}(l_{\tilde{y}_i,y_{i+1}}))$ by $N_{\delta}$. Thus it takes at most finite time $\tau$ (satisfying $\Omega\subset B(z,\sqrt{\lvert x_0-z\rvert^2+2\tau})$) for \sente to reach $N_{\delta}$. We see that the game ends at some point in $N_{2\delta}$ and \sente gains at least 
\[\min\left\{\inf_{y\in \bigcup_{i} B_{2\delta}(\Gamma_i)}\Psi_{-}(y), \inf_{y\in \bigcup_{i}B_{\delta}(l_{\tilde{y}_i,y_{i+1}})}u_0(y)\right\}>0,\]
regardless of $\epsilon\in(0,\delta/\sqrt{2})$. Therefore we conclude that $u(x,t)>0$ for any $t>\tau$, noticing that $\tau$ may depend on $x$, but does not depend on $\epsilon$.
\end{proof}

\begin{rem}
In general, fattening of the level set may occur under the assumption of Theorem \ref{thm:conv_general}. i.e., $\overline{D_t}=E_t$ may fail at some $t>0$. Theorem \ref{thm:conv_general} states that even if the curve evolutions are not unique, they have the same limit.
\end{rem}
We give some sufficient conditions to the assumption of Theorem \ref{thm:conv_general}.
\begin{cor}
\label{thm:conv}
Assume that $Co(O_{-})\subset D_0$. Then the same conclusion as that of Theorem \ref{thm:conv_general} holds.
\end{cor}

\begin{figure}[htbp]
 \begin{minipage}{0.5\hsize}
  \begin{center}
{\unitlength 0.1in%
\begin{picture}(16.0000,16.0000)(10.0000,-28.0000)%
%
\special{pn 0}%
\special{sh 0.300}%
\special{pa 1200 1400}%
\special{pa 1200 2600}%
\special{pa 2400 2600}%
\special{pa 2400 1400}%
\special{pa 2200 1400}%
\special{pa 2200 2400}%
\special{pa 1400 2400}%
\special{pa 1400 1400}%
\special{pa 1400 1400}%
\special{pa 1200 1400}%
\special{ip}%
\special{pn 8}%
\special{pa 1200 1400}%
\special{pa 1200 2600}%
\special{pa 2400 2600}%
\special{pa 2400 1400}%
\special{pa 2200 1400}%
\special{pa 2200 2400}%
\special{pa 1400 2400}%
\special{pa 1400 1400}%
\special{pa 1200 1400}%
\special{pa 1200 2600}%
\special{fp}%
%
\special{pn 8}%
\special{pa 1000 1200}%
\special{pa 1000 2800}%
\special{pa 2600 2800}%
\special{pa 2600 1200}%
\special{pa 2000 1200}%
\special{pa 2000 2200}%
\special{pa 1600 2200}%
\special{pa 1600 1200}%
\special{pa 1000 1200}%
\special{pa 1000 2800}%
\special{fp}%
\end{picture}}%
  \end{center}
  \caption{$O_{-}$ is connected}
  \label{fig:cnt_obstacle_3}
 \end{minipage}
 \begin{minipage}{0.5\hsize}
  \begin{center}
{\unitlength 0.1in%
\begin{picture}(16.0000,18.0000)(10.0000,-28.0000)%
%
\special{pn 0}%
\special{sh 0.300}%
\special{pa 1200 1400}%
\special{pa 1200 2600}%
\special{pa 2400 2600}%
\special{pa 2400 1400}%
\special{pa 2200 1400}%
\special{pa 2200 2400}%
\special{pa 1400 2400}%
\special{pa 1400 1400}%
\special{pa 1400 1400}%
\special{pa 1200 1400}%
\special{ip}%
\special{pn 8}%
\special{pa 1200 1400}%
\special{pa 1200 2600}%
\special{pa 2400 2600}%
\special{pa 2400 1400}%
\special{pa 2200 1400}%
\special{pa 2200 2400}%
\special{pa 1400 2400}%
\special{pa 1400 1400}%
\special{pa 1200 1400}%
\special{pa 1200 2600}%
\special{fp}%
%
\special{pn 8}%
\special{pa 1000 1200}%
\special{pa 1000 2800}%
\special{pa 2600 2800}%
\special{pa 2600 1200}%
\special{pa 2000 1200}%
\special{pa 2000 2200}%
\special{pa 1600 2200}%
\special{pa 1600 1200}%
\special{pa 1000 1200}%
\special{pa 1000 2800}%
\special{fp}%
%
\special{pn 4}%
\special{sh 1}%
\special{ar 1800 2000 10 10 0 6.2831853}%
\special{sh 1}%
\special{ar 1800 2000 10 10 0 6.2831853}%
\put(17.7000,-20.3000){\makebox(0,0)[rt]{$x$}}%
%
\special{pn 8}%
\special{ar 1800 1500 500 500 0.6435011 2.4980915}%
%
\special{pn 8}%
\special{pn 8}%
\special{pa 2300 1500}%
\special{pa 2300 1508}%
\special{fp}%
\special{pa 2298 1545}%
\special{pa 2297 1553}%
\special{fp}%
\special{pa 2292 1589}%
\special{pa 2291 1597}%
\special{fp}%
\special{pa 2282 1633}%
\special{pa 2280 1641}%
\special{fp}%
\special{pa 2268 1676}%
\special{pa 2265 1683}%
\special{fp}%
\special{pa 2250 1717}%
\special{pa 2247 1724}%
\special{fp}%
\special{pa 2229 1757}%
\special{pa 2225 1763}%
\special{fp}%
\special{pa 2205 1794}%
\special{pa 2200 1800}%
\special{fp}%
\special{pa 2177 1829}%
\special{pa 2171 1835}%
\special{fp}%
\special{pa 2145 1862}%
\special{pa 2139 1867}%
\special{fp}%
\special{pa 2112 1891}%
\special{pa 2105 1896}%
\special{fp}%
\special{pa 2075 1918}%
\special{pa 2069 1922}%
\special{fp}%
\special{pa 2037 1940}%
\special{pa 2030 1944}%
\special{fp}%
\special{pa 1997 1960}%
\special{pa 1989 1963}%
\special{fp}%
\special{pa 1954 1975}%
\special{pa 1947 1978}%
\special{fp}%
\special{pa 1911 1987}%
\special{pa 1904 1989}%
\special{fp}%
\special{pa 1867 1996}%
\special{pa 1859 1996}%
\special{fp}%
\special{pa 1823 2000}%
\special{pa 1815 2000}%
\special{fp}%
\special{pa 1778 1999}%
\special{pa 1770 1999}%
\special{fp}%
\special{pa 1733 1995}%
\special{pa 1725 1994}%
\special{fp}%
\special{pa 1689 1988}%
\special{pa 1681 1986}%
\special{fp}%
\special{pa 1646 1976}%
\special{pa 1638 1973}%
\special{fp}%
\special{pa 1604 1960}%
\special{pa 1596 1957}%
\special{fp}%
\special{pa 1563 1940}%
\special{pa 1557 1937}%
\special{fp}%
\special{pa 1525 1917}%
\special{pa 1518 1913}%
\special{fp}%
\special{pa 1488 1891}%
\special{pa 1482 1886}%
\special{fp}%
\special{pa 1455 1862}%
\special{pa 1449 1856}%
\special{fp}%
\special{pa 1424 1829}%
\special{pa 1418 1823}%
\special{fp}%
\special{pa 1395 1794}%
\special{pa 1391 1788}%
\special{fp}%
\special{pa 1371 1756}%
\special{pa 1367 1750}%
\special{fp}%
\special{pa 1349 1717}%
\special{pa 1346 1710}%
\special{fp}%
\special{pa 1332 1676}%
\special{pa 1329 1668}%
\special{fp}%
\special{pa 1318 1633}%
\special{pa 1316 1625}%
\special{fp}%
\special{pa 1308 1589}%
\special{pa 1307 1581}%
\special{fp}%
\special{pa 1302 1545}%
\special{pa 1301 1537}%
\special{fp}%
\special{pa 1300 1500}%
\special{pa 1300 1492}%
\special{fp}%
\special{pa 1302 1455}%
\special{pa 1303 1447}%
\special{fp}%
\special{pa 1308 1411}%
\special{pa 1310 1403}%
\special{fp}%
\special{pa 1318 1367}%
\special{pa 1320 1359}%
\special{fp}%
\special{pa 1332 1324}%
\special{pa 1335 1317}%
\special{fp}%
\special{pa 1349 1283}%
\special{pa 1353 1276}%
\special{fp}%
\special{pa 1371 1243}%
\special{pa 1375 1236}%
\special{fp}%
\special{pa 1396 1206}%
\special{pa 1400 1199}%
\special{fp}%
\special{pa 1423 1171}%
\special{pa 1429 1165}%
\special{fp}%
\special{pa 1455 1138}%
\special{pa 1461 1133}%
\special{fp}%
\special{pa 1489 1108}%
\special{pa 1495 1104}%
\special{fp}%
\special{pa 1525 1082}%
\special{pa 1531 1078}%
\special{fp}%
\special{pa 1564 1059}%
\special{pa 1571 1056}%
\special{fp}%
\special{pa 1604 1040}%
\special{pa 1612 1037}%
\special{fp}%
\special{pa 1646 1024}%
\special{pa 1654 1022}%
\special{fp}%
\special{pa 1689 1012}%
\special{pa 1697 1011}%
\special{fp}%
\special{pa 1733 1005}%
\special{pa 1741 1003}%
\special{fp}%
\special{pa 1778 1001}%
\special{pa 1786 1000}%
\special{fp}%
\special{pa 1823 1000}%
\special{pa 1831 1001}%
\special{fp}%
\special{pa 1867 1004}%
\special{pa 1875 1006}%
\special{fp}%
\special{pa 1911 1013}%
\special{pa 1919 1014}%
\special{fp}%
\special{pa 1954 1025}%
\special{pa 1962 1027}%
\special{fp}%
\special{pa 1997 1040}%
\special{pa 2004 1044}%
\special{fp}%
\special{pa 2037 1060}%
\special{pa 2044 1064}%
\special{fp}%
\special{pa 2075 1082}%
\special{pa 2082 1087}%
\special{fp}%
\special{pa 2112 1109}%
\special{pa 2118 1114}%
\special{fp}%
\special{pa 2145 1139}%
\special{pa 2151 1144}%
\special{fp}%
\special{pa 2176 1171}%
\special{pa 2182 1177}%
\special{fp}%
\special{pa 2205 1206}%
\special{pa 2210 1213}%
\special{fp}%
\special{pa 2229 1244}%
\special{pa 2233 1251}%
\special{fp}%
\special{pa 2250 1283}%
\special{pa 2254 1290}%
\special{fp}%
\special{pa 2268 1324}%
\special{pa 2271 1332}%
\special{fp}%
\special{pa 2282 1367}%
\special{pa 2284 1375}%
\special{fp}%
\special{pa 2292 1411}%
\special{pa 2293 1418}%
\special{fp}%
\special{pa 2298 1455}%
\special{pa 2299 1463}%
\special{fp}%
\special{pa 2300 1500}%
\special{pa 2300 1500}%
\special{fp}%
\put(17.7000,-15.3000){\makebox(0,0)[rt]{$z$}}%
%
\special{pn 4}%
\special{sh 1}%
\special{ar 1800 1500 10 10 0 6.2831853}%
\special{sh 1}%
\special{ar 1800 1500 10 10 0 6.2831853}%
\end{picture}}%
  \end{center}
  \caption{\sente's strategy to achieve positive game cost for the initial position $x$}
  \label{fig:cnt_obstacle_st}
 \end{minipage}
\end{figure}

\begin{cor}
Assume that $O_{-}$ is connected (Figure \ref{fig:cnt_obstacle_3} and \ref{fig:cnt_obstacle_st}).
Then the same conclusion as that of Theorem \ref{thm:conv_general} holds.
\end{cor}

Under the following assumption, the moving surface sticks to the obstacle in finite time:
\begin{equation}
\exists r>0 \ \forall w\in \partial O_{-} \ \exists z \in B_{r}(w), \  \mbox{s.t.} \ O_{-}\subset B_{\lvert z-w\rvert}(z).
\label{cond:inball}
\end{equation}
In the following theorem, there is no need to asuume $d=2$.
\begin{thm}
\label{thm:stick}
Assume \eqref{cond:inball}. Then 
\[\lim_{t\to\infty}D_t=O_{-}\]
and
\[\lim_{t\to\infty}E_t=C_{-}.\]
Moreover there exists $\tau>0$ such that
$D_t=D_{\tau}=O_{-}$ and $E_t=E_{\tau}=C_{-}$ for $t\ge \tau$.
\end{thm}
\begin{proof}
We notice that the condition \eqref{cond:inball} implies that $O_{-}$ is convex (Proposition \ref{prop:0707}). It is now clear that $O_{-}\subset D_t$ for any $t>0$ and hence $O_{-}\subset\bigcap_{t>0}D_t\subset \varliminf_{t\to\infty}D_t$. We prove that there exists $\tau>0$ such that $\bigcup_{t\ge\tau} E_t \subset \overline{O_{-}}$. Namely we show that there exists $\tau>0$ such that $u(x,t)<0$ for $t\ge\tau$ and $x\in (\overline{O_{-}})^c$. The difference from Lemma \ref{lem:elgeom} is that we now have to take $\tau$ independent of $x$. Indeed, for $x\in (\overline{O_{-}})^{c}$, we can take a ball $B_{\lvert z-w\rvert}(z)$ in \eqref{cond:inball} such that $x\notin B_{\lvert z-w\rvert}(z)$ and it suffices for \gote to keep taking a \zcenc{$z$} until the game ends. The value $r$ in \eqref{eq:0709} is now taken independent of $x$ and then so is $\tau$. Therefore we obtain
\[O_{-}\subset \varliminf_{t\to\infty}D_t \subset \varlimsup_{t\to\infty}D_t\subset \bigcup_{t\ge\tau}D_t \subset \bigcup_{t\ge\tau}E_t\subset\overline{O_{-}}.\]
Since $\bigcup_{t\ge\tau}D_t$ is open, we have $O_{-}=\bigcup_{t\ge\tau}D_t$, which means $\varliminf_{t\to\infty}D_t=\varlimsup_{t\to\infty}D_t$ and moreover $D_t=D_{\tau}=O_{-}$ for $t\ge \tau$.

We also have 
\[O_{-}\subset \bigcap_{t>0}D_t\subset \bigcap_{t>0}E_t \subset \varliminf_{t\to\infty}E_t \subset \varlimsup_{t\to\infty}E_t\subset \overline{O_{-}}.\]
Since $\bigcap_{t>0}E_t$ is closed, we similarly have $\overline{O_{-}}=\bigcap_{t>0}E_t$, which means $\varliminf_{t\to\infty}E_t=\varlimsup_{t\to\infty}E_t$ and moreover $E_t=E_0=\overline{O_{-}}$ for $t\ge 0$.
\end{proof}
\begin{rem}
The hair-clip solution, which can be regarded as an explicit solution to the curve shortening problem with the Dirichlet condition (See e.g. \cite{McDonald22}), implies that the solution of our obstacle problem can be apart from the asymptotic shape at any time. i.e., there are $D_0$ and $O_{-}$ such that $Co(O_{-})\subsetneq D_t$ for any $t>0$.
\end{rem}

\section{With driving force}
\label{sec:dri}
In this section we consider the following surface evolution equations in the plane that have obstacles on one side:
\begin{equation}
\label{eq:bobstacle}
\begin{cases}
V=-\kappa+\nu, \ \mbox{on $\partial D_t$}, \\
O_{-}\subset D_t, 
\end{cases}
\end{equation}
or
\begin{equation}
\label{eq:tobstacle}
\begin{cases}
V=-\kappa+\nu, \ \mbox{on $\partial D_t$}, \\
D_t\subset O_{+}, 
\end{cases}
\end{equation}
where $D_t$, $O_{-}$ and $O_{+}$ are open sets in $\mathbb{R}^2$ and $\nu>0$ is a constant. These equations are specific cases of \eqref{eq:general_ev}.

While Theorem \ref{thm:conv_general} is invarient with respect to similarity transformation, the behaviors of moving curves governed by \eqref{eq:bobstacle} or \eqref{eq:tobstacle} depend not only on shape of the obstacles and the initial curve but also on size of them. It is easily understood by considering the case $\partial D_0$ is a circle. For $D_0=B_R((0,0))$, the circle $\partial D_0$ shrinks if $R<\nu^{-1}$, and it spreads if $R>\nu^{-1}$ as time goes. So it is difficult to present a concise result as Theorem \ref{thm:conv_general}. We give several examples of computations of the asymptotic shapes here.
\subsection{Basic strategy of the game}
Concerning to the game interpretation for \eqref{eq:bobstacle}, the difference from the game in Section \ref{subsec:gi} is that \sente has the right to choose $w_i\in S^1$ at each round $i$ and the control system is
\begin{equation}
x_{i}=x_{i-1}+\sqrt{2} \epsilon b_i v_i+\nu \epsilon^2 w_i , \nonumber
\end{equation}
instead of \eqref{trajectory2}. In the game for \eqref{eq:tobstacle}, not \sente but \gote has the right to quit the game. If \gote quits the game at round $i$, then the cost is given by $\Psi_{+}(x_i)$. See Appendix \ref{app_ginp} in detail. As explained later in the proof of Lemma \ref{lem:elgeom}, the value functions $u^{\epsilon}$ locally uniformly converge to the solution $u$ of the corresponding level set equation. 

We now prepare several types of game strategies and give the properties of the game trajectories when they are used. The first one is similar to the one in Definition \ref{dfn:const0}.

\begin{dfn}[Concentric strategy]
Let $\nu>0$, $\epsilon>0$ and $z\in\mathbb{R}^2$. Let $x\in\mathbb{R}^2$ be the current position of the game. 
\begin{enumerate}
\item A choice $(v,w)\in S^1\times S^1$ by \sente is called a {\it \zcenp{$z$}} (by \sente) if 
\[w=\frac{z-x}{\lvert z-x\rvert}~\mbox{and}~v\perp w.\]
When $x=z$, any $(v,w)\in S^1\times S^1$ satisfying $v\perp w$ is called a {\it \zcenp{$z$}}.
\item Let $(v,w)\in S^1\times S^1$ be a choice by \sente in the same round. A choice $b\in \{\pm1\}$ by \gote is called a {\it \zcenc{$z$}} (by \gote) if
\[\langle b v, x+\nu\epsilon^2 w-z\rangle \ge0.\]
\end{enumerate}
\end{dfn}
As in Section \ref{subsec:bsotg}, let $d_n=\lvert x_n-z\rvert$ for fixed $z\in\mathbb{R}^2$. If \sente takes a \zcenp{$z$} through the game, then the sequence $\{d_n\}$ satisfies
\begin{equation}
R_{n+1}=\sqrt{(R_n-\nu \epsilon^2)^2+2\epsilon^2}.
\label{seq:2}
\end{equation}
If \gote takes a \zcenc{$z$} through the game, we have $d_n\ge R_n$, where $\{R_n\}$ is the solution to \eqref{seq:2} with $R_0=d_0$. In the following lemmas, we give basic propeties of the behaviors of the solutions to \eqref{seq:2}.
\begin{lem}
\label{lem:series2}
Fix $\nu>0$. Let $\epsilon>0$ and $\{R_n\}$ be a sequence satisfying the condition \eqref{seq:2}.
Then the following properties hold.
\begin{enumerate}
\item If $R_0=\nu^{-1}+\frac{\nu}{2}\epsilon^2$, then $R_n=\nu^{-1}+\frac{\nu}{2}\epsilon^2$ for all $n$. If $R_0> \nu^{-1}+\frac{\nu}{2}\epsilon^2$, then $R_n>\nu^{-1}+\frac{\nu}{2}\epsilon^2$ for all $n$ and $\{R_n\}$ is decreasing for $\epsilon\le \sqrt{2}\nu^{-1}$. If $\nu\epsilon^2\le R_0< \nu^{-1}+\frac{\nu}{2}\epsilon^2$, then $R_n<\nu^{-1}+\frac{\nu}{2}\epsilon^2$ for all $n$ and $\{R_n\}$ is increasing. 
\item \[\lim_{n \to \infty} R_n=\nu^{-1}+\frac{\nu}{2}\epsilon^2.\]
\item \begin{equation}
\label{ineq:shortcut}
R_{n+1}-R_n\le \frac{\epsilon^2}{R_n}-\nu\epsilon^2+\frac{\nu^2\epsilon^4}{2R_n}.
\end{equation}
If $R_{n+1}\ge R_n$, then
\begin{equation}
\label{ineq:shortcut2}
\frac{\epsilon^2}{R_{n+1}}-\nu\epsilon^2+\frac{\nu^2\epsilon^4}{2R_{n+1}}\le R_{n+1}-R_n.
\end{equation}
\end{enumerate}
\end{lem}

\begin{proof}
\begin{enumerate}
\item We first notice that $R_{n+1}\le R_n$ is equivalent to $R_n\ge \frac{\nu}{2}\epsilon^2+\nu^{-1}$. Also $R_{n+1}\ge R_n$ is equivalent to $R_n\le \frac{\nu}{2}\epsilon^2+\nu^{-1}$. These facts imply the first assertion. 

Assume that $R_k> \frac{\nu}{2}\epsilon^2+\nu^{-1}$ for some $k$. Since $(R_k-\nu\epsilon^2)^2> \left(\nu^{-1}-\frac{\nu}{2}\epsilon^2\right)^2$ for $\epsilon\le \sqrt{2}\nu^{-1}$, we have
\begin{align}
R_{k+1}&=\sqrt{(R_k-\nu\epsilon^2)^2+2\epsilon^2}>\sqrt{\left(\nu^{-1}-\frac{\nu}{2}\epsilon^2\right)^2+2\epsilon^2}=\frac{\nu}{2}\epsilon^2+\nu^{-1}. \nonumber
\end{align}
Hence, if $R_0> \frac{\nu}{2}\epsilon^2+\nu^{-1}$, we have  by induction that $R_n> \frac{\nu}{2}\epsilon^2+\nu^{-1}$ for all $n$. The second assertion follows from this. 

Assume that $\nu\epsilon^2\le R_k< \frac{\nu}{2}\epsilon^2+\nu^{-1}$. Since $(R_k-\nu\epsilon^2)^2< \left(\nu^{-1}-\frac{\nu}{2}\epsilon^2\right)^2$ for $\epsilon\in\left(0,\sqrt{2}\nu^{-1}\right)$, we get
$R_{k+1}< \frac{\nu}{2}\epsilon^2+\nu^{-1}.$
In the same way, $\nu\epsilon^2 \le R_0< \frac{\nu}{2}\epsilon^2+\nu^{-1}$ gives $R_n< \frac{\nu}{2}\epsilon^2+\nu^{-1}$. Therefore the third assertion is obtained.
\item Since $\{R_n\}$ is monotone and bounded, it is convergent. Its limit value is given by taking limit for both sides of \eqref{seq:2} and solving the limit equation. 
\item The inequality \eqref{ineq:shortcut} is given by the following computation.\[R_{n+1}-R_n=\frac{R_{n+1}^2-R_n^2}{R_{n+1}+R_n}\le \frac{R_{n+1}^2-R_n^2}{2R_n}= \frac{\epsilon^2}{R_n}-\nu\epsilon^2+\frac{\nu^2\epsilon^4}{2R_n}.\]Similarly \eqref{ineq:shortcut2} is obtained by
\end{enumerate}
\[R_{n+1}-R_n=\frac{R_{n+1}^2-R_n^2}{R_{n+1}+R_n}\ge \frac{R_{n+1}^2-R_n^2}{2R_{n+1}}\ge \frac{\epsilon^2}{R_{n+1}}-\nu\epsilon^2+\frac{\nu^2\epsilon^4}{2R_{n+1}}.\] 
\end{proof}
It is sometimes convenient to describe the behavior of the trajectory by an operator. For fixed $\nu>0$, we define the operator $T_{h}:\mathbb{R}\to\mathbb{R}$ as
\[T_{h}(R):=\sqrt{(R-\nu h)^2+2h}.\]
We denote $n$ times composition of $T_{h}$ by $T_{h}^n$. The solution $\{R_n\}$ to \eqref{seq:2} with $R_0=R$ is described by $R_n=T_{\epsilon^2}^n(R)$.

\begin{lem}
\label{lem:monotonicity}
\begin{enumerate}
\item $T_{h}(R)< T_{h}(R^{\prime})$, provided $\nu h\le R< R^{\prime}$.
\item $T_{h}(R)>T_{h/2}^2(R)$.
\end{enumerate}
\end{lem}
\begin{proof}
The proofs are done by direct computation, so are omitted.
\end{proof}
We also prepare a notation that represents the time to pass through the set $B_b(z)\setminus B_a(z)$ when \sente takes a \zcenp{$z$}. For $a,b\ge0$ and $\epsilon>0$ satisfying $a\le b\le \nu^{-1}+\frac{\nu}{2}\epsilon^2$, we define
\begin{equation}
\label{eqdef:exittime}
t_{\epsilon}(a,b):=\epsilon^2\lvert\{n\in \mathbb{N} \mid T_{\epsilon^2}^n(a)<b\}\rvert,
\end{equation}
where we denote the set $\{0,1,2, \cdots\}$ by $\mathbb{N}$.
\begin{lem}
\label{lem:0815}
Let $t>0$ and $0<R<\nu^{-1}$. Define $N=\lceil t\epsilon^{-2}\rceil$. Then $T_{\epsilon^2}^N(R)<\nu^{-1}$ for sufficiently small $\epsilon>0$.
\end{lem}
\begin{proof}
We prove that 
\[t_{\epsilon}\left(\nu^{-1}-\epsilon,\nu^{-1}\right)\ge C \log_2 \epsilon^{-1}\]
for some constant $C>0$.

First let us explain the property of $t_{\epsilon}(a,b)$. Let $0<a\le b\le c\le \nu^{-1}+\frac{\nu}{2}\epsilon^2$. By the definition of $t_{\epsilon}$, we see
\begin{equation}
t_{\epsilon}(a,b)\le t_{\epsilon}(a,c). \nonumber
\end{equation}
By Lemma \ref{lem:monotonicity} 1, we have
\begin{equation}
t_{\epsilon}(b,c)\le t_{\epsilon}(a,c). \nonumber
\end{equation}
Let $n^{\ast}=t_{\epsilon}(a,b)\epsilon^{-2}$ and $d=T_{\epsilon^2}^{n^{\ast}-1}(a)$. Then we have
\[t_{\epsilon}(a,c)=t_{\epsilon}(a,b)+t_{\epsilon}(d,c)-\epsilon^2.\]
Since $d<b$, we deduce that
\begin{equation}
t_{\epsilon}(a,c)\ge t_{\epsilon}(a,b)+t_{\epsilon}(b,c)-\epsilon^2.
\label{masuzu}
\end{equation}

We next estimate $t_{\epsilon}\left(\nu^{-1}-\epsilon,\nu^{-1}\right)$. When $\nu^{-1}-2^{-m}\epsilon \le R_n\le \nu^{-1}-2^{-m-1}\epsilon$ for $m,n\in\mathbb{N}$, the inequality \eqref{ineq:shortcut} implies
\begin{align}
R_{n+1}-R_n\le \frac{\epsilon^2}{\nu^{-1}-2^{-m}\epsilon}-\nu\epsilon^2+\frac{\nu^2\epsilon^4}{2\nu^{-1}-2^{-m+1}\epsilon}=\frac{2^{-m}\nu \epsilon^3+\nu^2\epsilon^4}{2\nu^{-1}-2^{-m+1}\epsilon}, \nonumber
\end{align}
and in particular, if $\epsilon\le \nu^{-1}$,
\begin{equation}
R_{n+1}-R_n\le\frac{2^{-m}\nu\epsilon^3+\nu^2\epsilon^4}{2\nu^{-1}-2^{-m+1}\epsilon} \le \frac{2^{-m}\nu\epsilon^3+\nu^2\epsilon^4}{\nu^{-1}}=\nu^2\epsilon^3\left(2^{-m}+\nu\epsilon\right). \nonumber
\end{equation}
Thus $t_{\epsilon}\left(\nu^{-1}-2^{-m}\epsilon,\nu^{-1}-2^{-m-1}\epsilon\right)$ can be compared to the exit time for sequences with the constant speed of $\nu^2\epsilon^3\left(2^{-m}+\nu\epsilon\right)$ per round. 
Then we have
\begin{align*}
&t_{\epsilon}\left(\nu^{-1}-2^{-m}\epsilon,\nu^{-1}-2^{-m-1}\epsilon\right) \\
&\ge\left(\left(\nu^{-1}-2^{-m-1}\epsilon\right)-\left(\nu^{-1}-2^{-m}\epsilon\right)\right)\frac{\epsilon^2}{\nu^2\epsilon^3\left(2^{-m}+\nu\epsilon\right)}=\frac{1}{2\nu^2+2^{m+1} \nu^3\epsilon}. 
\end{align*}
In particular, if $2^m\le \epsilon^{-1}$, we get
\begin{equation}
\frac{1}{2\nu^2+2^{m+1} \nu^3\epsilon}\ge \frac{1}{2\nu^2+2\nu^3}=\frac{1}{2\nu^2(\nu+1)}.
\nonumber
\end{equation}
Let $m^{\ast}$ be the minimal integer $m$ satisfying $2^m>\epsilon^{-1}$. Using \eqref{masuzu}, we obtain
\begin{align}
t_{\epsilon}\left(\nu^{-1}-\epsilon,\nu^{-1}\right)&\ge \sum_{k=0}^{m^{\ast}-1}\left(t_{\epsilon}\left(\nu^{-1}-2^{-k}\epsilon,\nu^{-1}-2^{-k-1}\epsilon\right)-\epsilon^2\right) \nonumber \\
&\ge \sum_{k=0}^{m^{\ast}-1}\frac{1}{3\nu^2(\nu+1)}=\frac{\log_2 \epsilon^{-1}}{3\nu^2(\nu+1)}, \nonumber
\end{align}
which is the conclusion.
\end{proof}

\begin{lem}
\label{lem:1109}
For $\nu>0$ and $\nu^{-1}\ge\delta>0$, there exist $\epsilon_0>0$ and $M>0$ such that
\[t_{\epsilon}(a,\nu^{-1}-\delta)\le M\]
for all $0<\epsilon<\epsilon_0$ and $0\le a\le \nu^{-1}-\delta$.
\end{lem}
\begin{proof}
Let $n_{\epsilon}=t_{\epsilon}(a,\nu^{-1}-\delta)\epsilon^{-2}$. Namely $T_{\epsilon^2}^{n_{\epsilon}}(a)<\nu^{-1}-\delta$ and $T_{\epsilon^2}^{n_{\epsilon}+1}(a)\ge\nu^{-1}-\delta$ are satisfied. Since $T_{\epsilon^2}(R)-R$ is monotonically decreasing with respect to $R$ and $T_{\epsilon^2}^{n_{\epsilon}+1}(a)\le \nu^{-1}-\delta/2$ for sufficiently small $\epsilon>0$, we have by \eqref{ineq:shortcut2} in Lemma \ref{lem:series2}
\begin{align*}
T_{\epsilon^2}^{n+1}(a)-T_{\epsilon^2}^{n}(a)\ge T_{\epsilon^2}^{n_{\epsilon}+1}(a)-T_{\epsilon^2}^{n_{\epsilon}}(a)\ge \frac{\epsilon^2}{T_{\epsilon^2}^{n_{\epsilon}+1}(a)}-\nu\epsilon^2+\frac{\nu^2\epsilon^4}{2T_{\epsilon^2}^{n_{\epsilon}+1}(a)}\ge\frac{\delta\nu^2}{2-\delta\nu}\epsilon^2 
\end{align*}
for $n=0,1,\cdots,n_{\epsilon}$. This means
\[n_{\epsilon}\le (\nu^{-1}-\delta-a)\frac{2-\delta\nu}{\delta\nu^2\epsilon^2}\]
and hence
\[t_{\epsilon}(a,\nu^{-1}-\delta)\le (\nu^{-1}-\delta-a)\frac{2-\delta\nu}{\delta\nu^2\epsilon^2}\epsilon^2\le(\nu^{-1}-\delta)\frac{2-\delta\nu}{\delta\nu^2},\]
which is the conclusion.
\end{proof}
To study the asymptotic behavior of solutions, the following types of strategies of the game are also useful.
\begin{dfn}[Push by moving circle]
For a sequence $\{z_n\}\subset\mathbb{R}^2$ such that $\lvert z_{n+1}-z_n\rvert\le C$ for some constant $C>0$, a strategy for \sente (resp. for \gote) to keep taking a \zcenp{$z_n$} at each round $n$ is called a {\it push by moving circle strategy} by \sente (resp. by \gote). 
\end{dfn}
\begin{lem}[Properties of push by moving circle strategies]
\label{lem:pmds}
Let $\delta>0$. Let $\epsilon>0$ be sufficiently small. i.e., we assume that $0<\epsilon<\epsilon_0(\nu,\delta)$ for some $\epsilon_0(\nu,\delta)>0$. 
\begin{enumerate}
\item If either \sente or \gote takes a push by moving circle strategy with $C=\frac{\delta}{2}\nu^2\epsilon^2$ and $x_0\in \overline{B(z_0, \nu^{-1}-\delta)}^c$, then $x_n\in \overline{B(z_n, \nu^{-1}-\delta)}^c$ for all round $n$.
\item If \sente takes a push by moving circle strategy with $C=\frac{\nu^2\delta}{2(1+\nu\delta)}\epsilon^2$ and $x_0\in \overline{B(z_0, \nu^{-1}+\delta)}$. then $x_n\in \overline{B(z_n, \nu^{-1}+\delta)}$ for all round $n$.
\end{enumerate}
\end{lem}
\begin{proof}
\begin{enumerate}
\item We notice that $\lvert x_{n+1}-z_n\rvert\ge T_{\epsilon^2}(\lvert x_n-z_n\rvert)$ whichever player takes the push by moving circle strategy. We prove that $x_n\in \overline{B(z_n, \nu^{-1}-\delta)}^c$ implies $x_{n+1}\in \overline{B(z_{n+1}, \nu^{-1}-\delta)}^c$. If $\lvert x_n-z_n\rvert\ge \nu^{-1}$, we see from Lemma \ref{lem:series2} that $\lvert x_{n+1}-z_n\rvert\ge \nu^{-1}$ and hence $\lvert x_{n+1}-z_{n+1}\rvert\ge \nu^{-1}-\frac{\delta}{2}\nu^2\epsilon^2\ge \nu^{-1}-\delta$ for sufficiently small $\epsilon$. Thus we asuume $\lvert x_n-z_n\rvert< \nu^{-1}$ hereafter.

It suffices to show $\vert x_{n+1}-z_n\vert> \nu^{-1}-\delta+\frac{\delta}{2}\nu^2\epsilon^2$. Indeed this inequality is obtained by 
\begin{equation}
\nonumber
\lvert x_{n+1}-z_n\rvert\ge T_{\epsilon^2}(\lvert x_n-z_n\rvert)> T_{\epsilon^2}(\nu^{-1}-\delta)\ge \nu^{-1}-\delta+\frac{\delta}{2}\nu^2\epsilon^2.
\end{equation}
The second inequality is derived from the monotonicity property of $T_{\epsilon^2}$ stated in Lemma \ref{lem:monotonicity}. The third inequality is computed by \eqref{ineq:shortcut2} in Lemma \ref{lem:series2}. We notice that $T_{\epsilon^2}(\nu^{-1}-\delta)\le \nu^{-1}-\delta/2$ for sufficiently small $\epsilon$. Letting $R_n=\nu^{-1}-\delta$ and hence $R_{n+1}=T_{\epsilon^2}(\nu^{-1}-\delta)$, we indeed have
\begin{align*}
R_{n+1}-R_n\ge \frac{\epsilon^2}{R_{n+1}}-\nu\epsilon^2+\frac{\nu^2\epsilon^4}{2R_{n+1}}\ge \frac{\epsilon^2}{\nu^{-1}-\delta/2}-\nu\epsilon^2+\frac{\nu^2\epsilon^4}{2\nu^{-1}-\delta}\ge \frac{\delta}{2}\nu^2\epsilon^2.
\end{align*}
Thus the proof is complete.
\item We prove that $x_n\in \overline{B(z_n, \nu^{-1}+\delta)}$ implies $x_{n+1}\in \overline{B(z_{n+1}, \nu^{-1}+\delta)}$. If $\lvert x_n-z_n\rvert<\nu\epsilon^2$, it is clear that $\lvert x_{n+1}-z_{n+1}\rvert\le \nu^{-1}+\delta$ for sufficiently small $\epsilon$. If not, we have
\begin{equation}
\nonumber
\lvert x_{n+1}-z_{n+1}\rvert\le T_{\epsilon^2}(\nu^{-1}+\delta)+\frac{\nu^2\delta}{2(1+\nu\delta)}\epsilon^2.
\end{equation}
Letting $R_n=\nu^{-1}+\delta$ and hence $R_{n+1}=T_{\epsilon^2}(\nu^{-1}+\delta)$, we have from \eqref{ineq:shortcut} in Lemma \ref{lem:series2}
\[R_n-R_{n+1}\ge -\frac{\epsilon^2}{R_{n}}+\nu\epsilon^2-\frac{\nu^2\epsilon^4}{2R_{n}}=\frac{\nu^2\delta}{1+\nu\delta}\epsilon^2+o(\epsilon^2).\]
Therefore we obtain
\[\lvert x_{n+1}-z_{n+1}\rvert\le \nu^{-1}+\delta\]
for sufficiently small $\epsilon$.
\end{enumerate}
\end{proof}

We set
\[\mathcal{L}_{\nu}:=\{\gamma([0,1]) \mid \gamma\in C^2([0,1];\mathbb{R}^2), \ \mbox{$\gamma$ is regular and} \ \lvert\kappa_{\gamma(t)}\rvert\le\nu \ \mbox{for all} \ t\in(0,1)\},\]
where we denote the curvature of $\gamma([0,1])$ at $\gamma(t)$ by $\kappa_{\gamma(t)}$.

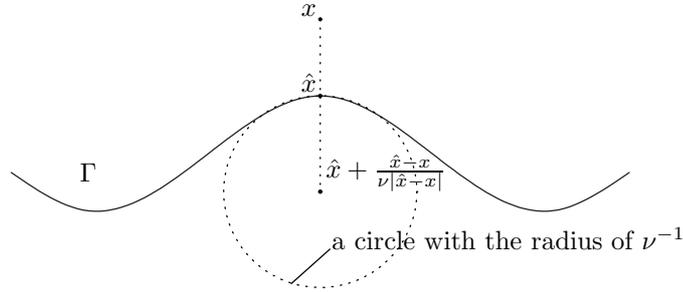
\begin{figure}[h]
\begin{center}
{\unitlength 0.1in%
\begin{picture}(32.0000,15.5000)(2.0000,-22.0000)%
%
\special{pn 8}%
\special{pa 200 1600}%
\special{pa 229 1620}%
\special{pa 257 1639}%
\special{pa 286 1659}%
\special{pa 314 1677}%
\special{pa 343 1696}%
\special{pa 372 1713}%
\special{pa 400 1729}%
\special{pa 429 1744}%
\special{pa 458 1758}%
\special{pa 486 1770}%
\special{pa 515 1780}%
\special{pa 543 1789}%
\special{pa 572 1796}%
\special{pa 601 1800}%
\special{pa 629 1802}%
\special{pa 658 1802}%
\special{pa 687 1800}%
\special{pa 715 1795}%
\special{pa 744 1789}%
\special{pa 772 1781}%
\special{pa 801 1771}%
\special{pa 830 1760}%
\special{pa 858 1747}%
\special{pa 887 1733}%
\special{pa 916 1717}%
\special{pa 944 1700}%
\special{pa 973 1683}%
\special{pa 1001 1664}%
\special{pa 1059 1624}%
\special{pa 1087 1603}%
\special{pa 1116 1582}%
\special{pa 1145 1560}%
\special{pa 1173 1538}%
\special{pa 1202 1516}%
\special{pa 1230 1494}%
\special{pa 1259 1471}%
\special{pa 1288 1449}%
\special{pa 1316 1427}%
\special{pa 1345 1406}%
\special{pa 1373 1385}%
\special{pa 1402 1364}%
\special{pa 1431 1345}%
\special{pa 1459 1326}%
\special{pa 1517 1290}%
\special{pa 1545 1275}%
\special{pa 1574 1260}%
\special{pa 1602 1247}%
\special{pa 1631 1235}%
\special{pa 1660 1224}%
\special{pa 1688 1216}%
\special{pa 1717 1209}%
\special{pa 1746 1204}%
\special{pa 1774 1201}%
\special{pa 1803 1200}%
\special{pa 1831 1201}%
\special{pa 1860 1205}%
\special{pa 1889 1210}%
\special{pa 1917 1217}%
\special{pa 1946 1226}%
\special{pa 1975 1237}%
\special{pa 2003 1249}%
\special{pa 2032 1263}%
\special{pa 2060 1278}%
\special{pa 2089 1294}%
\special{pa 2118 1311}%
\special{pa 2146 1329}%
\special{pa 2175 1348}%
\special{pa 2204 1368}%
\special{pa 2232 1389}%
\special{pa 2261 1410}%
\special{pa 2289 1432}%
\special{pa 2347 1476}%
\special{pa 2375 1498}%
\special{pa 2404 1520}%
\special{pa 2432 1543}%
\special{pa 2461 1565}%
\special{pa 2490 1586}%
\special{pa 2518 1608}%
\special{pa 2576 1648}%
\special{pa 2604 1668}%
\special{pa 2633 1686}%
\special{pa 2661 1704}%
\special{pa 2719 1736}%
\special{pa 2747 1750}%
\special{pa 2776 1762}%
\special{pa 2805 1773}%
\special{pa 2833 1783}%
\special{pa 2862 1790}%
\special{pa 2890 1796}%
\special{pa 2919 1800}%
\special{pa 2948 1802}%
\special{pa 2976 1802}%
\special{pa 3005 1799}%
\special{pa 3034 1795}%
\special{pa 3062 1788}%
\special{pa 3091 1779}%
\special{pa 3119 1768}%
\special{pa 3148 1755}%
\special{pa 3177 1741}%
\special{pa 3205 1726}%
\special{pa 3263 1692}%
\special{pa 3291 1674}%
\special{pa 3320 1655}%
\special{pa 3348 1636}%
\special{pa 3377 1616}%
\special{pa 3400 1600}%
\special{fp}%
%
\special{pn 4}%
\special{sh 1}%
\special{ar 1800 800 8 8 0 6.2831853}%
\special{sh 1}%
\special{ar 1800 1200 8 8 0 6.2831853}%
\special{sh 1}%
\special{ar 1800 1200 8 8 0 6.2831853}%
\put(17.8000,-7.8000){\makebox(0,0)[rb]{$x$}}%
\put(17.8000,-11.8000){\makebox(0,0)[rb]{$\hat{x}$}}%
\put(6.0000,-16.0000){\makebox(0,0){$\Gamma$}}%
%
\special{pn 8}%
\special{pn 8}%
\special{pa 2300 1700}%
\special{pa 2300 1708}%
\special{fp}%
\special{pa 2298 1745}%
\special{pa 2297 1753}%
\special{fp}%
\special{pa 2292 1789}%
\special{pa 2291 1797}%
\special{fp}%
\special{pa 2282 1833}%
\special{pa 2280 1841}%
\special{fp}%
\special{pa 2268 1876}%
\special{pa 2265 1883}%
\special{fp}%
\special{pa 2250 1917}%
\special{pa 2247 1924}%
\special{fp}%
\special{pa 2229 1957}%
\special{pa 2225 1963}%
\special{fp}%
\special{pa 2205 1994}%
\special{pa 2200 2000}%
\special{fp}%
\special{pa 2177 2029}%
\special{pa 2171 2035}%
\special{fp}%
\special{pa 2145 2062}%
\special{pa 2139 2067}%
\special{fp}%
\special{pa 2112 2091}%
\special{pa 2105 2096}%
\special{fp}%
\special{pa 2075 2118}%
\special{pa 2069 2122}%
\special{fp}%
\special{pa 2037 2140}%
\special{pa 2030 2144}%
\special{fp}%
\special{pa 1997 2160}%
\special{pa 1989 2163}%
\special{fp}%
\special{pa 1954 2175}%
\special{pa 1947 2178}%
\special{fp}%
\special{pa 1911 2187}%
\special{pa 1904 2189}%
\special{fp}%
\special{pa 1867 2196}%
\special{pa 1859 2196}%
\special{fp}%
\special{pa 1823 2200}%
\special{pa 1815 2200}%
\special{fp}%
\special{pa 1778 2199}%
\special{pa 1770 2199}%
\special{fp}%
\special{pa 1733 2195}%
\special{pa 1725 2194}%
\special{fp}%
\special{pa 1689 2188}%
\special{pa 1681 2186}%
\special{fp}%
\special{pa 1646 2176}%
\special{pa 1638 2173}%
\special{fp}%
\special{pa 1604 2160}%
\special{pa 1596 2157}%
\special{fp}%
\special{pa 1563 2140}%
\special{pa 1557 2137}%
\special{fp}%
\special{pa 1525 2117}%
\special{pa 1518 2113}%
\special{fp}%
\special{pa 1488 2091}%
\special{pa 1482 2086}%
\special{fp}%
\special{pa 1455 2062}%
\special{pa 1449 2056}%
\special{fp}%
\special{pa 1424 2029}%
\special{pa 1418 2023}%
\special{fp}%
\special{pa 1395 1994}%
\special{pa 1391 1988}%
\special{fp}%
\special{pa 1371 1956}%
\special{pa 1367 1950}%
\special{fp}%
\special{pa 1349 1917}%
\special{pa 1346 1910}%
\special{fp}%
\special{pa 1332 1876}%
\special{pa 1329 1868}%
\special{fp}%
\special{pa 1318 1833}%
\special{pa 1316 1825}%
\special{fp}%
\special{pa 1308 1789}%
\special{pa 1307 1781}%
\special{fp}%
\special{pa 1302 1745}%
\special{pa 1301 1737}%
\special{fp}%
\special{pa 1300 1700}%
\special{pa 1300 1692}%
\special{fp}%
\special{pa 1302 1655}%
\special{pa 1303 1647}%
\special{fp}%
\special{pa 1308 1611}%
\special{pa 1310 1603}%
\special{fp}%
\special{pa 1318 1567}%
\special{pa 1320 1559}%
\special{fp}%
\special{pa 1332 1524}%
\special{pa 1335 1517}%
\special{fp}%
\special{pa 1349 1483}%
\special{pa 1353 1476}%
\special{fp}%
\special{pa 1371 1443}%
\special{pa 1375 1436}%
\special{fp}%
\special{pa 1396 1406}%
\special{pa 1400 1399}%
\special{fp}%
\special{pa 1423 1371}%
\special{pa 1429 1365}%
\special{fp}%
\special{pa 1455 1338}%
\special{pa 1461 1333}%
\special{fp}%
\special{pa 1489 1308}%
\special{pa 1495 1304}%
\special{fp}%
\special{pa 1525 1282}%
\special{pa 1531 1278}%
\special{fp}%
\special{pa 1564 1259}%
\special{pa 1571 1256}%
\special{fp}%
\special{pa 1604 1240}%
\special{pa 1612 1237}%
\special{fp}%
\special{pa 1646 1224}%
\special{pa 1654 1222}%
\special{fp}%
\special{pa 1689 1212}%
\special{pa 1697 1211}%
\special{fp}%
\special{pa 1733 1205}%
\special{pa 1741 1203}%
\special{fp}%
\special{pa 1778 1201}%
\special{pa 1786 1200}%
\special{fp}%
\special{pa 1823 1200}%
\special{pa 1831 1201}%
\special{fp}%
\special{pa 1867 1204}%
\special{pa 1875 1206}%
\special{fp}%
\special{pa 1911 1213}%
\special{pa 1919 1214}%
\special{fp}%
\special{pa 1954 1225}%
\special{pa 1962 1227}%
\special{fp}%
\special{pa 1997 1240}%
\special{pa 2004 1244}%
\special{fp}%
\special{pa 2037 1260}%
\special{pa 2044 1264}%
\special{fp}%
\special{pa 2075 1282}%
\special{pa 2082 1287}%
\special{fp}%
\special{pa 2112 1309}%
\special{pa 2118 1314}%
\special{fp}%
\special{pa 2145 1339}%
\special{pa 2151 1344}%
\special{fp}%
\special{pa 2176 1371}%
\special{pa 2182 1377}%
\special{fp}%
\special{pa 2205 1406}%
\special{pa 2210 1413}%
\special{fp}%
\special{pa 2229 1444}%
\special{pa 2233 1451}%
\special{fp}%
\special{pa 2250 1483}%
\special{pa 2254 1490}%
\special{fp}%
\special{pa 2268 1524}%
\special{pa 2271 1532}%
\special{fp}%
\special{pa 2282 1567}%
\special{pa 2284 1575}%
\special{fp}%
\special{pa 2292 1611}%
\special{pa 2293 1618}%
\special{fp}%
\special{pa 2298 1655}%
\special{pa 2299 1663}%
\special{fp}%
\special{pa 2300 1700}%
\special{pa 2300 1700}%
\special{fp}%
%
\special{pn 4}%
\special{sh 1}%
\special{ar 1800 1700 8 8 0 6.2831853}%
\special{sh 1}%
\special{ar 1800 1700 8 8 0 6.2831853}%
\put(18.6000,-20.0000){\makebox(0,0)[lb]{a circle with the radius of $\nu^{-1}$}}%
%
\special{pn 8}%
\special{pa 1800 800}%
\special{pa 1800 1700}%
\special{dt 0.045}%
\put(18.3000,-17.0000){\makebox(0,0)[lb]{$\hat{x}+\frac{\hat{x}-x}{\nu\lvert \hat{x}-x\rvert}$}}%
%
\special{pn 8}%
\special{pa 1850 2000}%
\special{pa 1650 2180}%
\special{fp}%
\end{picture}}%
\end{center}
\caption{$\Gamma$ tube strategy}
\label{fig:hosest}
\end{figure}

\begin{dfn}[$\Gamma$ tube strategy]
\label{def:ghose}
Let $\nu>0$ and $\Gamma\in\mathcal{L}_{\nu}$. Let $x\in\mathbb{R}^2$ be the current position of the game. 
Let $\hat{x}$ satisfy $\min_{y\in\Gamma} \lvert x-y\rvert=\lvert x-\hat{x}\rvert$. A \zcenp{$\hat{x}+\frac{\hat{x}-x}{\nu\lvert\hat{x}-x\rvert}$} is called a {\it $\Gamma$ tube strategy} (Figure \ref{fig:hosest}).
\end{dfn}

\begin{lem}[Property of $\Gamma$ tube strategies]
\label{lem:ghose}
Let $\Gamma=\gamma([0,1])\in\mathcal{L}_{\nu}$. Let $\delta>0$ and $\epsilon\ll\delta$. If \sente takes a $\Gamma$ tube strategy at round $n$ and $x_n\in B_{\delta}(\Gamma)\setminus (B_{\delta}(\gamma(0))\cup B_{\delta}(\gamma(1)))$, then $x_{n+1}\in B_{\delta}(\Gamma)$.
\end{lem}
\begin{proof}
It is clear that the statement holds if $x_n\in B_{\delta/2}(\Gamma)$. Thus, there is no loss of generality to assume that $\hat{x}_n=(0,0)$ and $x_n\in \{(0,q) \mid  \delta/2\le q<\delta\}$, where $\hat{x}_n$ is a minimizer taken in Definition \ref{def:ghose}. There are a graph $f:\mathbb{R}\to\mathbb{R}$ and $\delta_0>0$ such that $\hat{\Gamma}:=\{(p,q) \mid q=f(p), -\delta_0<p<\delta_0\}\subset\Gamma$. Let $C=\partial B_{\nu^{-1}}((0,-\nu^{-1}))$. If \sente takes a $\Gamma$ tube strategy, we see that $dist(x_{n+1},C)<dist(x_n,C)<\delta$. Since $f(p)\ge -\nu^{-1}+\sqrt{\nu^{-2}-p^2}$ for $-\delta_0<p<\delta_0$, we have
\[dist(x_{n+1},\Gamma)\le dist(x_{n+1},\hat{\Gamma})\le dist(x_{n+1},C)<dist(x_n,C)<\delta,\]
which is the conclusion.
\end{proof}
\subsection{Examples}

We first observe the behavior of the solution to \eqref{eq:bobstacle} by considering two examples of $O_{-}$. In both of the problems we assume
\begin{itemize}
\item[(A1)] $D_0\subset B_{R}(z)$ for some $z\in\mathbb{R}^2$ and $R<\nu^{-1}$.
\end{itemize}
Set 
\[A:=\bigcap \left\{\overline{B_{\nu^{-1}}(z)} \mid O_{-}\subset \overline{B_{\nu^{-1}}(z)},  \ z\in\mathbb{R}^2\right\}\]
for $O_{-}\subset \mathbb{R}^2$. Under the assumption (A1), this $A$ is a major candidate of the asymptotic shape. Indeed we can show at least that the asymptotic shape is bounded by $A$ from above for general $O_{-}$ satisfying (A1).

\begin{lem}
\label{lem:0924}
Let $O\subset \mathbb{R}^d$ and $r>0$. Then the set $\{y\in\mathbb{R}^d \mid O\subset \overline{B_{r}(y)}\}$ is convex.
\end{lem}
\begin{proof}
Let $y_1,y_2\in\mathbb{R}^d$ satisfy $O\subset \overline{B_{r}(y_1)}$ and $O\subset \overline{B_{r}(y_2)}$. Then, for $x\in O$, we have $\lvert y_i-x\rvert\le r$ ($i=1,2$). This implies that 
\[\left\lvert\frac{y_1+y_2}{2}-x\right\rvert=\frac{\lvert y_1-x+y_2-x\rvert}{2}\le r\]
for all $x\in O$ and the lemma follows.
\end{proof}
We hereafter take $u_0$ and $\Psi_{-}$ as \eqref{good_ini} and \eqref{good_obs} respectively. We set $O_{\delta_1}:=\{x\in\mathbb{R}^2 \mid \Psi_{-}(x)>-\delta_1\}$ and $C_{\delta_1,\delta_2}:=\{z\in\mathbb{R}^2 \mid O_{\delta_1}\subset \overline{B(z,\nu^{-1}-\delta_2)}\}$ for $\delta_1,\delta_2>0$.
\begin{lem}
\label{lem:1011}
Assume $\mathrm{(A1)}$. For $x\in A^c$, there exist $\delta_1, \delta_2>0$ and $\hat{z}\in\mathbb{R}^2$ such that $\hat{z}\in C_{\delta_1,\delta_2}$ and $x\in \overline{B_{\nu^{-1}}(\hat{z})}^c$.
\end{lem}
\begin{proof}
{\bf Step 1.}
We first show that $B_{\delta_1+\delta_2}(z)\subset C_{0,0}$ implies $z\in C_{\delta_1,\delta_2}$. Indeed $B_{\delta_1+\delta_2}(z)\subset C_{0,0}$ implies
\[O_{-}\subset \bigcap\{\overline{B_{\nu^{-1}}(\tilde{z})} \mid \tilde{z}\in B_{\delta_1+\delta_2}(z)\}=\overline{B(z,\nu^{-1}-(\delta_1+\delta_2))}.\]
By taking $\delta_1$ neighborhood of both sides, we have
\[O_{\delta_1}\subset \overline{B(z,\nu^{-1}-\delta_2)},\]
which means $z\in C_{\delta_1,\delta_2}$.

{\bf Step 2.}
Let $C^x_{0,0}:=\{z\in\mathbb{R} \mid O_{-}\subset\overline{B_{\nu^{-1}}(z)}, x\in \overline{B_{\nu^{-1}}(z)}^c\}$. We next show $(C_{0,0}^x)^{int}\neq \emptyset$. By the assumption (A1), we see that $\overline{B(z,\nu^{-1}-R)}\subset C_{0,0}$ for some $z\in\mathbb{R}^2$. We also notice that $C^x_{0,0}=C_{0,0}\setminus \overline{B_{\nu^{-1}}(x)}$. If $\overline{B(z,\nu^{-1}-R)}\setminus \overline{B_{\nu^{-1}}(x)}\neq \emptyset$, then it has the interior and so does $C^x_{0,0}$. Since $x\in A^c$, we have $C^x_{0,0}\neq\emptyset$ and let $\tilde{z}\in C^x_{0,0}$. Since $C_{0,0}$ is convex (Lemma \ref{lem:0924}), we have $Co\left(\{\tilde{z}\}\cup \overline{B(z,\nu^{-1}-R)}\right)\subset C_{0,0}$. Hence, even if $\overline{B(z,\nu^{-1}-R)}\subset \overline{B_{\nu^{-1}}(x)}$, the set $Co\left(\{\tilde{z}\}\cup \overline{B(z,\nu^{-1}-R)}\right)\setminus \overline{B_{\nu^{-1}}(x)}$ has the interior and so does $C^x_{0,0}$.

Therefore, for small $\delta_1>0$ and $\delta_2>0$, there exists $\hat{z}\in C_{\delta_1,\delta_2}$ and $x\in \overline{B_{\nu^{-1}}(\hat{z})}^c$.
\end{proof}
\begin{lem}
\label{lem:gote1011}
Assume $\mathrm{(A1)}$. Then
\[\varlimsup_{t\to\infty}E_t\subset A.\]
\end{lem}
\begin{proof}
We give appropriate strategies of \gote as in the proof of Lemma \ref{lem:elgeom}. For $x\in A^c$, there exist $\delta_1, \delta_2>0$ and $z_0\in\mathbb{R}^2$ such that $R<\nu^{-1}-\delta_2$, $z_0\in C_{\delta_1,\delta_2}$ and $x\notin \overline{B(z_0,\nu^{-1}-\delta_2)}$ (Lemma \ref{lem:1011}). We define the sequence $\{z_n\}$ by
\[z_n:=z_0+\min\left\{\frac{\delta_2}{2}n\nu^2\epsilon^2, \lvert z-z_0\rvert\right\}\frac{z-z_0}{\lvert z-z_0\rvert},\]
where $z$ is a point taken in the assumption (A1). \gote's strategy is to take a push by moving circle strategy with this $\{z_n\}$. 

If she does so, then $z_n\in C_{\delta_1,\delta_2}$ for all $n$ because $z_0, z\in C_{\delta_1,\delta_2}$ and $C_{\delta_1,\delta_2}$ is convex (Lemma \ref{lem:0924}). By 1 in Lemma \ref{lem:pmds} we see that if \sente quits the game at round $i$, the game position $x_i$ is not in $O_{\delta_1}$. Thus the stopping cost is at most $-\delta_1$ regardless of $\epsilon$. If \sente does not quit the game, the last position $x_N$ of the game is not in $\overline{B(z,\nu^{-1}-\delta_2)}$ for sufficiently large $t$. Thus the terminal cost is at most $R-\nu^{-1}+\delta_2<0$ for large $t$ regardless of $\epsilon$. Therefore we conclude that for $x\in A^c$, there exists $\tau>0$ such that $u(x,t)<0$ for $t>\tau$.
\end{proof}
The first example of $O_{-}$ is the following:
\[O_{-}:=B_{R^{\prime}}((0,0))\setminus \{(p,q) \mid q\ge \lvert p\rvert\},\]
where $R^{\prime}<\nu^{-1}$ (Figure \ref{fig:pacman}). Notice that \\ $A^{int}=B_{R^{\prime}}((0,0))\setminus \left\{(p,q) \mid q\ge \sqrt{\nu^{-2}-p^2}+R^{\prime}-\sqrt{\nu^{-2}-R^{\prime 2}}\right\}$ for this obstacle, where we denote the interior of $A$ by $A^{int}$.

\begin{figure}[h]
\begin{center}
\input{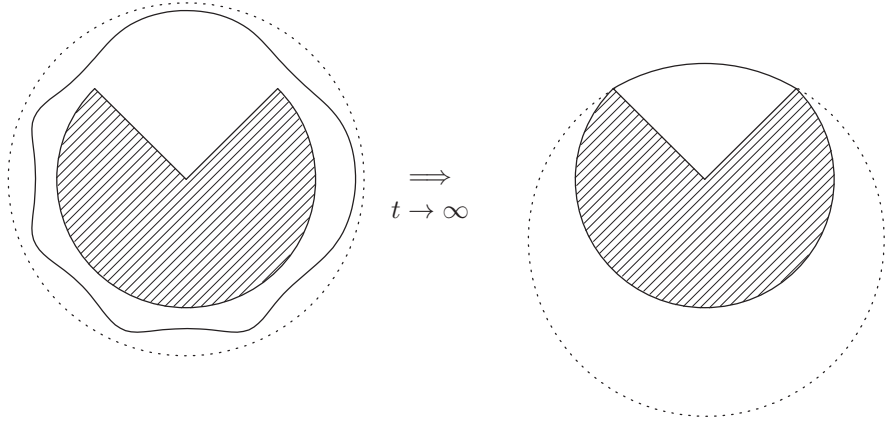}
\end{center}
\caption{Small Pac-Man}
\label{fig:pacman}
\end{figure}

\begin{prop}[Small Pac-Man]
\label{prop:pacman}
Assume $\mathrm{(A1)}$. Then 
\[A^{int}\subset \varliminf_{t\to\infty}D_t \subset \varlimsup_{t\to\infty}D_t\subset A\]
and
\[A^{int}\subset \varliminf_{t\to\infty}E_t \subset \varlimsup_{t\to\infty}E_t\subset A.\]
\end{prop}
\begin{proof}
By Lemma \ref{lem:gote1011}, it suffices to show $A^{int}\subset \varliminf_{t\to\infty}D_t$. As in the proof of Theorem \ref{thm:conv_general}, we do case analysis for the initial game position $x\in A^{int}$ and give an appropriate strategy of \sente. 

\bm{$1)~x\in O_{-}.$}
In this case, it suffices for \sente to quit the game at the first round and gain the stopping cost $\Psi_{-}(x)>0$ as in the proof of Theorem \ref{thm:conv_general}.

\begin{figure}[h]
\begin{center}
{\unitlength 0.1in%
\begin{picture}(27.4000,26.3000)(10.3000,-32.0000)%
%
\special{pn 8}%
\special{ar 2400 1807 1370 1370 5.4977871 3.9269908}%
%
\special{pn 8}%
\special{pa 2400 1807}%
\special{pa 1431 838}%
\special{fp}%
\special{pa 2400 1807}%
\special{pa 3369 838}%
\special{fp}%
%
\special{pn 8}%
\special{pn 8}%
\special{pa 1431 838}%
\special{pa 1438 834}%
\special{fp}%
\special{pa 1470 816}%
\special{pa 1477 812}%
\special{fp}%
\special{pa 1510 794}%
\special{pa 1516 790}%
\special{fp}%
\special{pa 1549 773}%
\special{pa 1557 769}%
\special{fp}%
\special{pa 1590 753}%
\special{pa 1597 750}%
\special{fp}%
\special{pa 1631 734}%
\special{pa 1638 731}%
\special{fp}%
\special{pa 1672 716}%
\special{pa 1680 713}%
\special{fp}%
\special{pa 1714 700}%
\special{pa 1721 696}%
\special{fp}%
\special{pa 1756 683}%
\special{pa 1764 681}%
\special{fp}%
\special{pa 1799 669}%
\special{pa 1806 666}%
\special{fp}%
\special{pa 1841 655}%
\special{pa 1849 653}%
\special{fp}%
\special{pa 1885 642}%
\special{pa 1892 640}%
\special{fp}%
\special{pa 1928 630}%
\special{pa 1936 628}%
\special{fp}%
\special{pa 1972 619}%
\special{pa 1980 617}%
\special{fp}%
\special{pa 2016 610}%
\special{pa 2024 608}%
\special{fp}%
\special{pa 2060 601}%
\special{pa 2068 600}%
\special{fp}%
\special{pa 2104 593}%
\special{pa 2112 592}%
\special{fp}%
\special{pa 2149 587}%
\special{pa 2157 586}%
\special{fp}%
\special{pa 2194 581}%
\special{pa 2202 580}%
\special{fp}%
\special{pa 2239 577}%
\special{pa 2247 576}%
\special{fp}%
\special{pa 2284 574}%
\special{pa 2292 573}%
\special{fp}%
\special{pa 2328 571}%
\special{pa 2337 571}%
\special{fp}%
\special{pa 2374 570}%
\special{pa 2382 570}%
\special{fp}%
\special{pa 2419 570}%
\special{pa 2427 570}%
\special{fp}%
\special{pa 2464 571}%
\special{pa 2472 571}%
\special{fp}%
\special{pa 2509 573}%
\special{pa 2517 574}%
\special{fp}%
\special{pa 2554 577}%
\special{pa 2562 577}%
\special{fp}%
\special{pa 2598 580}%
\special{pa 2606 581}%
\special{fp}%
\special{pa 2643 585}%
\special{pa 2651 586}%
\special{fp}%
\special{pa 2688 592}%
\special{pa 2696 593}%
\special{fp}%
\special{pa 2732 599}%
\special{pa 2740 601}%
\special{fp}%
\special{pa 2776 608}%
\special{pa 2784 610}%
\special{fp}%
\special{pa 2820 618}%
\special{pa 2828 620}%
\special{fp}%
\special{pa 2864 628}%
\special{pa 2872 630}%
\special{fp}%
\special{pa 2908 640}%
\special{pa 2915 642}%
\special{fp}%
\special{pa 2951 653}%
\special{pa 2959 655}%
\special{fp}%
\special{pa 2994 666}%
\special{pa 3001 669}%
\special{fp}%
\special{pa 3037 681}%
\special{pa 3044 683}%
\special{fp}%
\special{pa 3079 697}%
\special{pa 3086 699}%
\special{fp}%
\special{pa 3120 714}%
\special{pa 3128 717}%
\special{fp}%
\special{pa 3162 731}%
\special{pa 3169 734}%
\special{fp}%
\special{pa 3203 750}%
\special{pa 3210 753}%
\special{fp}%
\special{pa 3244 770}%
\special{pa 3251 773}%
\special{fp}%
\special{pa 3284 790}%
\special{pa 3291 794}%
\special{fp}%
\special{pa 3323 812}%
\special{pa 3330 816}%
\special{fp}%
\special{pa 3362 835}%
\special{pa 3369 838}%
\special{fp}%
%
\special{pn 8}%
\special{pn 8}%
\special{pa 1371 1085}%
\special{pa 1378 1081}%
\special{fp}%
\special{pa 1410 1062}%
\special{pa 1417 1058}%
\special{fp}%
\special{pa 1449 1041}%
\special{pa 1456 1037}%
\special{fp}%
\special{pa 1489 1020}%
\special{pa 1496 1016}%
\special{fp}%
\special{pa 1529 1000}%
\special{pa 1536 996}%
\special{fp}%
\special{pa 1569 981}%
\special{pa 1577 977}%
\special{fp}%
\special{pa 1611 963}%
\special{pa 1618 960}%
\special{fp}%
\special{pa 1652 945}%
\special{pa 1659 942}%
\special{fp}%
\special{pa 1694 929}%
\special{pa 1701 926}%
\special{fp}%
\special{pa 1736 913}%
\special{pa 1743 911}%
\special{fp}%
\special{pa 1778 899}%
\special{pa 1786 896}%
\special{fp}%
\special{pa 1821 886}%
\special{pa 1829 883}%
\special{fp}%
\special{pa 1864 873}%
\special{pa 1872 871}%
\special{fp}%
\special{pa 1908 862}%
\special{pa 1915 859}%
\special{fp}%
\special{pa 1951 851}%
\special{pa 1959 849}%
\special{fp}%
\special{pa 1995 841}%
\special{pa 2003 840}%
\special{fp}%
\special{pa 2039 833}%
\special{pa 2047 831}%
\special{fp}%
\special{pa 2083 825}%
\special{pa 2091 824}%
\special{fp}%
\special{pa 2128 819}%
\special{pa 2136 818}%
\special{fp}%
\special{pa 2172 813}%
\special{pa 2180 812}%
\special{fp}%
\special{pa 2217 808}%
\special{pa 2225 808}%
\special{fp}%
\special{pa 2262 805}%
\special{pa 2269 804}%
\special{fp}%
\special{pa 2306 802}%
\special{pa 2314 802}%
\special{fp}%
\special{pa 2351 801}%
\special{pa 2359 800}%
\special{fp}%
\special{pa 2396 800}%
\special{pa 2404 800}%
\special{fp}%
\special{pa 2441 800}%
\special{pa 2449 801}%
\special{fp}%
\special{pa 2486 802}%
\special{pa 2494 802}%
\special{fp}%
\special{pa 2531 804}%
\special{pa 2539 805}%
\special{fp}%
\special{pa 2575 808}%
\special{pa 2583 808}%
\special{fp}%
\special{pa 2620 812}%
\special{pa 2628 813}%
\special{fp}%
\special{pa 2664 817}%
\special{pa 2672 819}%
\special{fp}%
\special{pa 2709 824}%
\special{pa 2717 825}%
\special{fp}%
\special{pa 2753 831}%
\special{pa 2761 833}%
\special{fp}%
\special{pa 2797 840}%
\special{pa 2805 841}%
\special{fp}%
\special{pa 2841 849}%
\special{pa 2849 851}%
\special{fp}%
\special{pa 2885 860}%
\special{pa 2892 862}%
\special{fp}%
\special{pa 2928 871}%
\special{pa 2936 874}%
\special{fp}%
\special{pa 2971 883}%
\special{pa 2979 886}%
\special{fp}%
\special{pa 3014 896}%
\special{pa 3022 899}%
\special{fp}%
\special{pa 3057 911}%
\special{pa 3064 914}%
\special{fp}%
\special{pa 3099 926}%
\special{pa 3106 929}%
\special{fp}%
\special{pa 3141 942}%
\special{pa 3148 945}%
\special{fp}%
\special{pa 3182 960}%
\special{pa 3189 963}%
\special{fp}%
\special{pa 3223 977}%
\special{pa 3231 981}%
\special{fp}%
\special{pa 3264 996}%
\special{pa 3271 1000}%
\special{fp}%
\special{pa 3304 1016}%
\special{pa 3312 1019}%
\special{fp}%
\special{pa 3344 1037}%
\special{pa 3351 1040}%
\special{fp}%
\special{pa 3383 1059}%
\special{pa 3390 1062}%
\special{fp}%
\special{pa 3422 1081}%
\special{pa 3429 1085}%
\special{fp}%
%
\special{pn 4}%
\special{sh 1}%
\special{ar 2400 2800 8 8 0 6.2831853}%
\special{sh 1}%
\special{ar 2400 3000 8 8 0 6.2831853}%
\special{sh 1}%
\special{ar 2400 3200 8 8 0 6.2831853}%
\special{sh 1}%
\special{ar 2400 3200 8 8 0 6.2831853}%
\put(24.2000,-27.8000){\makebox(0,0)[lb]{$z_0$}}%
\put(24.2000,-29.8000){\makebox(0,0)[lb]{$z_1$}}%
\put(24.2000,-31.8000){\makebox(0,0)[lb]{$z_2$}}%
%
\special{pn 4}%
\special{sh 1}%
\special{ar 2400 1000 8 8 0 6.2831853}%
\special{sh 1}%
\special{ar 2400 1000 8 8 0 6.2831853}%
\put(24.2000,-9.8000){\makebox(0,0)[lb]{$x$}}%
%
\special{pn 8}%
\special{pa 3000 890}%
\special{pa 3000 1490}%
\special{fp}%
\put(26.0000,-14.9000){\makebox(0,0)[lt]{$\partial B(z_0,\nu^{-1}+\delta_2)$}}%
\end{picture}}%
\end{center}
\caption{A strategy of \sente}
\label{fig:pacman_st}
\end{figure}
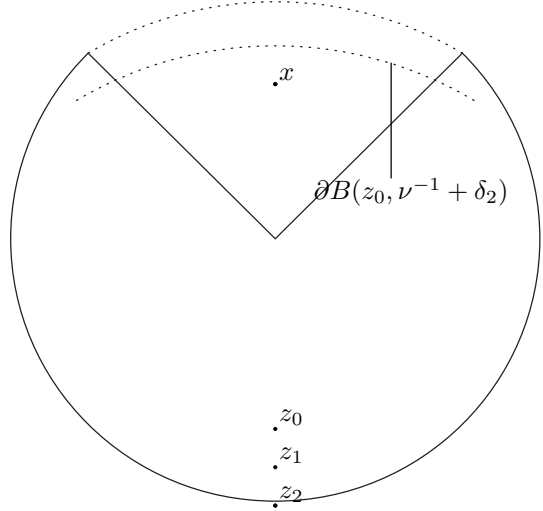
\bm{$2)~x\in A^{int}\setminus O_{-}.$} We set $O^{\delta_1}:=\{x\in\mathbb{R}^2 \mid \Psi_{-}(x)>\delta_1\}$. We see that for $x\in A^{int}\setminus O_{-}$, there exist $\delta_1,\delta_2>0$ and $z_0\in \{(0,q) \mid q\le x\cdot(0,1)\}$ such that $x\in \overline{B(z_0,\nu^{-1}+\delta_2)}$ and $O^{\delta_1}\cap \partial B(z_0,\nu^{-1}+\delta_2)\cap \{(p,q) \mid q>0\}\neq \emptyset$ (Figure \ref{fig:pacman_st}). Define the sequence $\{z_n\}$ so that 
\[z_{n+1}=z_n+\left(0,-\frac{\nu^2\delta_2}{2(1+\nu\delta_2)}\epsilon^2\right).\] 
\sente's strategy is to keep taking a push by moving circle strategy with this $\{z_n\}$ until he reaches $O^{\delta_1}$, where he quits the game. 

By doing this strategy, \sente actually gains positive game cost. Set
\[\Omega_n=B_{\delta_1}(\{(p,q) \mid q\ge \lvert p\rvert\})\cap B(z_n,\nu^{-1}+\delta_2)\]
for $n=0,1,\cdots$. By 2 in Lemma \ref{lem:pmds} we see that if $x_n\in \Omega_n$, then $x_{n+1}\in\Omega_{n+1}$ or \sente quits the game. If $n\epsilon^2$ is sufficiently large, then $\Omega_n=\emptyset$. That means \sente definitely quits the game in finite time and the stopping cost is at least
\[\inf_{y\in O^{\delta_1}} \Psi_{-}(y)=\delta_1>0\]
regardless of $\epsilon$. Hence, together with the case 1), it turns out that for $x\in A^{int}$, there exists $\tau>0$ such that $u(x,t)>0$ for $t>\tau$.
\end{proof}

\begin{figure}[h]
\begin{center}
\input{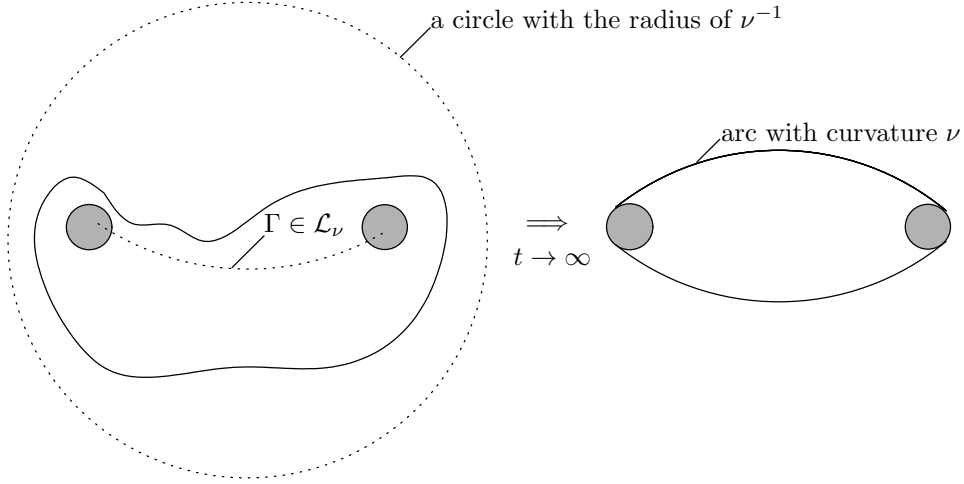}
\end{center}
\caption{Two small balls}
\label{fig:twoball_limitshape}
\end{figure}

We next consider \eqref{eq:bobstacle} with $O_{-}=B_r((d,0))\cup B_r((-d,0))$. In this problem we further assume


\begin{itemize}
\item[(A2)] There exists a function $f:[a,b]\to\mathbb{R}$ such that $\Gamma:=\{(p,q) \mid q=f(p), a\le p\le b\}\in \mathcal{L}_{\nu}$, $(a,f(a))\in B_r((-d,0))$, $(b,f(b))\in B_r((d,0))$ and $\Gamma\subset D_0$.
\end{itemize}
\begin{prop}[Two small balls]
Assume $\mathrm{(A1)}$ and $\mathrm{(A2)}$. Then
\[A^{int}\subset \varliminf_{t\to\infty}D_t \subset \varlimsup_{t\to\infty}D_t\subset A\]
and
\[A^{int}\subset \varliminf_{t\to\infty}E_t \subset \varlimsup_{t\to\infty}E_t\subset A.\]
\end{prop}
\begin{proof}
We take $\delta>0$ small enough to satisfy $B_{2\delta}(\Gamma)\subset D_0$ and $B_{2\delta}((a,f(a)))\cup B_{2\delta}((b,f(b)))\subset O_{-}$. 

\bm{$1)~x\in O_{-}.$}
The proof for this case is the same as before.

\begin{figure}[h]
\begin{center}
\input{two_ball.tex}
\end{center}
\caption{A strategy of \sente}
\label{fig:twoball}
\end{figure}

\bm{$2)~x\in B_{\delta}(\Gamma)\setminus O_{-}.$} \sente keeps taking a $\Gamma$ tube strategy until he reaches $B_{\delta}((a,f(a)))\cup B_{\delta}((b,f(b)))$. Once he reaches $B_{\delta}((a,f(a)))\cup B_{\delta}((b,f(b)))$, he quits the game.

By his doing this strategy, the game trajectory $\{x_n\}$ is restricted to $B_{\delta}(\Gamma)$ as shown in Lemma \ref{lem:ghose}. Thus, whether he quits the game or not, he gains at least $\delta>0$.

\bm{$3)~x\in A^{int}\setminus (B_{\delta}(\Gamma)\cup O_{-}).$} We extend $f$ as follows:
\begin{equation}
\tilde{f}(p):=\begin{cases}
\max\{f(p), \sqrt{(r-\delta_1)^2-(p+d)^2}\}, &-d-r+\delta_1\le p\le -d+r-\delta_1 \\
\max\{f(p), \sqrt{(r-\delta_1)^2-(p-d)^2}\}, &d-r+\delta_1\le p\le d+r-\delta_1 \\
f(p), &-d+r-\delta_1<p<d-r+\delta_1.
\end{cases}
\nonumber
\end{equation}
Without loss of generality we can assume $x\in\{(p,q) \mid q\ge \tilde{f}(p), -d-r+\delta_1\le p\le d+r-\delta_1\}$. We take $\delta_1,\delta_2>0$ and $\{z_n\}$ as in the case 2) in the proof of Proposition \ref{prop:pacman}. See also Figure \ref{fig:twoball}.

\sente's strategy is to keep taking a push by moving circle strategy with this $\{z_n\}$ until he reaches $O^{\delta_1}\cup B_{\delta}(\Gamma)$. Once he reaches $O^{\delta_1}$, he quits the game. Once he reaches $B_{\delta}(\Gamma)$, he takes the same strategy as in the case 2).

By adopting this strategy, \sente actually gains positive game cost. 
Set
\[\Omega_n=\{(p,q) \mid q\ge \tilde{f}(p), -d-r+\delta_1\le p\le d+r-\delta_1\}\cap B(z_n,\nu^{-1}+\delta_2)\setminus B_{\delta}(\Gamma)\]
for $n=0,1,\cdots$. By 2 in Lemma \ref{lem:pmds} we see that if $x_n\in \Omega_n$, then $x_{n+1}\in\Omega_{n+1}$ or \sente reaches $O^{\delta_1}\cup B_{\delta}(\Gamma)$ (Figure \ref{fig:twoball}). If $n\epsilon^2$ is sufficiently large, then $\Omega_n=\emptyset$. That means \sente definitely reaches $O^{\delta_1}\cup B_{\delta}(\Gamma)$ in finite time. Since the stopping cost is at least $\delta_1>0$ and the terminal cost is at least $\delta>0$, it turns out that for $x\in A^{int}$, there exists $\tau>0$ such that $u(x,t)>0$ for $t>\tau$. 
\end{proof}
\begin{rem}
We expect that the same conclusion holds for more general $O_{-}$.
As an analogue of Theorem \ref{thm:conv_general}, we define the graph $G=(V,E)$ as follows: 
\[V:=\{O\subset \mathbb{R}^2 \mid O \ \mbox{is a connected component of} \ O_{-}\},\]
\[E:=\{\langle O,P\rangle \mid \mbox{There exist $x\in O$, $y\in P$ and $\Gamma_{x,y}\in\mathscr{L}_{\nu}$ such that $\Gamma_{x,y}\subset D_0$}\},\]
where we denote a curve $\Gamma=\gamma([0,1])\in\mathscr{L}_{\nu}$ satisfying $\gamma(0)=x$ and $\gamma(1)=y$ by $\Gamma_{x,y}$. It seems that if $O_{-}$ satisfies (A1) and the graph $G$ is connected, then the same conclusion holds.
\end{rem}

\begin{figure}[h]
\begin{center}
\input{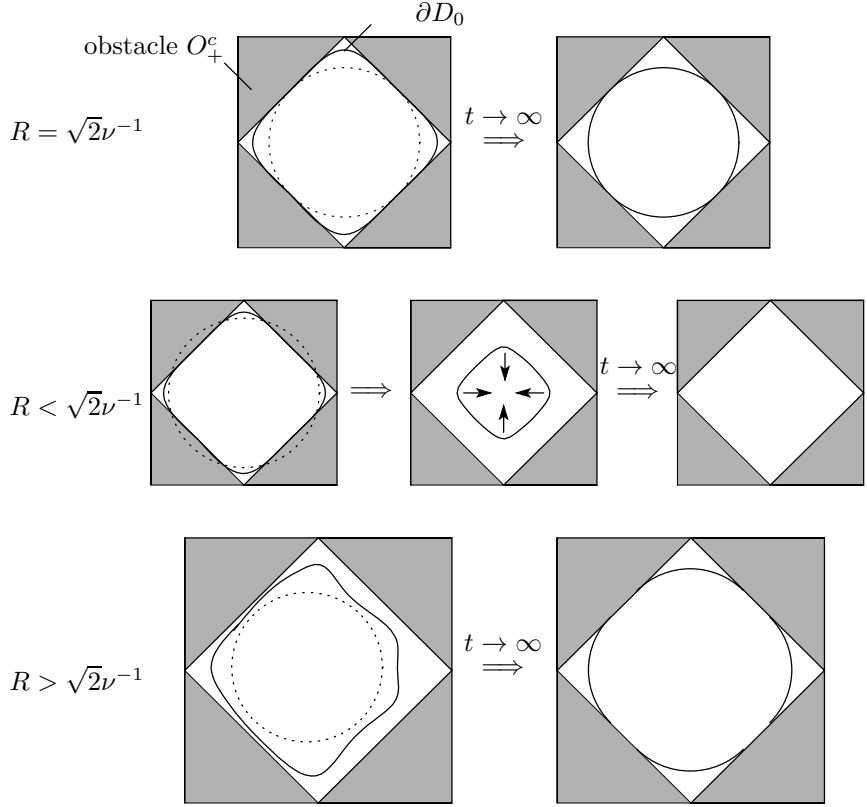}
\end{center}
\caption{Square boxes}
\label{fig:shikaku_waku}
\end{figure}

Finally we give an example of computation of the asymptotic shape of solutions to \eqref{eq:tobstacle}. We can deal with the problem that remains in \cite[Section 6]{GMT.16}, where $O_{+}=\mathbb{R}^2 \setminus \{ (p,q)\in\mathbb{R}^2 \mid \lvert p\rvert+\lvert q\rvert\le R\}$ with $R=\sqrt{2}\nu^{-1}$. We also give a game theoritic proof for the case $R\neq\sqrt{2}\nu^{-1}$, which is considered in a different way in \cite[Section 6]{GMT.16}. 

We set 
\[A:=\bigcup \left\{B_{\nu^{-1}}(z) \mid B_{\nu^{-1}}(z)\subset O_{+}, \ z\in\mathbb{R}^2\right\}.\]
We notice that $A=B_{\nu^{-1}}((0,0))$ if $R=\sqrt{2}\nu^{-1}$ and $A=\emptyset$ if $R<\sqrt{2}\nu^{-1}$. Figure \ref{fig:shikaku_waku} shows the result of the asymptotic shapes. In Figure \ref{fig:shikaku_waku}, dotted circles are circles with the radius of $\nu^{-1}$.

\begin{prop}
Assume either of the following:
\begin{enumerate}
\item $R=\sqrt{2}\nu^{-1}$ and $B_{\nu^{-1}}((0,0))\subset D_0$,
\item $R<\sqrt{2}\nu^{-1}$,
\item $R>\sqrt{2}\nu^{-1}$ and $B(\hat{z}, \nu^{-1}+\delta)\subset D_0$ for some $\hat{z}\in\mathbb{R}^2$ and $\delta>0$.
\end{enumerate}
Then
\[A\subset \varliminf_{t\to\infty}D_t \subset \varlimsup_{t\to\infty}D_t\subset \overline{A}\]
and
\[A\subset \varliminf_{t\to\infty}E_t \subset \varlimsup_{t\to\infty}E_t\subset \overline{A}.\]
\end{prop}
\begin{proof}
It suffices to take $u_0$ as \eqref{good_ini}. Similarly it suffices to let
\begin{equation}
\Psi_{+}(x)=\begin{cases}
dist(x, \partial O_{+}), \ x\in O_{+} \\
\max\{a, -dist(x, \partial O_{+})\}, \ x\in O_{+}^c.
\end{cases}
\nonumber
\end{equation}
\begin{enumerate}
\item \bm{$1)~x\in B_{\nu^{-1}}((0,0)).$} \sente's strategy is to keep taking a \zcenp{(0,0)}. We denote by $V^{\epsilon}(x,t)$ the total cost when \sente takes this strategy with the game variables $(x,t,\epsilon)$ and \gote does not quit the game on the way. Notice that $V^{\epsilon}(x,t)$ does not depend on \gote's choices $\{b_n\}$ because $u_0$ is radially symmetric. Since $u_0(x_n)$ is monotonically decreasing with respect to $n$ and $u_0(x_n)\le \Psi_{+}(x_n)$, \gote's optimal strategy for the strategy of \sente is not to quit the game on the way. Thus the inequality $u^{\epsilon}(x,t)\ge V^{\epsilon}(x,t)$ is satisfied. Fix $x\in B_{\nu^{-1}}((0,0))$ and $t>0$. By Lemma \ref{lem:0815} we have $V^{\epsilon}(x,t)> 0$ for sufficiently small $\epsilon>0$. Lemma \ref{lem:monotonicity} 2 implies that for any subsequence $\{\epsilon_n\}$ satisfying $\epsilon_n=2^{-n}\epsilon_0$, $V^{\epsilon_n}(x,t)$ is monotonically increasing with respect to $n$. Therefore we obtain $u(x,t)>0$.

\begin{figure}[h]
\begin{center}
{\unitlength 0.1in%
\begin{picture}(31.8000,23.7500)(8.0000,-31.7500)%
%
\special{pn 8}%
\special{pn 8}%
\special{pa 1266 2390}%
\special{pa 1266 2382}%
\special{fp}%
\special{pa 1267 2345}%
\special{pa 1267 2337}%
\special{fp}%
\special{pa 1270 2300}%
\special{pa 1270 2292}%
\special{fp}%
\special{pa 1274 2255}%
\special{pa 1275 2247}%
\special{fp}%
\special{pa 1280 2210}%
\special{pa 1282 2202}%
\special{fp}%
\special{pa 1289 2166}%
\special{pa 1291 2158}%
\special{fp}%
\special{pa 1299 2122}%
\special{pa 1300 2114}%
\special{fp}%
\special{pa 1310 2078}%
\special{pa 1313 2071}%
\special{fp}%
\special{pa 1324 2035}%
\special{pa 1326 2027}%
\special{fp}%
\special{pa 1339 1992}%
\special{pa 1341 1985}%
\special{fp}%
\special{pa 1355 1950}%
\special{pa 1358 1943}%
\special{fp}%
\special{pa 1374 1909}%
\special{pa 1378 1902}%
\special{fp}%
\special{pa 1395 1869}%
\special{pa 1398 1862}%
\special{fp}%
\special{pa 1416 1829}%
\special{pa 1420 1822}%
\special{fp}%
\special{pa 1439 1790}%
\special{pa 1443 1784}%
\special{fp}%
\special{pa 1464 1753}%
\special{pa 1468 1746}%
\special{fp}%
\special{pa 1490 1716}%
\special{pa 1495 1710}%
\special{fp}%
\special{pa 1518 1680}%
\special{pa 1523 1674}%
\special{fp}%
\special{pa 1548 1646}%
\special{pa 1553 1640}%
\special{fp}%
\special{pa 1578 1613}%
\special{pa 1584 1607}%
\special{fp}%
\special{pa 1610 1581}%
\special{pa 1616 1575}%
\special{fp}%
\special{pa 1643 1550}%
\special{pa 1649 1545}%
\special{fp}%
\special{pa 1677 1521}%
\special{pa 1684 1516}%
\special{fp}%
\special{pa 1713 1493}%
\special{pa 1719 1488}%
\special{fp}%
\special{pa 1749 1467}%
\special{pa 1756 1462}%
\special{fp}%
\special{pa 1787 1442}%
\special{pa 1794 1438}%
\special{fp}%
\special{pa 1826 1418}%
\special{pa 1833 1414}%
\special{fp}%
\special{pa 1865 1396}%
\special{pa 1872 1392}%
\special{fp}%
\special{pa 1906 1376}%
\special{pa 1913 1373}%
\special{fp}%
\special{pa 1947 1357}%
\special{pa 1954 1354}%
\special{fp}%
\special{pa 1989 1340}%
\special{pa 1996 1337}%
\special{fp}%
\special{pa 2031 1325}%
\special{pa 2039 1323}%
\special{fp}%
\special{pa 2074 1311}%
\special{pa 2082 1309}%
\special{fp}%
\special{pa 2118 1299}%
\special{pa 2125 1297}%
\special{fp}%
\special{pa 2162 1289}%
\special{pa 2170 1288}%
\special{fp}%
\special{pa 2206 1281}%
\special{pa 2214 1280}%
\special{fp}%
\special{pa 2251 1275}%
\special{pa 2259 1274}%
\special{fp}%
\special{pa 2296 1270}%
\special{pa 2304 1269}%
\special{fp}%
\special{pa 2341 1267}%
\special{pa 2349 1267}%
\special{fp}%
\special{pa 2386 1266}%
\special{pa 2394 1266}%
\special{fp}%
\special{pa 2431 1267}%
\special{pa 2439 1267}%
\special{fp}%
\special{pa 2476 1269}%
\special{pa 2484 1270}%
\special{fp}%
\special{pa 2521 1273}%
\special{pa 2529 1275}%
\special{fp}%
\special{pa 2566 1280}%
\special{pa 2574 1281}%
\special{fp}%
\special{pa 2610 1288}%
\special{pa 2618 1289}%
\special{fp}%
\special{pa 2655 1297}%
\special{pa 2662 1300}%
\special{fp}%
\special{pa 2698 1309}%
\special{pa 2706 1312}%
\special{fp}%
\special{pa 2742 1322}%
\special{pa 2749 1325}%
\special{fp}%
\special{pa 2784 1337}%
\special{pa 2791 1340}%
\special{fp}%
\special{pa 2826 1354}%
\special{pa 2833 1357}%
\special{fp}%
\special{pa 2867 1372}%
\special{pa 2875 1376}%
\special{fp}%
\special{pa 2908 1393}%
\special{pa 2915 1396}%
\special{fp}%
\special{pa 2948 1414}%
\special{pa 2954 1418}%
\special{fp}%
\special{pa 2986 1437}%
\special{pa 2993 1442}%
\special{fp}%
\special{pa 3024 1462}%
\special{pa 3031 1466}%
\special{fp}%
\special{pa 3061 1488}%
\special{pa 3067 1493}%
\special{fp}%
\special{pa 3097 1516}%
\special{pa 3103 1521}%
\special{fp}%
\special{pa 3131 1545}%
\special{pa 3137 1551}%
\special{fp}%
\special{pa 3165 1575}%
\special{pa 3170 1581}%
\special{fp}%
\special{pa 3197 1607}%
\special{pa 3202 1613}%
\special{fp}%
\special{pa 3227 1640}%
\special{pa 3233 1646}%
\special{fp}%
\special{pa 3257 1675}%
\special{pa 3262 1681}%
\special{fp}%
\special{pa 3285 1710}%
\special{pa 3290 1716}%
\special{fp}%
\special{pa 3311 1746}%
\special{pa 3316 1753}%
\special{fp}%
\special{pa 3336 1784}%
\special{pa 3341 1791}%
\special{fp}%
\special{pa 3360 1822}%
\special{pa 3364 1829}%
\special{fp}%
\special{pa 3382 1862}%
\special{pa 3386 1869}%
\special{fp}%
\special{pa 3403 1902}%
\special{pa 3406 1909}%
\special{fp}%
\special{pa 3421 1943}%
\special{pa 3425 1951}%
\special{fp}%
\special{pa 3439 1985}%
\special{pa 3441 1993}%
\special{fp}%
\special{pa 3454 2028}%
\special{pa 3456 2035}%
\special{fp}%
\special{pa 3468 2070}%
\special{pa 3470 2078}%
\special{fp}%
\special{pa 3479 2114}%
\special{pa 3481 2122}%
\special{fp}%
\special{pa 3490 2158}%
\special{pa 3492 2166}%
\special{fp}%
\special{pa 3498 2202}%
\special{pa 3499 2210}%
\special{fp}%
\special{pa 3505 2247}%
\special{pa 3506 2255}%
\special{fp}%
\special{pa 3510 2292}%
\special{pa 3511 2300}%
\special{fp}%
\special{pa 3513 2337}%
\special{pa 3513 2345}%
\special{fp}%
\special{pa 3514 2382}%
\special{pa 3514 2390}%
\special{fp}%
%
\special{pn 8}%
\special{pa 800 800}%
\special{pa 2390 800}%
\special{pa 800 2390}%
\special{pa 800 800}%
\special{pa 2390 800}%
\special{fp}%
%
\special{pn 8}%
\special{pa 3980 2390}%
\special{pa 3980 800}%
\special{pa 2390 800}%
\special{pa 3980 2390}%
\special{pa 3980 800}%
\special{fp}%
%
\special{pn 8}%
\special{ar 2400 2100 1075 1075 0.0000000 6.2831853}%
%
\special{pn 4}%
\special{sh 1}%
\special{ar 2400 2100 8 8 0 6.2831853}%
\special{sh 1}%
\special{ar 2400 2100 8 8 0 6.2831853}%
\put(24.0000,-21.0000){\makebox(0,0)[lb]{$z_0$}}%
%
\special{pn 8}%
\special{pa 2400 2100}%
\special{pa 3160 2860}%
\special{dt 0.045}%
\put(27.6000,-24.6000){\makebox(0,0)[lb]{$\nu^{-1}-\delta_1$}}%
%
\special{pn 4}%
\special{sh 1}%
\special{ar 2400 1030 8 8 0 6.2831853}%
\special{sh 1}%
\special{ar 2400 1030 8 8 0 6.2831853}%
\put(24.0000,-9.9000){\makebox(0,0)[lb]{$x$}}%
%
\special{pn 8}%
\special{pa 800 2320}%
\special{pa 2320 800}%
\special{dt 0.045}%
%
\special{pn 4}%
\special{sh 1}%
\special{ar 2400 1900 8 8 0 6.2831853}%
\special{sh 1}%
\special{ar 2400 1700 8 8 0 6.2831853}%
\special{sh 1}%
\special{ar 2400 1700 8 8 0 6.2831853}%
\put(24.0000,-19.0000){\makebox(0,0)[lb]{$z_1$}}%
\put(24.0000,-17.0000){\makebox(0,0)[lb]{$z_2$}}%
\end{picture}}%
\end{center}
\caption{\gote's strategy}
\label{fig:shikaku_waku_st1}
\end{figure}

\bm{$2)~x\in \overline{O_{+}}^c.$} \gote's strategy is to quit the game at the first round. 

\bm{$3)~x=(p,q)\in \overline{A}^c \setminus\overline{O_{+}}^c.$} We may assume $q\ge \lvert p\rvert$. We can take $z_0\in\mathbb{R}^2$, $\delta_1>0$ and $\delta_2>0$ so that $\lvert x-z_0\rvert\ge \nu^{-1}-\delta_1$, $z_0\in \{(0,y) \mid y>0\}$ and $\partial B(z_0,\nu^{-1}-\delta_1)\cap \overline{B_{\delta_2}(O_{+})}^c\neq \emptyset$. \gote's strategy is to take a push by moving circle strategy with $\{z_n\}\subset \{(0,q) \mid q>0\}$ satisfying $z_{n+1}=z_n+(0,\frac{\delta_1}{2}\nu^2\epsilon^2)$ for all $n$ until \sente reaches $\overline{B_{\delta_2}(O_{+})}^c$. Once he reaches there, she quits the game. See Figure \ref{fig:shikaku_waku_st1}.

By doing above strategy, \gote actually pays nagative game cost. To show it, set
\[\Omega_n=\{(p,q) \mid q\ge \lvert p\rvert\}\cap\left(\overline{B_{\delta_2}(O_{+})}\setminus B(z_n,\nu^{-1}-\delta_1)\right)\]
for $n=0,1,\cdots$. We see that if $x_n\in \Omega_n$, then $x_{n+1}\in \Omega_{n+1}$ or \gote quits the game at round $n+1$. If $n\epsilon^2$ is sufficiently large, then $\Omega_n=\emptyset$. That means \gote definitely quits the game in finite time and the stopping cost is at most
\[\sup_{y\in \overline{B_{\delta_2}(O_{+})}} \Psi_{+}(y)=-\delta_2<0\]
regardless of $\epsilon$. Therefore we obtain $u(x,t)<0$ for $x\in \overline{B_{\nu^{-1}}((0,0))}^c$ and sufficiently large $t>0$.
\item We take $\delta_1>0$ to satisfy $B((0,0),\nu^{-1}-2\delta_1)\cap \overline{O_{+}}^c\neq \emptyset$.

\bm{$1)~x\in \overline{O_{+}}^c.$} \gote's strategy is to quit the game at the first round.

\bm{$2)~x\in \overline{O_{+}}\setminus B((0,0),\nu^{-1}-\delta_1).$} \gote's strategy is similar to that in the case 3) in $1$.

\bm{$3)~x\in \overline{O_{+}}\cap B((0,0),\nu^{-1}-\delta_1).$} \gote keeps taking a \zcenc{$(0,0)$} until \sente is forced to reach $B((0,0),\nu^{-1}-\delta_1)^c\cap B_{\delta_1}(O_{+})^c$. By Lemma \ref{lem:1109} it takes at most finite time for \sente to reach there. Once \sente reaches $B_{\delta_1}(O_{+})^c$, \gote quits the game. Once \sente reaches $\overline{O_{+}}\setminus B((0,0),\nu^{-1}-\delta_1)$, \gote's strategy is as in the case 2).

Therefore, for $x\in\mathbb{R}^2$, we obtain $u(x,t)<0$ for sufficiently large $t>0$.

\item \bm{$1)~x\in \overline{O_{+}}^c.$} \gote's strategy is to quit the game at the first round.

\bm{$2)~x\in \overline{O_{+}}\setminus A.$} \gote's strategy is similar to that in the case 3) in $1$.

\bm{$3)~x\in A.$} We set $O^{\delta_1}:=\{x\in\mathbb{R}^2 \mid \Psi_{+}(x)>\delta_1\}$ and $C_{\delta_1,\delta_2}:=\{z\in\mathbb{R}^2 \mid B(z,\nu^{-1}+\delta_2)\subset O^{\delta_1}\}$ for $\delta_1,\delta_2>0$. For $x\in A$, there exist $\delta_1,\delta_2>0$ such that 
\[x\in\bigcup \left\{B(z,\nu^{-1}+\delta_2) \mid z\in C_{\delta_1,\delta_2}\right\}\]
and $C_{0,\delta}\subset C_{\delta_1,\delta_2}$.
Let $z_0\in\mathbb{R}^2$ be a point satisfying $x\in B(z_0,\nu^{-1}+\delta_2)$ and $z_0\in C_{\delta_1,\delta_2}$. We define the sequence $\{z_n\}$ by
\[z_n:=z_0+\min\left\{\frac{\delta_2}{2}n\nu^2\epsilon^2, \lvert\hat{z}-z_0\rvert\right\}\frac{\hat{z}-z_0}{\lvert\hat{z}-z_0\rvert}.\]
Since $C_{\delta_1,\delta_2}$ is now convex, we have $l_{z_0,\hat{z}}\subset C_{\delta_1,\delta_2}$.

\sente's strategy is to take a push by moving circle strategy with this $\{z_n\}$. Indeed if \gote quits the game on the way, the stopping cost is at least $\delta_1>0$ because the game trajectory $\{x_n\}$ is contained in $O^{\delta_1}$. If \gote does not quit the game, the terminal cost is at least $\delta-\delta_1$, which is positive. That is because $x_N\in B(\hat{z},\nu^{-1}+\delta_1)\subset B(\hat{z},\nu^{-1}+\delta)\subset D_0$ for any last position $x_N$ of the game.
\end{enumerate}
\end{proof}

\begin{appendices}

\section{Game interpretation and convergence of value functions}
\label{app_ginp}
In this appendix we give a game whose value functions converge to the viscosity solution to \eqref{eq:obstacle}. We introduce the rule of the game corresponding to \eqref{eq:obstacle} with $\nu\ge 0$ and $d=2$ and give the proof of the convergence result for the case. We also remark on the other $\nu\in\mathbb{R}$ and $d$.

The game is almost the same as explained in Section \ref{subsec:gi}. We define the total number $N$ of rounds by $N=\lceil t\epsilon^{-2}\rceil$, where $\lceil r \rceil$ stands for the minimum integer that is no less than $r$. The actions of both players in each round $i$ ($i=1,2,\cdots,N$) are modified as follows:

\begin{enumerate}
\item \sente decides whether to quit the game. 
\item \gote decides whether to quit the game.  
\item \sente chooses $v_i, w_i\in S^1$. ($S^1$ is the set of unit vectors in $\mathbb{R}^2$.)
\item \gote chooses $b_i\in\{\pm 1\}$ after \sente's choice.
\item Determine the next states as follows.
\begin{align}
x_{i}=x_{i-1}+\sqrt{2} \epsilon b_i v_i+\nu \epsilon^2 w_i.  \label{trajectory3}
\end{align}
\end{enumerate}

The total cost is also modified as follows. If \sente quits the game at round $i$, the total cost is given by $\Psi_{-}(x_i)$. If \gote quits the game at round $i$, it is given by $\Psi_{+}(x_i)$. If both players go throughout $N$ rounds of the game, it is given by $u_0(x_N)+\sum_{i=0}^{N-1}\epsilon^2 f(x_{i})$.
The value function $u^{\epsilon}(x,t)$ is defined inductively based on the following {\it Dynamic Programming Principle} and the initial condition:
\begin{equation}
\label{eq:DPP_ob}
\begin{split}
&u^{\epsilon}(x,t)= \\
&\max\{\Psi_{-}(x), \min\{\Psi_{+}(x), \sup_{v,w\in S^1}\min_{b=\pm1}[u^{\epsilon}(x+\sqrt{2}\epsilon b v+\nu w\epsilon^2, t-\epsilon^2)+\epsilon^2 f(x)]\}\}
\end{split}
\end{equation}
for $t>0$. 

\begin{equation}
u^{\epsilon}(x,t)=u_0(x)
\end{equation}
for $t\le 0$.

These value functions mean the total cost optimized by both players. 
\begin{rem}
As explained in Section \ref{subsec:gi}, we can generalize our game to the case $d\ge 3$. In the game corresponding to \eqref{eq:obstacle} with $\nu\ge 0$, \sente chooses a unit vector $w_i$ and $d-1$ orthogonal unit vectors $v_i^{j} (j=1,2,\cdots d-1)$. \gote chooses $d-1$ values $b_i^j\in \{\pm1\} (j=1,2,\cdots d-1)$. The state equation is $x_{i}=x_{i-1}+\nu w_i \epsilon^2+\sqrt{2} \epsilon \sum_{j=1}^{d-1}b_i^j v_i^j$ instead of \eqref{trajectory3}.
\end{rem}
\begin{rem}
The Dynamic Programming Principle corresponding to \eqref{eq:obstacle} with $\nu<0$ is given by
\begin{equation}
\nonumber
u^{\epsilon}(x,t)=\max\{\Psi_{-}(x), \min\{\Psi_{+}(x), \sup_{v\in S^1}\inf_{\substack{w\in S^1 \\ b=\pm1}}[u^{\epsilon}(x+\sqrt{2}\epsilon b v+\nu w\epsilon^2, t-\epsilon^2)+\epsilon^2 f(x)]\}\}.
\end{equation}
Namely, not \sente but \gote has the right to choose $w_i \in S^1$.
\end{rem}
For these value functions, the same type of result as Proposition \ref{prop:gconv_cor} holds.
\begin{prop}
\label{prop:gconv_gene}
the functions $\overline{u}$ and $\underline{u}$ are respectively viscosity sub- and supersolution of \eqref{eq:obstacle}. Moreover $\overline{u}(x,0)=\underline{u}(x,0)=u_0(x)$ for $x\in\mathbb{R}^d$. 
\end{prop}

\begin{rem}
As explained before, \eqref{eq:obstacle} with $f=0$ is a level set equation. By choosing $\Psi_{+}$ so that $\Psi_{+}>\|u_0\|_{\infty}$ for $O_{+}=\mathbb{R}^d$, we can ignore \gote's stopping cost $\Psi_{+}$ when we consider obstacle problems that have an obstacle on one side such as \eqref{eq:main}. Similarly, by choosing $\Psi_{-}$ so that $\Psi_{-}<-\|u_0\|_{\infty}$ for $O_{-}=\emptyset$, we can ignore \sente's stopping cost $\Psi_{-}$.
\end{rem}


We especially show the proof of Proposition \ref{prop:gconv_gene} with $d=2$ and $\nu\ge 0$ because the other case is similar. Our proof is based directly on the game as in \cite{KS.06}, whereas those in \cite{HL20,KS.10} are based on the properties of the operator whose fixed point is the solution of the Dynamic Programming Principle. Also since the proof in \cite{KS.06} is local argument, roughly speaking, all we have to do is to do the local argument in $\{(x,t) \mid \Psi_{-}(x)<\overline{u}(x,t)\}$ or in $\{(x,t) \mid \Psi_{+}(x)>\underline{u}(x,t)\}$. However we need to care about the point that $\{(x,t) \mid \Psi_{-}(x)<\overline{u}(x,t)\}$ or $\{(x,t) \mid \Psi_{+}(x)>\underline{u}(x,t)\}$ may not be open.

The proof consists of three steps. We show that the limits of the value functions satisfy the conditions (a) and (c) in Definition \ref{def:vis} in the first two propositions, (Proposition \ref{thm:ob_conv} and \ref{prop:subsol}) and they satisfy the initial condition (b) in the last one. (Proposition \ref{lem:inigame}) We mention that the initial data $u_0$ is assumed to be just continuous, not to be Lipschitz continuous as in \cite{HL20} or bounded uniformly continuous as in \cite{KS.10}. Regarding the last proposition, the idea of the proof is similar to that of \cite[Proposition 3.1]{GL09} though the situation is different. 


To visualize choices of players of the game, we give another description of the level-set mean curvature flow operator $F$:
\[F(Du,D^2 u)=-\left\langle D^2 u \frac{D^{\perp}u}{\lvert Du\rvert},\frac{D^{\perp}u}{\lvert Du\rvert} \right\rangle\]
for $Du\neq 0$. Here we denote by $D^{\perp}u$ a vector field satisfying $Du\cdot D^{\perp}u=0$ and $\lvert Du\rvert=\lvert D^{\perp}u\rvert$ in $\mathbb{R}^2$.

  \begin{prop}
 The function $\underline{u}$ is a viscosity supersolution of \eqref{eq:obstacle} in $\mathbb{R}^2\times (0,\infty)$.
  \label{thm:ob_conv}
  \end{prop}
\begin{proof}
As for Definition \ref{def:vis}-2(a), we directly have $\Psi_{-}(x)\le u^{\epsilon}(x,t)\le \Psi_{+}(x)$ by the Dynamic Programing Principle \eqref{eq:DPP_ob}. Thus we obtain $\Psi_{-}(x)\le \underline{u}(x,t)\le \Psi_{+}(x)$ since $\Psi_{+}$ and $\Psi_{-}$ are continuous. To prove the viscosity inequality, we argue by contradiction. For a smooth function $\phi:\mathbb{R}^2\times (0,\infty)\to\mathbb{R}$, a positive constant $\theta_0>0$ and $(x,t)\in\mathbb{R}^2\times (0,\infty)$, we consider the following condition (C): 
\begin{equation}
\partial_{t}\phi(x,t)-\nu\lvert D\phi(x,t)\rvert-\left\langle D^2\phi(x,t) \frac{D^{\perp}\phi(x,t)}{\lvert D\phi(x,t)\rvert},\frac{D^{\perp}\phi(x,t)}{\lvert D\phi(x,t)\rvert} \right\rangle-f_{\ast}(x)\le -\theta_0 <0
\label{ineq1}
\end{equation}
if $D\phi(x,t)\neq0$, and
\begin{equation}
\partial_{t}\phi(x,t)-\nu\lvert D\phi(x,t)\rvert-\inf_{\lvert \zeta\rvert=1} \left\langle D^2\phi(x,t)\zeta, \zeta\right\rangle-f_{\ast}(x)\le -\theta_0 <0
\label{ineq2}
\end{equation}
if $D\phi(x,t)=0$. 

We assume that there exist a smooth function $\phi$ and $(\hat{x}, \hat{t})$ such that
$(\hat{x}, \hat{t})$ is a strict local minimum of $\underline{u}-\phi$, $\underline{u}<\Psi_{+}$ at $(\hat{x}, \hat{t})$ and the condition (C) is satisfied at $(\hat{x}, \hat{t})$ with $\phi$ and some $\theta_0>0$. 
\
Then we can take a $\delta$ neighborhood of $(\hat{x},\hat{t})$ where $\underline{u}-\phi$ attains its unique minimum at $(\hat{x},\hat{t})$ and the condition (C) holds, retaking smaller $\theta_0>0$ if necessary. For technical reasons, we take such $\delta$ neighborhood as $N_{\delta}((\hat{x},\hat{t})):=\{(x,t)\in\mathbb{R}^2\times[0,\infty) \mid \lvert x-\hat{x}\rvert+\lvert t-\hat{t}\rvert< \delta\}$ and $\delta>0$ small enough to satisfy $\delta\le\frac{a}{2\max\{L,M\}}$, where $a:=\Psi_{+}(\hat{x})-\underline{u}(\hat{x},\hat{t})$, $L$ is the Lipschitz constant of $\Psi_{+}$ and $M=\sup_{y\in B_1(\hat{x})}\lvert f(y)\rvert$. 

From the definition of $\underline{u}$, there are sequence $\{\epsilon_n\}$, $\{x_{\epsilon_n}^0\}$, and $\{t_{\epsilon_n}^0\}$ satisfying
\[\epsilon_n \searrow 0,~~(x_{\epsilon_n}^0,t_{\epsilon_n}^0)\to (\hat{x},\hat{t}),~~u^{\epsilon_n}(x_{\epsilon_n}^0,t_{\epsilon_n}^0)\to \underline{u}(\hat{x},\hat{t}).\]
We may substitute $\epsilon$ for $\epsilon_n$ hereafter.
We take $\epsilon$ small enough to satisfy $(x_{\epsilon}^0, t_{\epsilon}^0)\in N_{\delta/2}((\hat{x},\hat{t}))$ and $\Psi_{+}(x_{\epsilon}^0)-u^{\epsilon}(x_{\epsilon}^0,t_{\epsilon}^0)\ge a/2$. For any $\epsilon$, we construct the sequence $\{X_k\}$ satisfying
\[X_0=(x_{\epsilon}^0,t_{\epsilon}^0),\]
\begin{equation}
X_{k+1}=X_k+(\sqrt{2}\epsilon b_k \eta_k^{\perp}+\nu\eta_k\epsilon^2, -\epsilon^2),
\label{seq1}
\end{equation}
where $\eta_k=\frac{D\phi(X_k)}{\lvert D\phi(X_k)\rvert}$ if $D\phi(X_k)\neq 0$, and is an arbitrary unit vector if $D\phi(X_k)=0$. We will determine $b_k$ later. Let $x_k$ and $t_k$ denote the spatial and time component of $X_k$ respectively hereafter. Let $k_{\epsilon}$ be the maximal $k$ satisfying 
\begin{equation}
X_j \in N_{\delta/2}((\hat{x},\hat{t})) \ {\rm for \ any}  \ j=0,1,\ldots ,k-1.
\nonumber
\end{equation}
Indeed such $k_{\epsilon}$ exists because of the definition of the sequence $\{X_k\}$. We prove by induction that $\Psi_{+}(x_k)>u^{\epsilon}(X_k)$ and
\begin{equation} 
u^{\epsilon}(X_0)-u^{\epsilon}(X_k) \ge \epsilon^2\sum_{m=0}^{k-1} f(x_m)
\label{ineq:1205}
\end{equation}
for all $k<k_{\epsilon}$. These inequalities hold for $k=0$. We assume that these hold for some $k<k_{\epsilon}$.
Then the Dynamic Programming Principle \eqref{eq:DPP_ob} and $\Psi_{+}(x_k)>u^{\epsilon}(X_k)$ imply
\begin{equation}
u^{\epsilon}(X_k)=\max\{\Psi_{-}(x_k), \sup_{v,w\in S^1}\min_{b=\pm1}u^{\epsilon}(x_k+\sqrt{2}\epsilon b v+\nu w\epsilon^2, t_k-\epsilon^2)+\epsilon^2 f(x_k)\}.
\nonumber
\end{equation}
Thus we have
\begin{equation}
u^{\epsilon}(X_k)\ge \min_{b=\pm1} u^{\epsilon}(x_k+\sqrt{2}\epsilon b \eta_{0}^{\perp}+\nu\eta_0 \epsilon^2, t_k-\epsilon^2)+\epsilon^2 f(x_k). \label{ineq11}
\end{equation}
We determine $b_k$ in $(\ref{seq1})$ as a minimizer $b$ in (\ref{ineq11}). We then get 
\[u^{\epsilon}(X_k)-u^{\epsilon}(X_{k+1}) \ge \epsilon^2 f(x_k).\]
Adding \eqref{ineq:1205}, we have
\[u^{\epsilon}(X_0)-u^{\epsilon}(X_{k+1}) \ge \epsilon^2\sum_{m=0}^{k} f(x_m)\]
and consequently 
\begin{equation}
\label{eq:val1208}
u^{\epsilon}(X_0)+M(k+1)\epsilon^2\ge u^{\epsilon}(X_{k+1}).
\end{equation}
If $k+1<k_{\epsilon}$, we see from the definition of $k_{\epsilon}$ that $X_{k+1}\in N_{\delta/2}((\hat{x},\hat{t}))$ and thus
\begin{equation}
\lvert x_{k+1}-x_{0}\rvert+\lvert t_{k+1}-t_0\rvert\le \lvert x_{k+1}-\hat{x}\rvert+\lvert x_0-\hat{x}\rvert+\lvert t_{k+1}-\hat{t}\rvert+\lvert t_0-\hat{t}\rvert\le \delta.
\label{ineq0323}
\end{equation}
From the Lipschitz continuity of $\Psi_{+}$, we have
\begin{equation}
\lvert\Psi_{+}(x_{k+1})-\Psi_{+}(x_0)\rvert \le L\lvert x_{k+1}-x_0\rvert\le L(\delta-\lvert t_{k+1}-t_0\rvert)=L(\delta-(k+1)\epsilon^2).
\label{eq:Lip1208}
\end{equation}
Combining \eqref{eq:val1208} and \eqref{eq:Lip1208}, we obtain
\begin{align*}
\Psi_{+}(x_{k+1})-u^{\epsilon}(X_{k+1})\ge \Psi_{+}(x_0)-L(\delta-(k+1)\epsilon^2)-u^{\epsilon}(X_0)-M(k+1)\epsilon^2
\end{align*}
\begin{align*}
&\ge \Psi_{+}(x_0)-\max\{L,M\}(\delta-(k+1)\epsilon^2)-u^{\epsilon}(X_0)-\max\{L,M\}(k+1)\epsilon^2 \\
&\ge \Psi_{+}(x_0)-u^{\epsilon}(X_0)-\max\{L,M\}\delta \\
&> \frac{a}{2}-\frac{a}{2}=0
\end{align*}
and conclude the induction.

Next we take the continuous path that affinely interpolates among $\{X_k\}$, i.e.,
$X(s)=X_k+\left(s\epsilon^{-2}-k\right)(X_{k+1}-X_k)$
for $k\epsilon^2 \le s \le (k+1)\epsilon^2$, and we write $X(s)=(x(s),t_{\epsilon}^0-s)$.
Using Taylor's theorem for $\phi(X(t))$ at $t=k\epsilon^2$, we get
\begin{equation}
\phi(X_{k+1})-\phi(X_k)=\epsilon^2\{-\partial_{t}\phi(X_k)+\nu\lvert D\phi(X_k)\rvert+\langle D^2\phi(X_k) \eta_{k}^{\perp}, \eta_{k}^{\perp}\rangle\}+\Psi_k(\epsilon),
\label{Taylor}
\end{equation}
where $\Psi_k(\epsilon)=o(\epsilon^2).$
Moreover, from the assumption (\ref{ineq1}), we have
\begin{equation}
\phi(X_{k+1})-\phi(X_k)\ge\epsilon^2(\theta_0-f_{\ast}(x_k))+\Psi_k(\epsilon).
\nonumber
\end{equation}
This inequality is also obtained in the case $D\phi(X_k)=0$, using \eqref{Taylor} and \eqref{ineq2}.
Summing up both sides, we have
\begin{equation}
\phi(X_k)-\phi(X_0) \ge k\epsilon^2\theta_0-\epsilon^2\sum_{m=0}^{k-1} f_{\ast}(x_m)+\sum_{m=0}^{k-1}\Psi_m(\epsilon).
\label{ineqphi}
\end{equation}
Provided $k<k_{\epsilon}$, we have
\[\lvert\Psi_k(\epsilon)\rvert \le C\epsilon^3,\]
where $C$ depends on $\phi$ in $\delta$ neighborhood around $(\hat{x},\hat{t})$, and does not depend on $k$. This estimation is derived from the Taylor expansion (\ref{Taylor}). Hence (\ref{ineqphi}) becomes
\[\phi(X_k)-\phi(X_0) \ge k\epsilon^2\theta_0-\epsilon^2\sum_{m=0}^{k-1} f_{\ast}(x_m)+k C \epsilon^3,~~~~~k \le k_{\epsilon}.\]
Adding this relation to (\ref{ineq:1205}), we have
\begin{align}
u^{\epsilon}(X_0)-\phi(X_0)&\ge u^{\epsilon}(X_k)-\phi(X_k)+\epsilon^2\sum_{m=0}^{k-1}(f(x_m)-f_{\ast}(x_m))+k\epsilon^2\theta_0+kC\epsilon^3 \nonumber \\ &\ge u^{\epsilon}(X_k)-\phi(X_k)+k\epsilon^2\theta_0+kC\epsilon^3. \nonumber
\end{align}
For sufficiently small $\epsilon$, we get
\begin{equation}
-\frac{k}{2}\epsilon^2\theta_0 \ge u^{\epsilon}(X_k)-u^{\epsilon}(X_0)-\phi(X_k)+\phi(X_0),~~~~~k \le k_{\epsilon}.
\label{ineq21}
\end{equation}

By the definition of $k_{\epsilon}$, we see that $Y_{\epsilon} \in \overline{N_{3\delta/4}((\hat{x},\hat{t}))\setminus N_{\delta/2}((\hat{x},\hat{t}))}$, where we substitute $Y_{\epsilon}$ for $X_{k_{\epsilon}}$. So there is a subsequence $\{Y_{\epsilon_n}\}_n$ such that 
\[\lim_{n \to \infty} Y_{\epsilon_n}=(x',t'),\]
where $(x',t') \in B_{\delta}((\hat{x},\hat{t}))$ and $(x',t')\neq (\hat{x},\hat{t})$.
From (\ref{ineq21}), we have
\[u^{\epsilon_n}(x_{\epsilon_n}^0,t_{\epsilon_n}^0)-\phi(x_{\epsilon_n}^0,t_{\epsilon_n}^0)\ge u^{\epsilon_n}(Y_{\epsilon_n})-\phi(Y_{\epsilon_n}).\]
Letting $n$ go to $\infty$, we obtain
\[\underline{u}(\hat{x},\hat{t})-\phi(\hat{x},\hat{t})\ge \underline{u}(x',t')-\phi(x',t').\]
This is a contradiction since $\underline{u}-\phi$ attains its unique minimum at $(\hat{x},\hat{t})$.
\end{proof}

The following lemma can be found in \cite[Lemma 2.3]{KS.06}. 
\begin{lem}
\label{lem:kohn}
Let $\phi$ be a $C^3$ function on a compact subset $K$ of $\mathbb{R}^2$. Let $x\in K$ and $\epsilon\in[0,\infty)$. If $D\phi(x)\neq 0$, there exists a constant $C_1$ (depending only on the $C^2$ norm of $\phi$) with the following two properties for all unit vectors $v\in\mathbb{R}^2$.

\
\begin{enumerate}
\item If $\sqrt{2}\lvert\langle D\phi(x),v \rangle\rvert\ge C_1 \epsilon$,
\[\sqrt{2}\epsilon\lvert\langle D\phi,v \rangle\rvert+\epsilon^2 \langle D^2\phi v, v\rangle \ge \epsilon^2 \left\langle D^2\phi \frac{D^{\perp}\phi}{\lvert D\phi\rvert},\frac{D^{\perp}\phi}{\lvert D\phi\rvert} \right\rangle\]
at $x$.
\item If $\sqrt{2}\lvert\langle D\phi(x),v \rangle\rvert\le C_1 \epsilon$, there exists a constant $C_2$ (depending only on the $C^2$ norm of $\phi$) such that
\[\sqrt{2}\epsilon\lvert\langle D\phi,v \rangle\rvert+\epsilon^2 \langle D^2\phi v, v\rangle \ge \epsilon^2 \left\langle D^2\phi \frac{D^{\perp}\phi}{\lvert D\phi\rvert},\frac{D^{\perp}\phi}{\lvert D\phi\rvert} \right\rangle-\frac{C_2\epsilon^3}{\lvert D\phi\rvert}\]
at $x$.
\end{enumerate}

\end{lem}
  \begin{prop}
 The function $\overline{u}$ is a viscosity subsolution of \eqref{eq:obstacle} in $\mathbb{R}^2\times (0,\infty)$.
 \label{prop:subsol}
  \end{prop}
\begin{proof}
As in Proposition \ref{thm:ob_conv}, we obtain $\Psi_{-}(x)\le \overline{u}(x,t)\le \Psi_{+}(x)$ and argue by contradiction to prove the viscosity inequality. We prepare the following conditions: 
\begin{equation}
\partial_{t}\phi(x,t)-\nu\lvert D\phi(x,t)\rvert-\left\langle D^2\phi(x,t) \frac{D^{\perp}\phi(x,t)}{\lvert D\phi(x,t)\rvert},\frac{D^{\perp}\phi(x,t)}{\lvert D\phi(x,t)\rvert} \right\rangle-f^{\ast}(x)\ge \theta_0 >0,
\label{ineq31}
\end{equation}
\begin{equation}
\partial_{t}\phi(x,t)-\nu\lvert D\phi(x,t)\rvert-\sup_{\lvert\zeta\rvert=1} \left\langle D^2\phi(x,t)\zeta, \zeta\right\rangle-f^{\ast}(x)\ge \theta_0 >0.
\label{ineq32}
\end{equation}

We assume that there exist a smooth function $\phi$ and $(\hat{x},\hat{t})$ such that
$(\hat{x},\hat{t})$ is a strict local maximum of $\overline{u}-\phi$, $\overline{u}>\Psi_{-}$ at $(\hat{x},\hat{t})$ and, for some $\theta_0>0$, \eqref{ineq31} is satisfied at $(\hat{x}, \hat{t})$ provided $D\phi(\hat{x},\hat{t})\neq0$, and \eqref{ineq32} is satisfied at $(\hat{x}, \hat{t})$ provided $D\phi(\hat{x},\hat{t})=0$. If $D\phi(\hat{x},\hat{t})\neq 0$, we take a $\delta$ neighborhood of $(\hat{x},\hat{t})$ where $\overline{u}-\phi$ attains its unique maximum at $(\hat{x},\hat{t})$, $\lvert D\phi\rvert>\theta_1$ for some $\theta_1>0$, and \eqref{ineq31} holds, retaking smaller $\theta_0>0$ if necessary. If $D\phi(\hat{x},\hat{t})=0$, we take a $\delta$ neighborhood of $(\hat{x},\hat{t})$ where $\overline{u}-\phi$ attains its unique maximum at $(\hat{x},\hat{t})$ and \eqref{ineq32} holds, retaking smaller $\theta_0>0$ if necessary. We take $\delta>0$ small enough to satisfy $\delta\le\frac{b}{3\max\{L,M\}}$, where $b:=\overline{u}(\hat{x},\hat{t})-\Psi_{-}(\hat{x})$.
From the definition of $\overline{u}$, there are some sequences $\{\epsilon_n\}$, $\{x_{\epsilon_n}^0\}$, and $\{t_{\epsilon_n}^0\}$ satisfying
\[\epsilon_n \searrow 0,~~(x_{\epsilon_n}^0,t_{\epsilon_n}^0)\to (\hat{x},\hat{t}),~~u^{\epsilon_n}(x_{\epsilon_n}^0,t_{\epsilon_n}^0)\to \overline{u}(\hat{x},\hat{t}).\]
We may substitute $\epsilon$ for $\epsilon_n$ hereafter. We take $\epsilon$ small enough to satisfy $(x_{\epsilon}^0, t_{\epsilon}^0)\in N_{\delta/2}((\hat{x},\hat{t}))$ and $u^{\epsilon}(x_{\epsilon}^0,t_{\epsilon}^0)-\Psi_{-}(x_{\epsilon}^0)\ge b/2$.

We construct the sequence $\{X_k\}$ and the functions $\Psi_k :S^1\times S^1\to\mathbb{R}$ inductively as follows. We first let
\[X_0:=(x_{\epsilon}^0,t_{\epsilon}^0),\]
and
\[\Psi_0(v,w):=\min_{b=\pm 1} u^{\epsilon}(x_{\epsilon}^0+\sqrt{2}\epsilon b v+\nu \epsilon^2 w, t_{\epsilon}^0-\epsilon^2)+\epsilon^2 f(x_{\epsilon}^0).\]
Then let $(v_0,w_0)$ satisfying
\[\Psi_0(v_0,w_0)\ge\sup_{(v,w)\in S^1\times S^1}\Psi_0(v,w)-\epsilon^3,\]
and we determine
\[X_{1}=X_0+(\sqrt{2}\epsilon b_0 v_0+\nu \epsilon^2 w_0, -\epsilon^2),\]
where we will decide $b_0$ later.
For any $k\in \mathbb{N}$, we similarly define
\[\Psi_k(v,w):=\min_{b=\pm 1} u^{\epsilon}(x_k+\sqrt{2}\epsilon b v+\nu \epsilon^2 w, t_k-\epsilon^2)+\epsilon^2 f(x_k),\]
and
\begin{equation}
X_{k+1}:=X_k+(\sqrt{2}\epsilon b_k v_k+\nu \epsilon^2 w_k, -\epsilon^2),
\label{seq11}
\end{equation}
where $(v_k,w_k)\in S^1 \times S^1$ satisfies
\[\Psi_k(v_k,w_k)\ge\sup_{(v,w)\in S^1\times S^1}\Psi_k(v,w)-\epsilon^3.\]

Define $k_{\epsilon}$ as in the proof of Proposition \ref{thm:ob_conv}. We prove by induction that $u^{\epsilon}(X_k)>\Psi_{-}(x_k)$ and
\begin{equation} 
u^{\epsilon}(X_0)-u^{\epsilon}(X_k) \le \epsilon^2\sum_{m=0}^{k-1} [f(x_m)]+k\epsilon^3
\label{ineq03233}
\end{equation}
for all $k<k_{\epsilon}$. These inequalities hold for $k=0$. We assume that these hold for some $k<k_{\epsilon}$.
Then the Dynamic Programming Principle \eqref{eq:DPP_ob} and $u^{\epsilon}(X_k)>\Psi_{-}(x_k)$ imply
\begin{equation}
u^{\epsilon}(X_k)=\min\{\Psi_{+}(x_k), \sup_{v,w\in S^1}\min_{b=\pm1}u^{\epsilon}(x_k+\sqrt{2}\epsilon b v+\nu w\epsilon^2, t_k-\epsilon^2)+\epsilon^2 f(x_k)\}.
\nonumber
\end{equation}
Thus we have
\begin{align}
u^{\epsilon}(X_k)&\le \sup_{v,w\in S^1} \Psi_{k}(v,w) \nonumber \\
&\le \Psi_{k}(v_k,w_k)+\epsilon^3 \nonumber \\
&=u^{\epsilon}(x_k+\sqrt{2}\epsilon b_k v_k+\nu w_k\epsilon^2, t_k-\epsilon^2)+\epsilon^2 f(x_k)+\epsilon^3 \nonumber
\end{align}
and hence
\begin{equation}
u^{\epsilon}(X_k)-u^{\epsilon}(X_{k+1}) \le \epsilon^2 f(x_k)+\epsilon^3,
\label{ineq51}
\end{equation}
which means \eqref{ineq03233} holds for $k+1$.
From the Lipschitz continuity of $\Psi_{-}$ and \eqref{ineq0323}, we have
\begin{equation}
-\Psi_{-}(x_{k+1})\ge -\Psi_{-}(x_0)-L(\delta-(k+1)\epsilon^2)
\label{ineq03232}
\end{equation}
provided $k+1<k_{\epsilon}$. Combining \eqref{ineq51} and \eqref{ineq03232}, we obtain
\begin{align*}
&u^{\epsilon}(X_{k+1})-\Psi_{-}(x_{k+1}) \\
&\ge u^{\epsilon}(X_0)-\Psi_{-}(x_0)-M(k+1)\epsilon^2-L(\delta-(k+1)\epsilon^2)-(k+1)\epsilon^3 \\
&\ge u^{\epsilon}(X_0)-\Psi_{-}(x_0)-\max\{M,L\}\delta-(k+1)\epsilon^3.
\end{align*}
We notice that $(k+1)\epsilon^3\le L\delta\epsilon$ from \eqref{eq:Lip1208}. Therefore $u^{\epsilon}(X_{k+1})>\Psi_{-}(x_{k+1})$ holds for sufficiently small $\epsilon$ and we conclude the induction.

Next we take the continuous path $X(s)$ and use the Taylor's theorem in the same way as in Proposition \ref{thm:ob_conv}. Then we have
\begin{align*}
&\phi(X_{k+1})-\phi(X_k) \\
&=\sqrt{2}\epsilon b_k \langle D\phi(X_k),v_k\rangle+\epsilon^2\{-\partial_{t}\phi(X_k)+\nu \langle D\phi(X_k),w_k\rangle+\langle D^2\phi(X_k) v_k, v_k\rangle\}+\Phi_k(\epsilon),
\end{align*}
where $\Phi_k(\epsilon)=o(\epsilon^2).$
By taking $b_k$ in (\ref{seq11}) properly, we get
\begin{equation}
\begin{split}
&\phi(X_{k+1})-\phi(X_k) \\
&\le-\sqrt{2}\epsilon \lvert\langle D\phi(X_k),v_k \rangle\rvert+\epsilon^2\{-\partial_{t}\phi(X_k)+\nu \lvert D\phi(X_k)\rvert+\langle D^2\phi(X_k) v_k, v_k\rangle\}+\Phi_k(\epsilon).
\end{split}
\label{Taylor2}
\end{equation}
We first consider the case $D\phi(\hat{x},\hat{t})\neq 0$. If $k<k_{\epsilon}$, we just consider $\phi$ in $N_{\delta}((\hat{x},\hat{t}))$. We now use the assumption \eqref{ineq31} and Lemma \ref{lem:kohn} replacing $\phi$ by $-\phi$ to get
\begin{align}
&\phi(X_{k+1})-\phi(X_k) \nonumber \\
&\le\epsilon^2\left\{-\partial_{t}\phi(X_k)+\nu \lvert D\phi(X_k)\rvert+\left\langle D^2\phi(X_k) \frac{D^{\perp}\phi(X_k)}{\lvert D\phi(X_k)\rvert}, \frac{D^{\perp}\phi(X_k)}{\lvert D\phi(X_k)\rvert}\right\rangle\right\} \nonumber \\
&+\frac{C_2 \epsilon^3}{\lvert D\phi(X_k)\rvert}+\Phi_k(\epsilon) \nonumber \\ &\le\epsilon^2\left(-\frac{\theta_0}{2}-f^{\ast}(x_k)\right),
\label{ineq52}
\end{align}
for sufficiently small $\epsilon$. This inequality is also obtained in the case $D\phi(\hat{x},\hat{t})=0$, using (\ref{Taylor2}) and the assumption \eqref{ineq32}.
From (\ref{ineq51}) and (\ref{ineq52}), we have
\begin{equation}
u^{\epsilon}(X_0)-u^{\epsilon}(X_k)-\phi(X_0)+\phi(X_k)\le-\frac{k}{4}\epsilon^2\theta_0,~~~~~k \le k_{\epsilon}.
\label{ineq53}
\end{equation}

By the same argument as in Proposition \ref{thm:ob_conv}, we obtain
\[\overline{u}(\hat{x},\hat{t})-\phi(\hat{x},\hat{t})\le \overline{u}(x',t')-\phi(x',t'),\]
where $(x',t')\neq(\hat{x},\hat{t})$.
This is a contradiction since $\overline{u}-\phi$ attains its unique maximum at $(\hat{x},\hat{t})$.
\end{proof}
The proof of Proposition \ref{prop:gconv_gene} is completed by checking that $\overline{u}$ and $\underline{u}$ satisfy the initial condition. To prove the last proposition, we need additional property of the solution to \eqref{seq:2} and strategies of the game.
\begin{lem}
\label{lem:seqconti}
Let $\delta\in (0,\nu^{-1}]$. For sufficiently small $\epsilon>0$, we have
\[\frac{r_2-\delta}{\delta^{-1}-\frac{\nu}{2}}\le t_{\epsilon}(r_1,r_2),\] 
for any $r_1$ and $r_2$ satisfying $0\le r_1\le\delta\le r_2\le \nu^{-1}$.
\end{lem}
\begin{proof}
Let $\delta\le r\le \nu^{-1}$. We take $\epsilon>0$ small enough to 
satisfy $T_{\epsilon^2}(r)\ge r$. Concretely we assume $\nu\epsilon^2\le \delta$. Then the inequality \eqref{ineq:shortcut} implies
\[T_{\epsilon^2}(r)-r\le \left(\delta^{-1}-\nu\right)\epsilon^2+\frac{\nu^2\epsilon^4}{2\delta}.\]
Hence we obtain
\[t_{\epsilon}(r_1,r_2)\ge t_{\epsilon}(\delta,r_2)\ge \frac{r_2-\delta}{\left(\delta^{-1}-\nu\right)\epsilon^2+\frac{\nu^2\epsilon^4}{2\delta}}\epsilon^2 \ge \frac{r_2-\delta}{\delta^{-1}-\nu+\frac{\nu}{2}}=\frac{r_2-\delta}{\delta^{-1}-\frac{\nu}{2}}.\] 
\end{proof}
\begin{dfn}[Reversed concentric strategy]
Let $\nu>0$, $\epsilon>0$ and $z\in\mathbb{R}^2$. Let $x\in\mathbb{R}^2$ be the current position of the game. Let $(v,w)\in S^1\times S^1$ be a choice by \sente in the same round. A choice $b\in \{\pm1\}$ by \gote is called a {\it \zcencr{$z$}} if
\[\langle b v, x+\nu\epsilon^2 w-z\rangle \le0.\]
\end{dfn}
If Carol takes \zcencr{$z$} through the game, we get $\lvert x_n-z\rvert\le P_n$, where $P_n$ satisfies
\begin{equation}
P_{n+1}=\sqrt{(P_n+\nu \epsilon^2)^2+2\epsilon^2}
\label{seq:3}
\end{equation}
with $P_0=\lvert x_0-z\rvert$. We define $t_{\epsilon}(a,b)$ in the same way as \eqref{eqdef:exittime}, replacing the operator $T_{h}$ as follows:
\[T_{h}(R):=\sqrt{(R+\nu h)^2+2h}.\]
\begin{lem}
Let $\delta>0$. For sufficiently small $\epsilon>0$, we have 
\[\frac{r_2-\delta}{\delta^{-1}+\nu+1}\le t_{\epsilon}(r_1,r_2)\]
for any $r_1$ and $r_2$ satisfying $0\le r_1\le\delta\le r_2$.
\label{lem:4}
\end{lem}
\begin{proof}
The proof is similar to that of Lemma \ref{lem:seqconti}, so is omitted.
\end{proof}
\begin{prop}
\label{lem:inigame}
Let $u_0$ be a continuous function. Then $\overline{u}(x,0)=\underline{u}(x,0)=u_0(x)$ for all $x\in\mathbb{R}^2$.
\end{prop}
\begin{proof}
Let $x\in\mathbb{R}^2$. For the initial position $y\in\mathbb{R}^2$, the terminal time $s>0$ and the step size $\epsilon>0$, we define $V^{-}(y,s,\epsilon)$ as the minimum total cost when \sente takes a \zcenp{$x$} through the game. Similarly we define $V^{+}(y,s,\epsilon)$ as the supremum total cost when \gote takes \zcencr{$x$} through the game. 
It is clear by the property of the value functions that
\[V^{-}(y,s,\epsilon)\le u^{\epsilon}(y,s)\le V^{+}(y,s,\epsilon).\]
It is sufficient to show
\begin{equation}
\varliminf_{\substack{(y,s) \to (x,0) \\ \epsilon \searrow 0}} V^{-}(y,s,\epsilon)\ge u_0(x)
\label{ineq:1117c}
\end{equation}
and
\begin{equation}
\varlimsup_{\substack{(y,s) \to (x,0) \\ \epsilon \searrow 0}} V^{+}(y,s,\epsilon)\le u_0(x).
\label{ineq:1116P}
\end{equation}

We first analyze $V^{-}$. We denote by $V^{-}_{quit}(y,s,\epsilon)$ (resp. $V^{-}_{end}(y,s,\epsilon)$) the minimum total cost when \sente takes \zcenp{$x$} through the game and \gote quits (resp. does not quit) the game on the way. Then we write
\[V^{-}(y,s,\epsilon)=\min\{V^{-}_{end}(y,s,\epsilon), V^{-}_{quit}(y,s,\epsilon)\}.\]
Furthermore we analyze $V^{-}_{end}$. We denote by $V^{-}_{run}(y,s,\epsilon)$ (resp. $V^{-}_{ter}(y,s,\epsilon)$) the minimum running cost (resp. terminal cost) in the same situation as $V^{-}_{end}(y,s,\epsilon)$. Obviously we have
\[V^{-}_{end}(y,s,\epsilon)\ge V^{-}_{run}(y,s,\epsilon)+V^{-}_{ter}(y,s,\epsilon).\]

Since Paul takes a \zcenp{$x$}, he can stay in $B(x,2\nu^{-1})$ by Lemma \ref{lem:series2}. So the running cost is at most $M:=\sup_{z\in B(x,2\nu^{-1})}\lvert f(z)\rvert$, and at least $-M$ per round. Hence we have
\begin{align}
\lvert V^{-}_{run}(y,s,\epsilon)\rvert&\le \epsilon^2 NM=\epsilon^2 \left\lceil s\epsilon^{-2}\right\rceil M \le \epsilon^2\left(s\epsilon^{-2}+1\right)M=(s+\epsilon^2)M.\label{eq:1117}
\end{align}
We denote by $Ter(y,s,\epsilon)$ a terminal point $x_N$ in the situation of $V^{-}_{end}(y,s,\epsilon)$. Since $u_0$ is continuous, what we have to prove about the terminal cost is
\begin{equation}
\lim_{\substack{(y,s) \to (x,0) \\ \epsilon \searrow 0}} Ter(y,s,\epsilon)=x
\nonumber
\end{equation}
for any choices of \gote. Let $\{(y_n,s_n,\epsilon_n)\}\subset \mathbb{R}^2\times(0,\infty)\times(0,\infty)$ be any sequence satisfying
\[\epsilon_n \searrow 0,~~y_n\rightarrow x,~~s_n\rightarrow 0.\]
Let $\delta>0$. Then we shall show that $\lvert Ter(y_n,s_n,\epsilon_n)-x\rvert<\delta$ for sufficiently large $n$.
Indeed, from Lemma \ref{lem:seqconti}, there exists $\tilde{\epsilon}>0$ such that
\begin{equation}
\frac{\delta/3}{3\delta^{-1}-\nu+1}\le t_{\epsilon}(r,2\delta/3)
\label{eq:1117a}
\end{equation}
for all $r\in [0,\delta/3)$ and all $\epsilon\in (0,\tilde{\epsilon})$.
We take $n$ large enough so that
\[\lvert y_n-x\rvert<\delta/3,~\lvert s_n\rvert<\frac{\delta/3}{3\delta^{-1}-\nu+1},~\epsilon_n<\tilde{\epsilon}.\]
Hence, together with \eqref{eq:1117}, we obtain
\begin{equation}
\label{eq:rute}
\lim_{\substack{(y,s) \to (x,0) \\ \epsilon \searrow 0}}V^{-}_{run}(y,s,\epsilon)+V^{-}_{ter}(y,s,\epsilon)=u_0(x).
\end{equation}

If \gote quits the game on the way, the game positions $\{x_n\}$ are in $B(x, \lvert Ter(y,s,\epsilon)-x\rvert)$. Thus we have
\begin{equation}
\nonumber
V^{-}_{quit}(y,s,\epsilon)\ge \inf\{\Psi_{+}(z) \mid z\in B(x, \lvert Ter(y,s,\epsilon)-x\rvert)\}
\end{equation}
and hence
\begin{align*}
\varliminf_{\substack{(y,s) \to (x,0) \\ \epsilon \searrow 0}} V^{-}_{quit}(y,s,\epsilon)&\ge \lim_{\substack{(y,s) \to (x,0) \\ \epsilon \searrow 0}} \inf\{\Psi_{+}(z) \mid z\in B(x, \lvert Ter(y,s,\epsilon)-x\rvert)\} \\ 
&=\Psi_{+}(x)\ge u_0(x).
\end{align*}
Together with \eqref{eq:rute}, we obtain \eqref{ineq:1117c}.

We next estimate $V^{+}$. We denote by $V^{+}_{quit}(y,s,\epsilon)$ (resp. $V^{+}_{end}(y,s,\epsilon)$) the supremum total cost when \gote takes a \zcencr{$x$} through the game and \sente quits (resp. does not quit) the game on the way. Then we write
\[V^{+}(y,s,\epsilon)=\max\{V^{+}_{end}(y,s,\epsilon), V^{+}_{quit}(y,s,\epsilon)\}.\]
We further denote by $V^{+}_{run}(y,s,\epsilon)$ (resp. $V^{+}_{ter}(y,s,\epsilon)$) the supremum running cost (resp. terminal cost) in the same situation as $V^{+}_{end}(y,s,\epsilon)$. Obviously we have
\[V^{+}_{end}(y,s,\epsilon)\le V^{+}_{run}(y,s,\epsilon)+V^{+}_{ter}(y,s,\epsilon).\]

From Lemma \ref{lem:4}, Paul is forced to stay in a compact set for sufficiently small $s$. Thus as in \eqref{eq:1117}, we have
\begin{equation}
\lvert V^{+}_{run}(y,s,\epsilon)\rvert\le \epsilon^2 NM\le(s+\epsilon^2)M. \nonumber
\end{equation}
The values $V^{+}_{ter}$ and $V^{+}_{quit}$ are also estimated in the same way as $V^{-}_{ter}$ and $V^{-}_{quit}$ respectively. The only difference is
\begin{equation}
\frac{\delta/3}{3\delta^{-1}+\nu+1}\le t_{\epsilon}(r,2\delta/3)\nonumber
\end{equation}
instead of \eqref{eq:1117a}. Thus \eqref{ineq:1116P} is also obtained.
\end{proof}

\section{Set theory}
\label{app:set}
The following are supplementary propositions related to general topology and convex sets.

\begin{lem}
\label{lem:topolo}
Let $A\subset\mathbb{R}^d$ be an open set. Let $K\subset A$ be a compact set. Then, for sufficiently small $\delta>0$,
\[B_{\delta}(K)\subset A.\]
\end{lem}
\begin{proof}
We can assume without loss of generality that $A$ is bounded. We define $f(x):=\sup\{\delta>0 \mid B_{\delta}(x)\subset A\}$ for $x\in A$ and check that it is a lower semicontinuous function. Let $\epsilon>0$. For $x,y\in A$ satisfying $\lvert x-y\rvert<\epsilon$, it is clear that $B_{f(x)-\epsilon}(y)\in A$ and then $f(y)\ge f(x)-\epsilon$.

Since $f$ is lower semicontinuous, it has a minimizer $\hat{x}$ in $K$ by the extreme value theorem. Letting $\delta=f(\hat{x})$, we obtain the conclusion.
\end{proof}

\begin{lem}
\label{lem:convCara}
Let $A\subset \mathbb{R}^{2}$ be a connected open set. Then $Co(A)=\{x\in l_{a,b} \mid a,b\in A\}$.
\end{lem}
\begin{proof}
It is clear that $Co(A)\supset\{x\in l_{a,b} \mid a,b\in A\}$.
By Carath\'{e}odory's theorem, we have
\[Co(A)=\left\{\sum_{i=1}^{3}\lambda_{i}x_i; \ x_i\in A, \lambda_{i}\in[0,1], \sum_{i=1}^3 \lambda_i=1\right\}.\]
Fix any element $x\in Co(A)$. Then we can write $x=\sum_{i=1}^{3}\lambda_{i}x_i$ for some $x_i\in A$ and $\lambda_{i}\in[0,1]$. It suffices to consider the case $x_1, x_2, x_3$ are different and $\lambda_{i}\in(0,1)$. We can assume $x=(0,0)$, $x_1=(0,1)$, $x_2\in\{(p,q)\in\mathbb{R}^2 \mid p<0\}$, and $x_3\in\{(p,q)\in\mathbb{R}^2 \mid p>0\}$. Since $A$ is a a connected open set, it is also path-connected. Thus there is a continuous path $\Gamma\subset A$ that connects $x_2$ and $x_3$. By the intermediate value theorem, the path $\Gamma$ crosses $y$-axis. If $\Gamma$ crosses $x_4\in\{(0,q)\in\mathbb{R}^2 \mid q<0\}$, then $x_1,x_4\in A$ and $x\in l_{x_1,x_4}$. Otherwise let $l$ be the line satisfying $x\in l$ and $l_{x_2,x_3} \parallel l$. The path $\Gamma$ crosses points $x_4\in l\cap\{(p,q)\in\mathbb{R}^2 \mid p<0\}$ and $x_5\in l\cap\{(p,q)\in\mathbb{R}^2 \mid p>0\}$. Thus we have $x_4,x_5\in A$ and $x\in l_{x_4,x_5}$. Therefore we conclude that $x\in l_{a,b}$ for some $a,b\in A$.
\end{proof}

\begin{prop}
\label{prop:conv}
If $A\subset \mathbb{R}^d$ is open, then $Co(A)$ is open.
\end{prop}
\begin{proof}
Fix $x\in Co(A)$. By Carath\'{e}odory's theorem, we have $x=\sum_{i=1}^{d+1}\lambda_{i}x_i$ for some $x_i\in A$ and $\lambda_{i}\in[0,1]$ satisfying $\sum_{i=1}^{d+1} \lambda_i=1$.
Since $A$ is open, we see that $\cup_{i=1}^{d+1} B_{r_0}(x_i)\subset A$ for some $r_0>0$. Therefore, for any unit vector $v\in S^1$, we have $x+rv\in Co(A)$ for $0\le r< r_0$ since $x+rv=\sum_{i=1}^{d+1}\lambda_{i}(x_i +rv)$.
\end{proof}
\begin{prop}
\label{prop:0707}
Let $O_{-}\subset \mathbb{R}^d$. If $O_{-}$ satisfies \eqref{cond:inball}, then $\overline{O_{-}}$ is strictly convex.
\end{prop}
\begin{proof}
Assume that $\overline{O_{-}}$ is not strictly convex, i.e., there exist $x,y\in\overline{O_{-}}$ such that $\lambda x+(1-\lambda)y\in (O_{-}^{int})^c$ for some $\lambda\in (0,1)$.

{\boldmath $1)~\lambda x+(1-\lambda)y=:z\in \partial O_{-}$ {\bf for some} $\lambda\in (0,1)$.} By \eqref{cond:inball} we can take an open ball $B$ that satisfies $z\in \partial B$ and $\overline{O_{-}}\subset \overline{B}$. Since $x\in \overline{B}$ and $z\in \partial B$, we have $y\in (\overline{B})^c$, which is a contradiction.

{\boldmath $2)~\lambda x+(1-\lambda)y\in (\overline{O_{-}})^c$ {\bf for any} $\lambda\in (0,1)$.} Let $z=\frac{x+y}{2}$. Let $\delta>0$ satisfy $B_{\delta}(z)\subset \left(\overline{O_{-}}\right)^c$. Since $x\in \partial O_{-}$, there exists $w\in O_{-}$ such that $w\in B_{2\delta}(x)$. Since $\frac{w+y}{2}\in\left(\overline{O_{-}}\right)^c$, we see that $\lambda w+(1-\lambda)y\in \partial O_{-}$ for some $\lambda\in(0,1)$ and hence deduce a contradiction.
\end{proof}

\section{Graph theory}
\label{sec:graphtheory}
We present the notion of graph and some related notions in the graph theory.
\begin{dfn}
For a non empty set $V$ and a set $E$ of unordered pairs in $V$, the pair of the sets $(V,E)$ is called a {\it graph}. A graph $H=(V^{\prime},E^{\prime})$ is called a {\it subgraph} of $G$ if $V^{\prime}\subset V$ and $E^{\prime}\subset E$. A subgraph $H=(V^{\prime},E^{\prime})\subset G$ is called a {\it path} of $G$ if $V^{\prime}$ is a finite set $\{x_0,x_1,\cdots,x_n\}$(duplication is permitted.) and $E^{\prime}=\{\langle x_i,x_{i+1} \rangle \mid i=0,1,\cdots,n-1\}$, where we denote unordered pairs by $\langle , \rangle$. A graph $G=(V,E)$ is \it{connected} if for any $v_1,v_2\in V$, there is a path of $G$ whose endpoints are $v_1$ and $v_2$.
\end{dfn}
To precisely indicate the path of graph introduced in the proof of Theorem \ref{thm:conv_general} , we present the following proposition, though the assertion seems to be obvious.
\begin{prop}
\label{prop:graph}
Let $G=(V,E)$ be a connected graph. Let $\langle a,b\rangle,\langle c,d\rangle\in E$. Then there is a path $P=(V^{\prime},E^{\prime})$ such that $a,b,c,d\in V^{\prime}$ and $\langle a,b\rangle, \langle c,d\rangle\in E^{\prime}$.
\end{prop}
\begin{proof}
If $a=c$, $(\{b,a,d\},\{\langle b,a\rangle, \langle a,d\rangle\})$ is a required path. Hereafter we consider the case neither $a=c$, $b=c$, $a=d$ nor $b=d$. Let $P_0=(V_0,E_0)$ be a path with endpoints $a$ and $c$. 

If $b,d\notin V_0$, define $V_2:=V_0\cup \{b,d\}$ and $E_2:=E_0\cup \{\langle b,a\rangle, \langle c,d\rangle\}$. 

If $b\notin V_0$ and $d\in V_0$, define $V_1:=V_0\cup \{b\}$ and $E_1:=E_0\cup \{\langle b,a\rangle\}$. Writing
\begin{equation}
\label{not:graph}
\begin{split}
V_1&=\{x_0,x_1,\cdots,x_n\}, \\
E_1&=\{\langle x_i,x_{i+1} \rangle \mid i=0,1,\cdots,n-1\},
\end{split}
\end{equation}
we see that $x_0=b$, $x_1=a$, $x_n=c$ and $x_j=d$ for some $j\in\{2,3,\cdots,n\}$. Then we further define $V_2:=\{x_0,x_1,\cdots,x_j,x_n\}$ and $E_2:=\{\langle x_0,x_1 \rangle, \langle x_1,x_2 \rangle, \cdots \langle x_{j-1},x_j \rangle, \langle x_j,x_n \rangle\}$. 

We finally consider the case $b,d\in V_0$. When we follow the path $P_0$ from $a$ to $c$, there are two cases: whether we find $b$ earlier than $d$ or not. In the former case, writing $V_0$ and $E_0$ as \eqref{not:graph}, we see $x_j=b$, $x_m=d$ for some $0\le j<m\le n$. We then define $V_2:=\{x_0, x_j, x_{j+1}, \cdots, x_m, x_n\}$ and $E_2:=\{\langle x_0,x_j \rangle, \langle x_j,x_{j+1} \rangle, \cdots \langle x_{m-1},x_m \rangle, \langle x_m,x_n \rangle\}$. In the latter case, we see $x_j=d$, $x_m=b$ for some $0\le j<m\le n$. We define $V_2:=\{x_l, x_0, x_1, \cdots, x_j, x_n\}$ and $E_2:=\{\langle x_l,x_0 \rangle, \langle x_0,x_1 \rangle, \cdots \langle x_{j-1},x_j \rangle, \langle x_j,x_n \rangle\}$.

In any case, $P_2:=(V_2, E_2)$ is a required path.
\end{proof}

\section{Curve theory}
\label{sec:curve}
To complement the proof of Thorem \ref{thm:conv_general}, we will mathematically describe the construction of a Jordan curve $\hat{C}$ that is included in a given closed curve $C$ and includes a given point $x\in C$. We begin with a general property of connected sets. In what follows we especially notice that two points in an open and connected subset of $\mathbb{R}^d$ can be connected by a polygonal line. 
\begin{dfn}[polygonal line connected]
A path is called a {\it polygonal line} if it consists of finite line segments. Let $A\subset\mathbb{R}^d$. The set $A$ is called {\it polygonal line connected} if for any two points $x,y\in A$, there exists a polygonal line in $A$ that connects $x$ and $y$.
\end{dfn}
\begin{prop}
Let $A\subset\mathbb{R}^d$ be an open set. Then the following statements are equivalent.
\begin{enumerate}
\item $A$ is connected.
\item $A$ is path-connected.
\item $A$ is polygonal line connected.
\end{enumerate}
\end{prop}
\begin{proof}
Without loss of generality, we can assume $A\neq \emptyset$ since otherwise all the statements are obviously true.
 
3$\Rightarrow$2.

This is clear because a polygonal line is a path.

2$\Rightarrow$1.

Fix $x\in A$. Since a path is a connected set, all elements $y\in A$ are in the connected component including $x$. Therefore $A$ is connected.

1$\Rightarrow$3.

Fix $x\in A$. We define \\ $O:=\{y\in A \mid \ \mbox{there exists a polygonal line in $A$ that connects $x$ and $y$}\}$. We first show that $O$ is an open set. Let $y\in O$. Since $y\in A$, there is an open ball $B_{\delta}(y)$ such that $B_{\delta}(y)\subset A$. For any $z\in B_{\delta}(y)$, we can make a polygonal line in $A$ that connects $x$ and $z$, combining the line segment between $y$ and $z$ with a polygonal line between $x$ and $y$. Therefore we have $B_{\delta}(y)\subset O$, which means $O$ is an open set.

We show that $A\setminus O$ is also an open set. Let $y\in A\setminus O$. As before, there is an open ball $B_{\delta}(y)$ such that $B_{\delta}(y)\subset A$. If a point $z$ in $B_{\delta}(y)$ is in $O$, we can make a polygonal line in $A$ that connects $x$ and $y$. This is a contradiction. Hence we have $B_{\delta}(y)\subset A\setminus O$.

Since $A$ is connected, $O$ must be $\emptyset$ or $A$. Since $x\in O$, we have $O=A$ and conclude that $A$ is polygonal line connected.
\end{proof}

We state the condition on components of the closed curve $C$. In what follows we call a map $f \vert_{A}$ injective at $t\in A$ if $s\in A$ and $f(s)=f(t)$ imply $s=t$. Also we call a map $f \vert_{A}$ injective in $B(\subset A)$ if $s\in A$, $t\in B$ and $f(s)=f(t)$ imply $s=t$. We set a class $\mathscr{C}$ of curves in $\mathbb{R}^2$. We make the assumption on $\mathscr{C}$:
\begin{itemize}
\item[(A1)] There exists a map $\mathscr{C}\ni C \mapsto \gamma_C\in C([0,1];\mathbb{R}^2)$ such that
\begin{enumerate}
\item $\gamma_{C}([0,1])=C$ and the set $\{t\in(0,1) \mid \gamma_{C} \ \mbox{is not injective at $t$}\}$ is at most finite.\label{cond:0515}
\item For any $C,D\in\mathscr{C}$, $\{t\in(0,1) \mid \gamma_{C}(t)\notin D\}$ is at most a finite union of open intervals.
\end{enumerate}
\end{itemize}
We now state the assumption on the closed curve $C$:
\begin{itemize}
\item[(A2)] For some $C_1,C_2,\cdots,C_N\in\mathscr{C}$, $C=\cup_{i=1}^N C_i$, $\gamma_{C_i}(1)=\gamma_{C_{i+1}}(0)$ for $i\in\{1,2,\cdots,N-1\}$ and $\gamma_{C_N}(1)=\gamma_{C_1}(0)$.
\end{itemize}
\begin{rem}
The set of line segments and arcs in $\mathbb{R}^2$ satisfies $\mathrm{(A1)}$. Hence the closed curve in the proof of Theorem \ref{thm:conv_general} satisfies $\mathrm{(A2)}$.
\end{rem}

Set $\gamma:[0,N]\to\mathbb{R}^2$ as
\[\gamma(t):=\gamma_{C_i}(t-[t]),\]
if $i-1\le t\le i$. Here we denote by $[t]$ the maximal integer that is no more than $t$. For a point $x\in C$, there is no loss of generality to assume $\gamma(0)=x$. The reason is the following: 
\begin{prop}
Let $\mathscr{C}$ satisfy $\mathrm{(A1)}$. Let $C^{\prime}\in\mathscr{C}$. Then 
$\mathscr{C}^{\prime}:=\mathscr{C}\cup\{C^{\prime}_1, C^{\prime}_2\}$ also satisfies $\mathrm{(A1)}$, where we define
\[C^{\prime}_1:=\gamma_{C^{\prime}}([0,a]) \ \mbox{and} \ C^{\prime}_2:=\gamma_{C^{\prime}}([a,1])\]
for $a\in(0,1)$.
\end{prop}
\begin{proof}
Let $\gamma_{C^{\prime}_1}(t)=\gamma_{C^{\prime}}(t/a)$ and $\gamma_{C^{\prime}_2}(t)=\gamma_{C^{\prime}}\left(a+(1-a)t\right)$. The proof is done by checking the assumption $\mathrm{(A1)}$ directly. 
\end{proof}
We assume that $\gamma \vert_{[0,N)}$ is injective in a neighborhood of $0$. i.e., 
\begin{itemize}
\item[(A3)] There exists $\delta>0$ such that $\gamma \vert_{[0,N)}$ is injective in $[0,\delta)$.
\end{itemize}
\begin{rem}
The closed curve $C\cup \hat{\Gamma}$ in the proof of Theorem \ref{thm:conv_general} satisfies $\mathrm{(A3)}$ because $B_{3\delta}(\hat{\Gamma})\subset L$ and $x\notin L$.
\end{rem}
We inductively define
\[t_1:=0,\]
\[s_i:=\sup\{ \tau \mid \gamma \vert_{(t_i,N]} \ \mbox{is injective in} \ (t_i,\tau)\},\]
\[t_{i+1}:=\sup\{ \tau \mid \gamma(s_i)=\gamma(\tau)\}\]
for $i=1,2,\cdots$.
\begin{prop}
For some $m\in\mathbb{N}$, $s_m=N$ or $t_m=N$.
\end{prop}
\begin{proof}
We first prove $t_j<s_j$ for all $j$. The assumption that $\gamma \vert_{[0,N)}$ is injective in a neighborhood of $0$ implies $t_1<s_1$. If $s_1<N$, then fix $j\in\{2,3,\cdots\}$. Let $i=[t_j]+1$. Set 
\[A_i:=\bigcap_{i+1\le k\le N}\{t\in (i-1,i) \mid \gamma(t)\notin C_k\}\]
for $1\le i\le N-1$ and $A_N:=(N-1,N)$. We also set 
\[B_i:=\{t\in (i-1,i) \mid \gamma \vert_{(i-1,i)} \ \mbox{is not injective at} \ t\}.\] 
From the assumption $\mathrm{(A1)}$, $A_i$ is a finite union of open intervals and $B_i$ is at most finite. By the definition of $t_j$ we see that if $i-1<t_j<i$, then $t_j\in A_i$. Hence, if $i-1<t_j<i$, we have $(t_j,t_j+\delta)\subset A_i\cap B_i^c$ for some $\delta>0$. Also this assertion holds for $t_j=i-1$. To show it, we prove by contradiction that $\inf A_i=i-1$. We assume that $\inf A_i>i-1$. Then we have $\inf \{t\in (i-1,i) \mid \gamma(t)\notin C_k\}>i-1$ for $i+1\le k\le N$. From the continuity of $\gamma$, we obtain $\gamma(t_j)\in C_k$, which is a contradiction with the definition of $t_j$. 

Now it turns out that there exists $l\ge 0$ such that $i-1\le t_j<i$ implies $i\le s_{j+l}$ or $i\le t_{j+l+1}$. Indeed, if $s_j<i$, then $s_j\in\partial A_i\cup B_i$. If $s_j<i$ and $s_j\in\partial A_i$, then $i\le t_{j+1}$. If $s_j<i$ and $s_j\in B_i$, then $t_{j+1}\in B_i$. Thus the proof is complete.
\end{proof}
If $s_m=N$ or $t_{m+1}=N$, then let $T=\sum_{i=1}^m (s_i-t_i)$. Define $\hat{\gamma}:[0,T]\to\mathbb{R}^2$ as follows:
\[\hat{\gamma}(t):=\gamma(t)\]
if $t\le s_1$ and
\[\hat{\gamma}(t):=\gamma\left(t+t_{k+1}-\sum_{i=1}^{k}(s_i-t_i)\right) \]
if $\sum_{i=1}^{k}(s_i-t_i)\le t \le\sum_{i=1}^{k+1}(s_i-t_i)$.
\begin{prop}
$\hat{C}:=\hat{\gamma}([0,T])$ is a Jordan closed curve.
\end{prop}
\begin{proof}
This assertion is obvious from the definitions of $t_i$ and $s_i$.
\end{proof}

\begin{figure}[htbp]
 \begin{minipage}{0.5\hsize}
  \begin{center}
{\unitlength 0.1in%
\begin{picture}(12.0000,15.7000)(6.0000,-22.0000)%
%
\special{pn 8}%
\special{pa 1200 800}%
\special{pa 1200 1400}%
\special{fp}%
\special{pa 600 1400}%
\special{pa 1200 1400}%
\special{fp}%
\special{pa 600 1400}%
\special{pa 600 2200}%
\special{fp}%
\special{pa 600 2200}%
\special{pa 1800 2200}%
\special{fp}%
\special{pa 1800 2200}%
\special{pa 1800 1400}%
\special{fp}%
\special{pa 1800 1400}%
\special{pa 1200 1400}%
\special{fp}%
%
\special{pn 4}%
\special{sh 1}%
\special{ar 1200 800 8 8 0 6.2831853}%
\special{sh 1}%
\special{ar 1200 800 8 8 0 6.2831853}%
\put(12.0000,-7.6000){\makebox(0,0)[lb]{$x$}}%
\end{picture}}%
  \end{center}
  \caption{An example of $C$ and $x$ that we avoid}
  \label{fig:curve_ex1}
 \end{minipage}
 \begin{minipage}{0.5\hsize}
  \begin{center}
{\unitlength 0.1in%
\begin{picture}(22.0000,9.1400)(10.0000,-22.7400)%
%
\special{pn 8}%
\special{pa 1000 1360}%
\special{pa 1542 1360}%
\special{pa 1270 1902}%
\special{pa 1000 1360}%
\special{pa 1542 1360}%
\special{fp}%
%
\special{pn 8}%
\special{pa 1542 1630}%
\special{pa 1812 1630}%
\special{pa 1676 1902}%
\special{pa 1542 1630}%
\special{pa 1812 1630}%
\special{fp}%
%
\special{pn 8}%
\special{pa 1812 1766}%
\special{pa 1948 1766}%
\special{pa 1880 1902}%
\special{pa 1812 1766}%
\special{pa 1948 1766}%
\special{fp}%
%
\special{pn 8}%
\special{pa 1948 1834}%
\special{pa 2016 1834}%
\special{pa 1982 1902}%
\special{pa 1948 1834}%
\special{pa 2016 1834}%
\special{fp}%
%
\special{pn 8}%
\special{pa 2016 1868}%
\special{pa 2050 1868}%
\special{pa 2032 1902}%
\special{pa 2016 1868}%
\special{pa 2050 1868}%
\special{fp}%
%
\special{pn 8}%
\special{pa 2083 1630}%
\special{pa 2354 1630}%
\special{pa 2218 1902}%
\special{pa 2083 1630}%
\special{pa 2354 1630}%
\special{fp}%
%
\special{pn 8}%
\special{pa 2354 1766}%
\special{pa 2490 1766}%
\special{pa 2422 1902}%
\special{pa 2354 1766}%
\special{pa 2490 1766}%
\special{fp}%
%
\special{pn 8}%
\special{pa 2490 1834}%
\special{pa 2556 1834}%
\special{pa 2523 1902}%
\special{pa 2490 1834}%
\special{pa 2556 1834}%
\special{fp}%
%
\special{pn 8}%
\special{pa 2556 1868}%
\special{pa 2590 1868}%
\special{pa 2574 1902}%
\special{pa 2556 1868}%
\special{pa 2590 1868}%
\special{fp}%
%
\special{pn 8}%
\special{pa 2624 1766}%
\special{pa 2760 1766}%
\special{pa 2692 1902}%
\special{pa 2624 1766}%
\special{pa 2760 1766}%
\special{fp}%
%
\special{pn 8}%
\special{pa 2760 1834}%
\special{pa 2828 1834}%
\special{pa 2794 1902}%
\special{pa 2760 1834}%
\special{pa 2828 1834}%
\special{fp}%
%
\special{pn 8}%
\special{pa 2828 1868}%
\special{pa 2862 1868}%
\special{pa 2844 1902}%
\special{pa 2828 1868}%
\special{pa 2862 1868}%
\special{fp}%
%
\special{pn 8}%
\special{pa 2896 1834}%
\special{pa 2963 1834}%
\special{pa 2930 1902}%
\special{pa 2896 1834}%
\special{pa 2963 1834}%
\special{fp}%
%
\special{pn 8}%
\special{pa 2963 1868}%
\special{pa 2996 1868}%
\special{pa 2980 1902}%
\special{pa 2963 1868}%
\special{pa 2996 1868}%
\special{fp}%
%
\special{pn 8}%
\special{pa 3030 1868}%
\special{pa 3064 1868}%
\special{pa 3048 1902}%
\special{pa 3030 1868}%
\special{pa 3064 1868}%
\special{fp}%
%
\special{pn 8}%
\special{pa 1000 1902}%
\special{pa 2050 1902}%
\special{fp}%
\special{pa 2083 1902}%
\special{pa 2590 1902}%
\special{fp}%
\special{pa 2862 1902}%
\special{pa 2624 1902}%
\special{fp}%
\special{pa 2896 1902}%
\special{pa 2996 1902}%
\special{fp}%
\special{pa 3030 1902}%
\special{pa 3064 1902}%
\special{fp}%
%
\special{pn 13}%
\special{pa 2050 1902}%
\special{pa 2083 1902}%
\special{fp}%
\special{pa 2590 1902}%
\special{pa 2624 1902}%
\special{fp}%
\special{pa 2862 1902}%
\special{pa 2896 1902}%
\special{fp}%
\special{pa 2996 1902}%
\special{pa 3030 1902}%
\special{fp}%
\special{pa 3064 1902}%
\special{pa 3098 1902}%
\special{fp}%
%
\special{pn 20}%
\special{pa 3098 1902}%
\special{pa 3132 1902}%
\special{fp}%
%
\special{pn 8}%
\special{pa 1000 1902}%
\special{pa 1000 2240}%
\special{fp}%
\special{pa 1000 2240}%
\special{pa 3200 2240}%
\special{fp}%
\special{pa 3200 2240}%
\special{pa 3200 1902}%
\special{fp}%
\special{pa 3200 1902}%
\special{pa 3132 1902}%
\special{fp}%
%
\special{pn 4}%
\special{sh 1}%
\special{ar 2016 2240 8 8 0 6.2831853}%
\special{sh 1}%
\special{ar 2016 2240 8 8 0 6.2831853}%
\put(20.1600,-22.7400){\makebox(0,0)[lt]{$x$}}%
\end{picture}}%
  \end{center}
  \caption{An example of $C$ that includes an infinite number of loops}
  \label{fig:self_sim_ex}
 \end{minipage}
\end{figure}

\begin{rem}
By assuming $\rm{(A3)}$, we can avoid a closed curve $C$ and a point $x$ in it such as Figure \ref{fig:curve_ex1}. We also avoid a closed curve $C$ that includes an infinite number of loops such as Figure \ref{fig:self_sim_ex} by the assumption $\rm{(A2)}$.
\end{rem}

\end{appendices}

\bibliographystyle{plain}
\bibliography{bibdata}
\end{document}